\documentclass{amsart}
\usepackage{amsmath,amsfonts,amsthm,enumerate,color,amscd}
 \usepackage[colorlinks=true, pdfstartview=FitV, linkcolor=blue, citecolor=blue, urlcolor=blue]{hyperref}

\usepackage{pdfsync}

\newcommand{\propref}[1]{Proposition~\ref{#1}}

\newcommand{\lemref}[1]{Lemma~\ref{#1}}
\newcommand{\eqnref}[1]{~(\ref{#1})}

\newcommand{\germ}{\mathfrak}
\newcommand{\stack}[2]{\genfrac{}{}{0pt}{}{#1}{#2}}
\newtheorem{thm}{Theorem}[section]
\newtheorem{lem}[thm]{Lemma}

\theoremstyle{definition}

\newtheorem{prop}[thm]{Proposition}

\numberwithin{equation}{section}

\usepackage{amssymb}\title{A Wakimoto type realization of toroidal $\mathfrak{sl}_{n+1}$.}

\author{Samuel Buelk \and Ben L. Cox  \and Elizabeth Jurisich}
\date{}                                           

\begin{document}
\begin{abstract}The authors construct a Wakimoto type realization of toroidal $\mathfrak{sl}_{n+1}$ The representation constructed in this paper utilizes non-commuting differential operators acting on the tensor product of two polynomial rings in many commuting variables. 
\end{abstract}

\maketitle
\section{Introduction}
Toroidal Lie algebras were first introduced in \cite{MR} as a natural generalization of affine algebras.  Given a finite-dimensionl simple Lie algebra $\mathfrak a$ a toroidal algebra is a central extension of $\mathfrak a \otimes \mathbb C [t_1, \dots, t_n, t_1^{-1}, \dots ,t^{-1}_n]$, where the $t_i$ are commuting variables. Toroidal algebras can be thought of as iterated loop algebras in many commuting variables. Such algebras can also be defined using generators and relations as we do here. 

One motivation for the study of toroidal Lie algebras is for potential applications to math and physics. For instance, one of the cocycles used in the construction of the toroidal extended affine Lie algebra is also used in Y. Billig's study of a magnetic hydrodynamics equation with asymmetric stress tensor (see \cite{MR2290921} and \cite{MR2131250}).  In addition Billig and independently Iohara, Saito, and Wakimoto (see \cite{MR1700475}, and \cite{MR2000d:17037}) derive Hirota bilinear equations arising from both homogeneous and principal realizations of the vertex operator representations of 2-toroidal Lie algebras of type $A_l, D_l, E_l$.  They derive the hierarchy of Hirota equations and present their soliton-type solutions.  In  \cite{MR1937604},   Kakei, Ikeda, and Takasaki construct the hierarchy associated to the $(2+1)$-dimensional nonlinear Schr\"{o}dinger (NLS) equation and show how the representation theory of toroidal $\mathfrak{sl}_2$ can be used to derive the Hirota-type equations for $\tau$-functions. On the somewhat more mathematical side, in interesting work of V. Ginzburg, M. Kapranov, and E. Vasserot (see 
\cite{MR1324698}) on Langland's reciprocity for algebraic surfaces,  Hecke operators are constructed for vector bundles on an algebraic surface. The main point 
of their paper is that under certain conditions the corresponding algebra of Hecke operators is the homomorphic image of a quantum toroidal algebra.  One should also note some of the recent  work of Slodowy, Berman and Moody, Benkart and Zelmanov  on generalized intersection matrix algebras involve their relationship to toroidal Lie algebras (see \cite{MR93e:17031}, \cite{MR97k:17044} and \cite{MR832211}.) In addition,
Wakimoto's free field realization of affine $\widehat{\mathfrak{sl}_2}$ and Feigin and Frenkel's generalization to non-twisted affine algebras $\hat{\mathfrak g}$ play a fundamental role in describing integral solutions to the Knizhnik-Zamolodchikov equations (see for example \cite{Wa}, \cite{MR92d:17025},  \cite{MR1138049},  \cite{MR2001b:32028}, \cite{MR1077959}, \cite{MR93b:17067} and \cite{MR93b:17067}).

The representation constructed here is similar to what is often called a ``free field" representation, that is our Lie algebra elements will be realized as formal power series of noncommuting differentiable operators $a_n, n \in \mathbb Z$ acting on a given vector space $V$, where the formal power series associated with the Lie algebra become finite when applied to an element  $v \in V$. Our representation is constructed by first finding a representation of an infinite-dimensional Heisenberg like algebra, and then ``inducing" to the full toroidal algebra. 
The free field representation in this paper is a generalization of the works of the second author \cite{MR2003g:17034} and \cite{C2} which were in turn motivated by the work of Feigin and Frenkel constructing free field realizations of affine Kac Moody and $W$-algebras \cite{FF} and \cite{FF90}, as well as \cite{BF}. A completely different representation of a class of toroidal algebras given by free bosonic fields appears in \cite{JMX}. Interestingly, some free field representations of toroidal lie algebras can be used to construct vertex algebras  \cite{MR1919810} of a certain type where all simple graded modules can be classified \cite{MR2102089}. 

Part of our motivation for studying Wakimoto type realizations of toroidal $\mathfrak{sl}_{n+1}$ is to gain insight into the role of  $2$-cocycles in a more general construction of free field realizations for universal central extensions of Lie algebras of the form $\mathfrak g\otimes R$ where $R$ is an algebra over the complex numbers.  Another motivation is that they often can provide, in the generic setting, realizations in terms of partial differential operatos of imaginary type Verma modules for toroidal Lie algebras.  We plan to see how the realizations in this paper are related to these modules in future work.

\section{Notation and preliminary setting}
All vector spaces are over the field of complex numbers $\mathbb C$. 

Let $A_n=(A_{ij})_{i,j=0}^n$ be the indecomposable Cartan matrix of affine type  $A_n^1$ with $n\geq 2$.   
Let $\Pi = \{\alpha_0,\alpha_1,\dots, \alpha_n\}$ denote the simple roots, a basis for the set of roots denoted $\Delta$.   
Let $Q$ be the root lattice, i.e. the free $\mathbb Z$-module with generators $\alpha_0,\alpha_1,\dots, \alpha_n$.    The matrix $A_n$ induces a symmetric  bilinear form $(\cdot|\cdot)$ on $Q$ satisfying $(\alpha_i|\alpha_j)=A_{ij}$.   For $0\leq i\leq n$, we set $\check \alpha_i:=\alpha_i$.

We review some of the calculus of formal series, following reference \cite{LeLi} and we introduce a slightly modified for of the ${\boldsymbol \lambda}$-bracket notation and Fourier transform of \cite{Ka} which provides a very compressed notation, however many of the calculations are actually done in the more expanded form of \cite{LeLi}. 
As pointed out in \cite{LeLi}, the formal calculus generalizes to several commuting variables, the case used here. Throughout this paper $z_i, w_i, x_i, y_i, \lambda_i$ will denote mutually commuting formal variables, with $i$ ranging over some index set.
We use multi-index notation, for a positive integer $k$ given an element  $(m_0, m_1, m_2 , \dots m_k)\in \mathbb Z^{k+1}$ we write  $\mathbf m =(m_0, m_1, m_2 , \dots m_k)$, and define $\mathbf z^{\mathbf m} = z_0^{m_0} z_1^{m_1}z_2^{m_2}  \dots z_k^{m_k}$. 
Denote by $\mathbf 0$ the $k$-tuple of all zeros, and by $\mathbf 1$ the $k$-tuple of all ones. Fix a decomposition of $\mathbb Z^{k+1}=\mathbb Z^{k+1}_-\cup\{\mathbf 0\}\cup\mathbb Z^{k+1}_+$
into three disjoint subsets such that $\mathbb Z^{k+1}_\pm$ are sets closed under
vector addition i.e.
for example, if
$\mathbf j,\mathbf k\in\mathbb Z^{k+1}_+$, then $\mathbf j+\mathbf k\in
\mathbb Z^{k+1}_+$. 
Define $\mathbf m >\mathbf 0$ if
$\mathbf m\in
\mathbb Z^{k+1}_+$  and
$\mathbf m <\mathbf 0$ if
$\mathbf m\in \mathbb Z^{k+1}_-$.  Define the function
$\theta$ by
$$
\theta(\mathbf m)=\begin{cases}
1\quad \text{if}\quad\mathbf m>0 \\[2mm]
0\quad \text{otherwise.} \\[2mm]
\end{cases}
$$

We work with formal series  
\begin{equation}
a (\mathbf z) = \sum_{\mathbf n \in \mathbb Z^{k+1}} a_{\mathbf n}  \mathbf z^{-\mathbf n}
\end{equation}
with $a_{\mathbf n} \in   \mathrm {End} (V)$ for a vector space $V$ described below. The series in this paper are summable in the sense of \cite{LeLi}, i.e. the coefficient of any monomial in the formal sum acts as a finite sum of operators when applied to any vector $v \in  \mathrm {End} (V)$. For a multivariable $\mathbf z =(z_0, z_2, \dots z_k)$, let $\mathbf z^{-\mathbf 1}$ denote $ (z_0^{-1}, z_1^{-1} , \dots z_{k}^{-1}) $. To simplify notation we denote $ \mathbb C [[z_0, z_0^{-1}, z_1, z_1^{-1} , \dots z_k, z_k^{-1}]]$ as $ \mathbb C[[\mathbf z,\mathbf z^{-\mathbf 1}]]$.  Define 
\begin{equation}
\delta(\mathbf z) := \sum_{n \in \mathbb Z^{k+1} } \mathbf z^{\mathbf n} \in  \mathbb C[[\mathbf z,\mathbf z^{-1}]. 
\end{equation}
Similarly, 
$$
\delta(  \mathbf z/  \mathbf w):= \sum_{\mathbf m\in\mathbb Z^{k+1}}\mathbf z^{\mathbf m}\mathbf w^{-\mathbf m} \in \mathbb 
C[[\mathbf z, \mathbf z^{-\mathbf 1}, \mathbf w, \mathbf w^{-\mathbf 1}]]
$$
so that
$$
\delta (\mathbf z/ \mathbf w)=\prod_{i=0}^k \delta(z_i/w_i)
\qquad\text{where}\qquad  
\delta(z_i/w_i)=\sum_{k\in\mathbb Z}z_i^k w_i^{-k}.
$$

The following properties of $\delta $ hold, see Proposition 2.1.8 \cite{LeLi} which we reproduce here in the multivariable setting, 
\begin{prop}\begin{enumerate}
\item Let $f(\mathbf z) \in  V [\mathbf z,\mathbf z^{-\mathbf 1}]$. Then
\begin{equation}
f(\mathbf z) \delta (\mathbf z) = f(\mathbf 1) \delta(\mathbf z).
\end{equation}
\item  Let $ f(\mathbf z,\mathbf w) \in \mathrm{End} V [[\mathbf z,   \mathbf z^{-\mathbf 1}, \mathbf w , \mathbf w^{-\mathbf 1}  ]] $ such that $\lim_{\mathbf z \rightarrow \mathbf w} f(\mathbf z,\mathbf w)$ exists. 
 Then in $\mathrm {End}(V) [\mathbf z, \mathbf z^{-\mathbf 1},\mathbf w, \mathbf w^{-\mathbf 1}]]$
\begin{equation}
f(\mathbf z,\mathbf w) \delta (\mathbf z/\mathbf w) = f(\mathbf z,\mathbf z)\delta(\mathbf z/\mathbf w) 
= f(\mathbf w,\mathbf w)\delta(\mathbf z/\mathbf w)
\end{equation}
\end{enumerate}
\end{prop}

The formal residue for an element $f(\mathbf z) \in V[[\mathbf z,\mathbf z^{-\mathbf 1}]]$ is
$$
\mathrm{Res}_{z_i} \sum_{\mathbf n \in \mathbb Z^{k+1}} a_{\mathbf n} \mathbf z^{\mathbf n} = \sum_{\genfrac{}{}{0pt}{}{\mathbf n}  {n_i = -1} } a_{\mathbf n} z_i \mathbf z^{\mathbf n} 
$$
Alternatively we can define $\mathrm {Res}_{\mathbf z}$ as
$$
\mathrm{Res}_{\mathbf z}   \sum_{\mathbf n \in \mathbb Z^{k+1}} a_{\mathbf n}  \mathbf z^{\mathbf n} = a_{-\mathbf 1}
$$
the coefficient of $(-1, -1, \dots -1)$.

We introduce a slightly modified form of V. Kac's ${\boldsymbol \lambda}$-bracket notation and Fourier transform (see \cite{Ka}).  
For any 
$$
a( \mathbf z,  \mathbf w)=\sum_{\mathbf m,\mathbf n}a_{\mathbf m,\mathbf n}\mathbf z^{\mathbf m}\mathbf w^{\mathbf n}
$$
we define the Fourier transform
\begin{align*}
F_{ \mathbf z,  \mathbf w}^{\boldsymbol \lambda}a(  \mathbf z,  \mathbf w)=\text{Res}_{z_0}
\dots \text{Res}_{z_N}e^{\sum_{i=0}^k\lambda_i(z_i-w_i)}a(  \mathbf z, \mathbf w).
\end{align*}

For $\mathbf j=(j_0,\dots, j_k)\in\mathbb N^{k+1}$, set $\mathbf j!=j_0!j_1!\cdots j_k!$, $ \partial^{(j_i)}_{w_i}=\frac{1}{j_i!}\partial_{w_i}^{j_i}$, and  
$\partial^{(\mathbf j)}=\prod_{j=0}^k \partial^{(j_i)}_{w_i}$. We also write $ {\boldsymbol \lambda}^{(\mathbf j)}:={\boldsymbol \lambda}^{\mathbf j}/\mathbf j!=\prod_{i=0}^k \frac{\lambda_i^{j_i}}{j_i!}$. 
Using this notation allows us to compress many of the formal series we shall encounter, due to the following identity
\begin{equation}
F^{\boldsymbol \lambda}_{ \mathbf z,  \mathbf w}\partial^{(\mathbf j)}\delta(\mathbf z/   \mathbf w)={\boldsymbol \lambda}^{(\mathbf j)}
\label{IDlambda} \end{equation}
To prove identity \ref{IDlambda}  we recall a few properties shown in (\cite[Prop. 2.1]{Ka}):  For $j>0$, 
\begin{align*}
\text{Res}_z\partial_{w}^{(j)}\delta(z/w)&=0 , \\
(z-w)\partial^{(j+1)}_{w}\delta(z/w) &=\partial^{(j)}_{w}\delta(z/w),\quad\text{ and}\\
(z-w)^{j+1}\partial^{(j)}_{w}\delta(z/w) &=0  .
\end{align*}
Thus
\begin{align*}
F_{\mathbf z,\mathbf w}^{\boldsymbol \lambda}\partial^{(\mathbf j)}\delta(\mathbf z/\mathbf w)&=\text{Res}_{z_0}
\cdots \text{Res}_{z_n}e^{\sum_{i=0}^k\lambda_i(z_i-w_i)}\partial^{(\mathbf j)}\delta(\mathbf z/\mathbf w) \\
&=\text{Res}_{z_0}
\cdots \text{Res}_{z_n}\left(\prod_{i=0}^k\left(\sum_{k_i=0}^\infty\frac{1}{k_i!}  \lambda_i^{k_i}(z_i-w_i)^{k_i} \right)\prod_{i=0}^k \partial^{(j_i)}_{w_i}\delta(z_i/w_i)
\right)  \\
&=\prod_{i=0}^k\text{Res}_{z_i}
\left(\sum_{k_i=0}^\infty\frac{1}{k_i!}  \lambda_i^{k_i}(z_i-w_i)^{k_i} \right)\partial^{(j_i)}_{w_i}\delta(z_i/w_i). \\
&=\prod_{i=0}^k \frac{\lambda_i^{j_i}}{j_i!} = {\boldsymbol \lambda}^{(\mathbf j)}.
\end{align*}
If $a(  \mathbf z)$, $b(  \mathbf w)$ and $c^{\mathbf j}(  \mathbf w)$ are formal distributions satisfying 
$$
[a( \mathbf z),b( \mathbf w)]=\sum_{\mathbf j\in\mathbb N^{k+1}}c^{\mathbf j}( \mathbf w)\partial^{(\mathbf j)}\delta( \mathbf z/ \mathbf w),
$$
we have
$$
F^{\boldsymbol \lambda}_{ \mathbf z, \mathbf w}[a ( \mathbf z),b( \mathbf w)]=\sum_{\mathbf j\in\mathbb N^{k+1}}c^{\mathbf j}(\mathbf w){\boldsymbol \lambda}^{(\mathbf j)}.
$$
The  ${\boldsymbol \lambda}$-{\it bracket} is defined as
$$
[a(\mathbf w)_{\boldsymbol \lambda} b(\mathbf w)]=\sum_{\mathbf j\in\mathbb N^{k+1}}c^{\mathbf j}(\mathbf w){\boldsymbol \lambda}^{(\mathbf j)},
$$
achieving the compressed notation. When the variables are clear from the context, we sometimes omit the formal multi-variables $\mathbf z,\mathbf w$. Properties of the ${\boldsymbol \lambda}$-bracket that we use frequently include 
\begin{equation}
   [a_{\boldsymbol \lambda} [ b_{\boldsymbol \eta} c]] = [[a_{\boldsymbol \lambda} b]_{{\boldsymbol \lambda} + {\boldsymbol \eta} }c] + [b_{\boldsymbol \eta} [a_{\boldsymbol \lambda} c]] \end{equation}
   \begin{equation}
   [a_{\boldsymbol \lambda} (bc) ] = [a_{\boldsymbol \lambda} b]c  + b[a_{\boldsymbol \lambda} c] 
\label{rule2}
\end{equation}

\subsection{The toroidal Lie algebra}

Fix a positive integer $N$. We define the toroidal Lie algebra $\tau(A_n)$ by generators 

\begin{align*}
K_{\mathbf m,j},\quad H_i(\mathbf m),\quad E_i(\mathbf m),\quad F_i(\mathbf m),\quad 0\leq i\leq n,\, 0\leq j\leq N,\, \mathbf m \in\mathbb Z^{N+1}
\end{align*}
and relations  

\begin{enumerate}
\item[(0)]  \begin{itemize}
\item[i.] The $K_{\mathbf m,j}, 0\leq j \leq N$ are central;
\item[ii.] $\sum_{i=0}^Nm_iK_{\mathbf m,i}=0$;
\end{itemize}  
\label{relation1}
\item [(1)] $\left[H_i(\mathbf m),H_j(\mathbf n)\right]=A_{ij}\left(\sum_{l=0}^Nm_l K_{\mathbf m+\mathbf n, l}\right)$ ($0\leq i,j\leq n$);  \\ 
\item [(2)] $\left[H_i(\mathbf m),E_j(\mathbf n)\right]=A_{ij}E_{j }(\mathbf m+\mathbf n)$, \\
 	$\left[H_i(\mathbf m),F_{j}(\mathbf n)\right]=-A_{ij}F_{j}(\mathbf m+\mathbf n)$;  \\
\item [(3)] $\displaystyle{\left[ E_{i}(\mathbf m),F_{j}(\mathbf n)\right]=-\delta_{i,j}\left(H_{i}(\mathbf m+\mathbf n)+\frac{2}{A_{ij}}\sum_{l=0}^Nm_l K_{\mathbf m+\mathbf n ,l}\right)}$;  \\
\item[(4)] 
\begin{itemize}
\item[i.] $ \left[ E_{i}(\mathbf m),E_{i }(\mathbf n)\right]=0= \left[ F_{i}(\mathbf m),F_{i}(\mathbf n)\right]$; 
\item[ ii.]
 $\text{ad}\, E_{i }(\mathbf m)^{-A_{ij}+1}E_{j}(\mathbf n)=0$ for $i\neq j$; 
 \item[iii.]
$\text{ad}\, F_{i}(\mathbf m)^{-A_{ij}+1}F_{j}(\mathbf n)=0$ for $i\neq j$; 
\end{itemize}
\end{enumerate}
We also write generating functions for the generators of $\tau ( A_n)$ $ 1\leq i \leq n, 0\leq s \leq N$:
\begin{alignat}{2}
K_s(\mathbf z) &= \sum_{\mathbf m \in \mathbb Z^{N+1}} K_{\mathbf m, s} \mathbf z^{-\mathbf m},  
&H_i(\mathbf z)& = \sum_{\mathbf m \in \mathbb Z^{N+1}} H_i( \mathbf m ) \mathbf z^{-\mathbf m},  \\
E_i(\mathbf z)& = \sum_{\mathbf m \in \mathbb Z^{N+1}}  E_i( \mathbf m ) \mathbf z^{-\mathbf m},
\hskip 30pt& F_i(\mathbf z)& = \sum_{\mathbf m \in \mathbb Z^{N+1}}E_i( \mathbf m ) \mathbf z^{-\mathbf m}\notag
\end{alignat}

Let $\partial_{z_i} = { \frac   {\partial  }{ \partial   z_{i} } } $ denote formal differentiation. Define the operator $ D_{z_i}  = { \frac   {\partial  }{ \partial   z_{i} } }$, and $  D := \sum_{s= 0}^N D_{z_s}$ (the indeterminate in use is understood in the context of the formula). 
\begin{equation}
K(\mathbf z)\cdot D = \sum_{i=0}^N K_i(\mathbf z){\frac{\partial} {\partial z_i}},  
 K(\mathbf z)   =\sum_{i=0}^N K_i(\mathbf z) =  \sum_{i=0}^N\sum_{\mathbf m}  K_{\mathbf m,i} \mathbf z^{-\mathbf m}z_i,  
\end{equation}
\begin{equation}
  D\cdot K (\mathbf z) = \sum_{s=0}^N \sum_{\mathbf m\in \mathbb Z^{N+1}} m_p K_{\mathbf m, s} z^{\mathbf m}
\end{equation} 
\begin{enumerate}
\item[(R0)]  
\begin{itemize}
\item[i.] $K_i(\mathbf z)$ is central;
\item[ii.] $  D \cdot K(\mathbf z) =0$;  
\label{relations}
\end{itemize}
\item [(R1)] $\left[H_i(\mathbf z),H_j(\mathbf w)\right]=A_{ij} K(\mathbf w) \cdot D \delta (\mathbf z/\mathbf w)$;  \\ 
\item [(R2)] $\left[H_i(\mathbf z),E_j(\mathbf w)\right]=A_{ij}E_{j }(\mathbf w)\delta(\mathbf z/\mathbf w)$, \\
 	$\left[H_i(\mathbf z),F_{j}(\mathbf w)\right]=-A_{ij}F_{j}(\mathbf w)\delta(\mathbf z/\mathbf w)$;  \\
\item [(R3)] $\displaystyle{\left[ E_{i}(\mathbf z),F_{j}(\mathbf w)\right]=-\delta_{i,j}\left(H_{i}(\mathbf w)+\frac{2}{A_{ij}}K({\mathbf w})\cdot D\right) \delta (\mathbf w/\mathbf z)}$;  \\
\item[(R4)] \label{relation 4}
\begin{itemize}
\item[i.] $ \left[ E_{i}(\mathbf z),E_{i }(\mathbf w)\right]=0= \left[ F_{i}(\mathbf z),F_{i}(\mathbf w)\right]$; 
\item[ii.]
 $\text{ad}\, E_{i }^{-A_{ij}+1} (\mathbf z) E_{j}(\mathbf w)=0$ for $i\neq j$; 
 \item[iii.]
$\text{ad}\, F_{i}^{-A_{ij}+1}(\mathbf z)F_{j}(\mathbf w)=0$ for $i\neq j$; 
\end{itemize}
\end{enumerate}
We demonstrate how to write relation (1) as (R1):
\begin{align*}
\sum_{\mathbf m,\mathbf n}
[H_i(\mathbf m),H_j(\mathbf n)]\mathbf z^{-\mathbf m}\mathbf w^{-\mathbf n}&=A_{ij}
\sum_{\mathbf m,\mathbf n}
\left(\sum_{l=0}^Nm_l K_{\mathbf m+\mathbf n, l}\right)\mathbf z^{-\mathbf m}\mathbf w^{-\mathbf n}  \\
&=A_{ij}
\sum_{\mathbf m,\mathbf n}
\left(\sum_{l=0}^NK_{\mathbf m+\mathbf n, l}\mathbf w^{-\mathbf m-\mathbf n}\right)m_l (\mathbf w^{\mathbf m}\mathbf z^{-\mathbf m} )  \\
&=A_{ij}
\sum_{\mathbf r }
\left(\sum_{l=0}^NK_{\mathbf r, l}\mathbf w^{-\mathbf r}\right)\sum_{\mathbf m}m_l (\mathbf w^{\mathbf m}\mathbf z^{-\mathbf m} )  \\
&=A_{ij}
\sum_{\mathbf r }
\left(\sum_{l=0}^NK_{\mathbf r, l}\mathbf w^{-\mathbf r}\right) w_l\partial_{w_l}\sum_{\mathbf m}  (\mathbf w^{\mathbf m}\mathbf z^{-\mathbf m} )  \\
&=A_{ij}
\sum_{l=0}^N 
 K_{ l}(\mathbf w)\partial_{w_l}\delta(\mathbf z/\mathbf w) 
\end{align*}

\section{The toroidal Heisenberg algebra}

Define the {\it the toroidal Heisenberg algebra}, $\mathfrak B$, as the Lie algebra with generators $b_{i}(\mathbf r$) (1$\leq i \leq n$) and $K_{\mathbf r,p}$ (0 $\leq p \leq N$, $\mathbf r \in \mathbb{Z}^{n-1}$) which satisfy the following relation:

\begin{equation}\label{heisenbergrelns}
 [b_{i}(\mathbf r), b_{j}(\mathbf s)] = A_{ij}\sum_{p=0}^{N}\mathbf r_{p}K_{\mathbf {r+s},p},\quad \text{ and }\quad
 \sum_{p=0}^{N}\mathbf r_{p}K_{\mathbf {r},p}=0,\quad \forall \mathbf r\in\mathbb Z^{N+1}.
\end{equation}
Here $A_{ij}$ denotes the $i, j^{th}$ entry of the Cartan matrix $A_n$ where we have deleted the first row and column.  
If we set \begin{equation}  \label{defb0}b_0(\mathbf m):=-\sum_{i=1}^nb_i(\mathbf m),\end{equation}   then one can check that the first equality in \eqnref{heisenbergrelns} is satisfied also for $i=0$ or $j=0$. 

\subsection{Representation of the Heisenberg algebra $\mathfrak B$}

We define a polynomial ring  over indeterminates indexed by $0< i \leq n+1  $ and $\mathbf k \in {\mathbb Z}^{N+1}$. 
$$
\mathbb C[\mathbf y]:=\mathbb{C}[y_{i}(\mathbf k) | \mathbf k \in \mathbb{Z}^{N+1},\mathbf k>0,1\leq i\leq n]. 
$$
For  fixed $\kappa_{\mathbf m,p}\in\mathbb C$, $0\leq p\leq N$ and $\lambda_i\in\mathbb C$ we define a map $\Phi: \mathfrak B \rightarrow  \mathrm{End}\, \mathbb C[\mathbf y]   $ below by an action on the generators. The construction of the map is similar to that appearing in \cite{C2}. 
The motivation for the definition of $\Phi$ uses heuristic ideas about how the toroidal Lie algebra ``should" act on sections of certain (not well defined) line bundles.  For readers who are interested in this heuristic type of construction, one could consult \cite{FF90}, \cite{BF},  \cite{MR2271362}. The resulting map $\Phi$ is twisted as in \cite{C2} so that $\Phi (b_i (\mathbf m ) )$ is a well defined element of $\mathrm{End} \, \mathbb C[\mathbf y]$. The definition of $\Phi (b_0 (\mathbf m ) )$ follows from the definition(\ref{defb0}).

\begin{prop}[Realization of the Toroidal Heisenberg Algebra] \label{bosonprop} Fix $\kappa_{\mathbf m,p}\in\mathbb C$, $0\leq p\leq N$ and $\lambda_i\in\mathbb C$ where $0\leq i\leq n$.   Assume
\begin{gather}
\sum_{p=0}^Nm_p\kappa_{\mathbf m,p}=0,\quad \text{ for all }\mathbf m,\label{eqn1}  \\ 
 \sum_{p=0}^N
m_p\kappa_{-\mathbf m-\mathbf n,p}=0\quad\text{ for }\mathbf m>0\text{ and }\mathbf n>0.\label{eqn2}
\end{gather} Then the map $\Phi : \mathfrak B \rightarrow \mathrm{End}\, \mathbb C[\mathbf y]$ given by 
\begin{align*}
\Phi\big(b_{i}(\mathbf m)\big) & = \theta(-\mathbf m)\sum_{p=0}^{N}\sum_{\mathbf s>0}\bigg(\partial_{y_{i-1}(\mathbf s)}  -  \partial_{y_{i}(\mathbf s)}\bigg)   m_{p}\kappa_{\mathbf{-m+s},p}\\
& + \theta(\mathbf{m})\sum_{p=0}^{N}\sum_{\mathbf s>0} \bigg(\partial_{y_{i-1}(\mathbf s)}  - 2 \partial_{y_{i}(\mathbf s)}  +  \partial_{y_{i+1}(\mathbf s)}  \bigg)   m_{p}\kappa_{\mathbf{-m+s},p} \\
&+ \theta(-\mathbf m)y_{i}(-\mathbf m) - \delta_{\mathbf m, 0}\lambda_{i}\\ \\
\Phi(K_{\mathbf m+\mathbf n,p})&=-\kappa_{-\mathbf m-\mathbf n,p}
\end{align*}
for $1\leq i\leq n$, $\mathbf m,\mathbf n\in\mathbb Z^{N+1}$ defines a representation $\mathfrak B$ on $\mathbb C[\mathbf y]$.
\end{prop}

 \begin{proof}  For $1<i,j\leq n$ we have
\begin{align*}
[\Phi&\big(b_{i}(\mathbf m)\big)  ,\Phi\big(b_{j}(\mathbf n)\big) ] 
 =\bigg[ \theta(-\mathbf m)\sum_{p=0}^{N}\sum_{\mathbf s>0}\bigg(\partial_{y_{i-1}(\mathbf s)}   -  \partial_{y_{i}(\mathbf s)}\bigg)  m_{p}\kappa_{\mathbf{-m+s},p} \\
  &+ \theta(\mathbf{m})\sum_{p=0}^{N}\sum_{\mathbf s>0} \bigg(\partial_{y_{i-1}(\mathbf s)}  - 2 \partial_{y_{i}(\mathbf s)}  
   +  \partial_{y_{i+1}(\mathbf s)}  \bigg) \mathbf m_{p}\kappa_{\mathbf{-m+s},p}  
  + \theta(-\mathbf m)y_{i}(-\mathbf m),  \enspace   \theta(-\mathbf n)y_{j}(-\mathbf n)\bigg]\\
 &+\bigg[ \theta(-\mathbf m)y_{i}(-\mathbf m) ,  \enspace   \theta(-\mathbf n)\sum_{q=0}^{N}\sum_{\mathbf t>0}\bigg(\partial_{y_{j-1}(\mathbf t)}  -  \partial_{y_{j}(\mathbf t)}\bigg) \mathbf n_{q}\kappa_{\mathbf{-n+t},q}\\
& + \theta(\mathbf{n})\sum_{q=0}^{N}\sum_{\mathbf t>0} \bigg(\partial_{y_{j-1}(\mathbf t)}  - 2 \partial_{y_{j}(\mathbf t)}  +  \partial_{y_{j+1}(\mathbf t)}  \bigg) \mathbf n_{q}\kappa_{\mathbf{-n+t},q} 
  + \theta(-\mathbf n)y_{j}(-\mathbf n) \bigg]  \\
 =& \theta(-\mathbf n)\theta(-\mathbf m)\bigg(\delta_{j,i-1}  - \delta_{j,i}\bigg) \sum_{p=0}^{N}m_p\kappa_{\mathbf{-m-n},p} 
  + \theta(-\mathbf n)\theta(\mathbf{m})\bigg(\delta_{j,i-1} - 2\delta_{i,j}  +\delta_{j,i+1}  \bigg)  \sum_{p=0}^{N} m_p\kappa_{\mathbf{-m-n},p}     \\  
 & - \theta(-\mathbf m)\theta(-\mathbf n)\bigg(\delta_{i,j-1} - \delta_{i,j}\bigg)\sum_{q=0}^{N} n_q\kappa_{\mathbf{-m-n},q} 
  -\theta(-\mathbf m)\theta(\mathbf{n})\bigg(\delta_{i,j-1}  - 2\delta_{i,j} +\delta_{i,j+1} \bigg) \sum_{q=0}^{N} n_q\kappa_{\mathbf{-m-n},q} \\  
 =&\bigg(\theta(\mathbf n)\theta(-\mathbf m) + \theta(-\mathbf n)\theta(\mathbf{m})
 +\theta(-\mathbf m)\theta(-\mathbf{n}) \bigg)\bigg(\delta_{j,i-1} - 2\delta_{i,j} +\delta_{j,i+1}  \bigg)  \sum_{p=0}^{N} m_p\kappa_{\mathbf{-m-n},p}  \\  
 =&\bigg(\delta_{j,i-1} - 2\delta_{i,j}  + \delta_{j,i+1}  \bigg)  \sum_{p=0}^{N} m_p\kappa_{\mathbf{-m-n},p}
 =-A_{ij}  \sum_{p=0}^{N} m_p\kappa_{\mathbf{-m-n},p}.
   \end{align*}
where in the last two equalities we used the hypothesis \eqnref{eqn1} and \eqnref{eqn2} respectively. The remaining relations are also straightforward. \end{proof}

\section{Main Result, the representation of $\tau(A_n)$}
Let $i,j \leq n+1$ and $\mathbf m \in \mathbb Z^{N+1}$
$$
\mathbb C[ \mathbf x   ]:= \mathbb C[x_{ij}(\mathbf m)| 0< i< j\leq  n+1 , \mathbf m \in \mathbb Z^{N+1}] 
$$
  The elements $x_{ij}(\mathbf m)$ act via multiplication on the ring $\mathbb C[\mathbf x]$, and hence on the ring $\mathbb C[\mathbf x]\otimes \mathbb C[\mathbf y]$ (as $x_{ij}(\mathbf m)\otimes 1$).  Define the following differential operators acting on the polynomial ring $\mathbb C[\mathbf x]\otimes \mathbb C[\mathbf y]$.
\begin{equation}\label{normalordering}
a_{ij, \bf m}:= -x_{ij}(\mathbf m), \quad a_{ij,\mathbf m}^* := \frac {\partial}{  \partial x_{ij}(-\mathbf  m)}.
\end{equation}
 With corresponding generating functions:
$$
a_{ij} (\mathbf z) = \sum_{\mathbf m \in \mathbb Z^{N+1}} a_{ij, \bf m} \mathbf z^{-\mathbf m}, \quad
a^*_{ij} (\mathbf z) = \sum_{\mathbf m \in \mathbb Z^{N+1}} a^*_{ij, \bf m} \mathbf z^{-\mathbf m}, $$
$$
\kappa_s(\mathbf z) = \sum_{\mathbf m \in \mathbb Z^{N+1}} \kappa_{ \mathbf m, s} \mathbf z^{\mathbf m} 
$$
Define the operators
\begin{alignat}{2}
\\
\kappa(\mathbf z)\cdot D&= \sum_{i=0}^N \kappa_i(\mathbf z) {\frac{\partial} {\partial z_i}},  
\hskip 30pt& \kappa (\mathbf z) & =\sum_{i=0}^N\kappa_i(\mathbf z) =  \sum_{i=0}^N\sum_{\mathbf m} \kappa_{ \mathbf m,i} \mathbf z^{\mathbf m}z_i.\notag
\end{alignat}
Note that $K(\mathbf z)\cdot D$ is a weighted version of Euler's differential operator.  The operators $\Phi(b_i)$ commute with the $a_{ij,\mathbf m}, a_{ij,\mathbf m}^*  $ and act on $\mathbb C[\mathbf x]\otimes \mathbb C[\mathbf y]$ as $1 \otimes \Phi(b_i)$.

\begin{thm}[Realization]  Let $\kappa_{\mathbf m,l}$ be fixed complex numbers satisfying conditions \eqnref{eqn1} and \eqnref{eqn2} and fix $\lambda_i\in\mathbb C$ for $0\leq i\leq n$.
Then the generating functions given below 
\begin{align*}
\rho(F_{r})(\mathbf z) &=a_{r,r+1} (\mathbf z) -\sum_{j=1}^{r-1}a_{j,r+1}(\mathbf z) a_{jr}^*( \mathbf z),
\\
\rho(H_r  )(\mathbf z) &=2a_{r,r+1}(\mathbf z) a_{r,r+1}^ *(\mathbf z)+\sum_{i=1}^{r-1}\left(
        a_{i,r+1}(\mathbf z)a_{i,r+1}^*(\mathbf z)-a_{ir}(\mathbf z) a_{ir}^*(\mathbf z) \right) \\
        &\quad+\sum_{j=r+2}^{n+1}\left(a_{rj}(\mathbf z) a_{rj}^*(\mathbf z)
       - a_{r+1,j}(\mathbf z) a_{r+1,j}^*(\mathbf z)\right)+\Phi(b_r)(\mathbf z), \\
\rho(E_{r} )(\mathbf z) &= a_{r,r+1}(\mathbf z)a_{r,r+1}^*(\mathbf z) a_{r,r+1}^*(\mathbf z)\\
    &\quad
      -\sum_{j=r+2}^{n+1} a_{r+1,j}(\mathbf z) a_{rj}^*(\mathbf z)+\sum_{j=1}^{r-1}
      a_{jr}(\mathbf z)a^*_{j,r+1}(\mathbf z)\\
    &\quad+\sum_{j=r+2}^{n+1} \left(
    a_{rj}(\mathbf z)a_{rj}^*(\mathbf z)
        -a_{r+1,j}(\mathbf z)a_{r+1,j}^*(\mathbf z) \right)a_{r,r+1}^*(\mathbf z)\\
    &\quad+a_{r,r+1}^* (\mathbf z)\Phi(b_r)(\mathbf z)+\kappa\cdot Da_{r,r+1}^*(\mathbf z),
\end{align*}
for $1\leq r\leq n$, together with 
\begin{align*}
\rho(E_{0}  )(\mathbf z)  &=-a_{1,n+1}(\mathbf z),
\\
\rho(H_0   ) (\mathbf z) &=-\sum_{r=1}^n\rho(H_r)(\mathbf z)= -\sum_{r=1}^{n} 
        a_{r,n+1}(\mathbf z)a_{r,n+1}^*(\mathbf z)  -\sum_{r=2}^{n+1} (\mathbf z)a_{1r}(\mathbf z) a_{1r}^*(\mathbf z)
         +\Phi(b_0)(\mathbf z), \\
\rho(F_{0}   ) (\mathbf z) & =\sum_{1\leq r<j  \leq n+1} -a_{rj} (\mathbf z) \sum_{\mathbf q;j= q_i;r\geq q_{i-1}} \prod_{l=1}^{i-1} a_{q_l q_{l+1}}^* (\mathbf z) a_{r,n+1}^* (\mathbf z) \\
&\quad -\sum_{1 \leq r <n+1}\sum_{ r \geq  q_i,\mathbf q} 
    \prod_{j=1}^{i} a_{q_{j}q{_{j +1} }}^*(\mathbf z)  \Phi(b_r) (\mathbf z)   \\ 
&\quad -\sum_{1\leq r <n+1}\sum_{ r=q_i,\mathbf q} 
   \prod_{j=1}^{i -1}a_{q_{j}q{_{j +1} }}^* (\mathbf z) \kappa \cdot Da_{r,n+1}^* (\mathbf z)
\end{align*} 
defines an action of the generators
$E_{r}(\mathbf m)$, $F_{r}(\mathbf m)$ a
nd
$H_{r}(\mathbf m)$ on the Fock space $\mathbb C[\mathbf x]\otimes
\mathbb C[\mathbf y]$ (notation given earlier). In the partitions above $ 1=q_1<q_2<\dots <q_i,  q_{i+1} = n+1$.
In addition $K_{\mathbf m,l}$ acts as left multiplication by $-\kappa_{-\mathbf m,l} $ .
\end{thm}

Note it appears that one should also have for $k<l$,
$$
\rho(E_{lk})=-a_{kl}+\sum_{j=1}^{k-1}a_{jl}a_{jk}^*
$$
but we don't seem to need this general formula, so we don't determine whether it is always true. 

\section{Proof of the main result}

We should point out that the proof requires very lengthy (at least to us) calculations. We have selected representative portions of the calculations to include here, from an original manuscript of over one hundred pages. Calculations similar to those omitted can be found in \cite{MR2003g:17034} and \cite{C2}, students may also wish to specialize to the special cases of  cases of type $A_2$ and $A_3$ especially the latter which is a good guide for the general setting of $n\geq 2$.

Let
$$
\Phi(b_r):=\Phi(b_r)(\mathbf z)=\sum_{\mathbf m}\Phi(b_r)(\mathbf m)\mathbf z^{-\mathbf m}
$$
then we can write the last calculation in the proof of \propref{bosonprop}  as
$$
[\Phi(b_r)_{\boldsymbol \lambda}\Phi(b_s)]=A_{rs} \sum_{l=0}^N\rho(K_l)  {\lambda}_l=-A_{rs}\kappa\cdot {\boldsymbol \lambda}.
$$
Set 
$$
K_l(\mathbf w):=\left(\sum_{\mathbf r }K_{\mathbf r, l} w^{-\mathbf r}\right)w_l.
$$
Relations (R0) are satisfied by definition of the $\kappa_{\mathbf m, l}$. 
The relations (R1)-(R4) will follow if the following ${\boldsymbol \lambda}$-brackets are satisfied 
\begin{enumerate}
\label{relationT}
\item [(T1)] $\left[\rho(H_i)(\mathbf w )_{\boldsymbol \lambda}\rho(H_j)(\mathbf w) \right]=A_{ij} \sum_{l=0}^N\rho(K_l) (\mathbf w) \lambda_l$ ($0\leq i,j\leq n$);  \\ 
\item [(T2)] $\left[\rho(H_i)(\mathbf w ){_{\boldsymbol \lambda}}\rho(E_j)(\mathbf w )\right]=A_{ij}\rho(E_{j })(\mathbf w )$, \\
 	$\left[\rho(H_i)(\mathbf w ){_{\boldsymbol \lambda}} \rho(F_{j}) (\mathbf w )\right]=-A_{ij}\rho(F_{j}) (\mathbf w )$;  \\
\item [(T3)] $\displaystyle{\left[ \rho(E_{i})(\mathbf w ){_{\boldsymbol \lambda}}\rho(F_{j})(\mathbf w )\right]=-\delta_{i,j}\left(\rho(H_{i})(\mathbf w )+\frac{2}{A_{ij}}\sum_{l=0}^N\rho(K_{l})(\mathbf w )\lambda_l\right)}$;  \\
\item[(T4)] \label{relation T4}
\subitem{i.} $ \left[ \rho(E_{i})(\mathbf w )_{\boldsymbol \lambda}\rho(E_{j } )(\mathbf w )\right]=0= \left[\rho( F_{i})(\mathbf w )_{\boldsymbol \lambda}\rho(F_{j})(\mathbf w )\right]$, if $|i-j|\neq 1$. 
\subitem{ii.}
 $[\rho(E_{i }(\mathbf w ){_{\boldsymbol \lambda}} )[\rho(E_{i})(\mathbf w )_{\boldsymbol \mu }\rho(E_{j})(\mathbf w )]]=0$ if $i=j\pm 1$.
 \subitem{iii.}
 $[\rho(F_{i })(\mathbf w ){_{\boldsymbol \lambda}} [\rho(F_{i})(\mathbf w )_{\boldsymbol \mu }\rho(F_{j})(\mathbf w )]]=0$ if $i=j\pm 1$.\end{enumerate}
\begin{proof}
We demonstrate how to write relation (R1) in ${\boldsymbol \lambda}$-bracket form:
\begin{align*}
\sum_{\mathbf m,\mathbf n}
[H_i(\mathbf m),H_j(\mathbf n)]\mathbf z^{-\mathbf m}\mathbf w^{-\mathbf n}
 &=A_{ij} \sum_{l=0}^N  K_{ l}(\mathbf w)\partial_{w_l}\delta(\mathbf z/\mathbf w) \\
 &= A_{ij} \sum_{\mathbf j \in \mathbb N^{N+1}} c^{\mathbf j} (\mathbf w) \partial^{(\mathbf j)} \delta (\mathbf z/\mathbf w) 
\end{align*}
where the $c^{\mathbf j } (\mathbf w)$ is defined as follows: we take $\mathbf e_l $ be the $N$-tuple with a $1$ in the $l$-th position an zeros elsewhere, and define $c^{\mathbf e_l} (\mathbf w) = K_l (\mathbf w)$ and $c^{\mathbf j} (\mathbf w)= 0$ if $\mathbf j \neq e_l$ for some $0\leq l \leq N$.  Applying $F^{\boldsymbol \lambda}$ gives the result.
 
\end{proof}

\subsection{Preliminary Lemmas} We have the following identities for the $a_{i,j,\mathbf m}, a_{i,j,\mathbf m}^*$ as in (\ref{normalordering}),  proofs of which carry over from \cite{C2}  Lemma 4.1. The identities are 
written which we write in terms of the $\boldsymbol \lambda$-bracket. In the interest of compressing the notation, we shall often repress the variables $\mathbf z, \mathbf w$ in the computations, especially when using the $\boldsymbol \lambda$ notation, where the presence of the multivariable $\mathbf w$ is assumed. 
 
\begin{lem}[\cite{C2}]\label{prelimlem} Let $i,j,k,l \in \mathbb Z$. Then for the generating functions $a_{ij}(\mathbf w), a_{ij}^*(\mathbf w)$ the following identities hold:
 \begin{enumerate}
\item[(a)] $ [a_{ij}(\mathbf w ){_{\boldsymbol \lambda}} a^*_{kl}(\mathbf w )]=\delta_{i,k}\delta_{j,l}$,
\item[(b)] $[a_{ij}(\mathbf w )a_{ij}^*(\mathbf w ){_{\boldsymbol \lambda}} a_{ij}(\mathbf w )a^*_{ij}(\mathbf w )]=0$ ,
\item[(c)] $[a_{ij}(\mathbf w ){_{\boldsymbol \lambda}} \kappa\cdot Da_{kl}^*(\mathbf w )]
    =\delta_{i,k}\delta_{j,l}\sum_{p=0}^N\kappa_p{\lambda}_p=[\kappa \cdot Da_{kl}^*(\mathbf w ){_{\boldsymbol \lambda}}a_{ij}(\mathbf w ) ], $ 
\item[(d)] $ \displaystyle\sum_{j=r+2}^{n+1}\sum_{k=1}^{s-1}
        \Big[a_{ks} (\mathbf w )a_{k,s+1}^*(\mathbf w ){_{\boldsymbol \lambda}}
       a_{rj}(\mathbf w )a^*_{rj}(\mathbf w )\Big]= -\delta_{s,r+1}
        a_{r,r+1}(\mathbf w )a_{r,r+2}^*(\mathbf w ) , $
\item[(e)] $\displaystyle\sum_{j=1}^{r-1}\sum_{k=s+2}^{n+1} 
        \Big[a_{s+1,k} (\mathbf w )a_{sk}^*(\mathbf w ){_{\boldsymbol \lambda}}
        a_{jr}(\mathbf w )a_{j,r+1}^*(\mathbf w )\Big]=0,$
\item[(f)] $\displaystyle\sum_{j=r+2}^{n+1}\sum_{k=1}^{s-1} 
     \Big[ a_{ks}(\mathbf w )a^*_{k,s+1}(\mathbf w ){_{\boldsymbol \lambda}}
        a_{r+1,j}(\mathbf w )a^*_{r+1,j}(\mathbf w ) \Big]  =0,$
\item[(g)] \begin{align*} \displaystyle \sum_{j=r+2}^{n+1}\sum_{k=s+2}^{n+1}
 &   \Big[ a_{s+1,k} (\mathbf w )a^*_{sk}(\mathbf w ){_{\boldsymbol \lambda}}
   \left(a_{rj} (\mathbf w )a^*_{rj} (\mathbf w )
        -a_{r+1,j} (\mathbf w )a^*_{r+1,j}(\mathbf w ) \right)
         \Big]  \\
   &  =-2\delta_{r,s}\sum_{j=r+2}^{n+1} 
       a_{r+1,j}(\mathbf w )a^*_{rj} (\mathbf w )   + \delta_{r,s+1} 
    \sum_{j=r+2}^{n+1}a_{rj}(\mathbf w )a^*_{r-1,j}(\mathbf w )  \\
    &\ +\delta_{s,r+1} 
       \sum_{j=r+3}^{n+1}
        a_{r+2,k}(\mathbf w )a^*_{r+1,j}(\mathbf w ), \end{align*} 
\item[(h)] $\sum_{j=1}^{r-1}
        \Big[a^*_{s,s+1}(\mathbf w ){_{\boldsymbol \lambda}}
        a_{jr}a^*_{j,r+1}(\mathbf w )\Big]=-
        \delta_{r,s+1}a^*_{r-1,r+1}(\mathbf w ) ,$ 
\item[(i)] $\sum_{j=r+2}^{n+1} 
        \Big[ a^*_{s,s+1}(\mathbf w ){_{\boldsymbol \lambda}}
        a_{r+1,j}(\mathbf w )a_{rj}^*(\mathbf w )\Big]=-
        \delta_{r+1,s}a^*_{r,r+2}(\mathbf w ).  $
\end{enumerate}
\end{lem}

\begin{proof}  Only statement (c) is new, 
\begin{align*}
[a_{ij}(\mathbf z),  \kappa (\mathbf w)\cdot Da_{kl}^*(\mathbf w)] & =\sum_{\mathbf m}\sum_{p=0}^N \sum_{\mathbf n,\mathbf q}\kappa_{\mathbf n,p} [a_{ij}(\mathbf m),
a_{kl}^*(\mathbf q)] \mathbf w^{\mathbf n}w_p{\frac{\partial} {\partial w_p}}\mathbf z^{-\mathbf m}\mathbf w^{-\mathbf q}  \\
&=\delta_{i,k}\delta_{j,l}\sum_{p=0}^N \sum_{\mathbf n}\kappa_{\mathbf n,p}  \mathbf w^{\mathbf n}\sum_{\mathbf m}m_p
\mathbf z^{-\mathbf m}\mathbf w^{\mathbf m}  \\ 
&=\delta_{i,k}\delta_{j,l}\sum_{p=0}^N \sum_{\mathbf n}\kappa_{\mathbf n,p}\mathbf w^{\mathbf n}w_p{\frac{\partial} {\partial w_p}}\delta(\mathbf z/\mathbf w).
\end{align*}
and
\begin{align*}
[\kappa (\mathbf z)\cdot & Da_{kl}^*(\mathbf z),   a_{ij}(\mathbf w)] =\sum_{\mathbf m}\sum_{p=0}^N \sum_{\mathbf n,\mathbf q}\kappa_{\mathbf n,p} [
a_{kl}^*(\mathbf q),a_{ij}(\mathbf m)]\mathbf z^{\mathbf n}\mathbf w^{-\mathbf m}z_p{\frac{\partial} {\partial z_p}} \mathbf z^{-\mathbf q}  \\
&=-\delta_{i,k}\delta_{j,l}\sum_{p=0}^N \sum_{\mathbf n}\kappa_{\mathbf n,p} \mathbf z^{\mathbf n}\sum_{\mathbf m}m_p
 \mathbf z^{\mathbf m} \mathbf w^{-\mathbf m}  \\ 
&=\delta_{i,k}\delta_{j,l}\sum_{p=0}^N \sum_{\mathbf n}\kappa_{\mathbf n,p}\mathbf z^{\mathbf n}w_p{\frac{\partial} {\partial w_p}}\delta(\mathbf z/\mathbf w)  \\ 
 &=\delta_{i,k}\delta_{j,l}\sum_{p=0}^N \sum_{\mathbf n}\kappa_{\mathbf n,p}\mathbf w^{\mathbf n}w_p{\frac{\partial} {\partial w_p}}\delta(\mathbf z/\mathbf w)+\delta_{i,k}\delta_{j,l}\sum_{p=0}^N w_p{\frac{\partial} {\partial w_p}}\left( \sum_{\mathbf n}\kappa_{\mathbf n,p}\mathbf w^{\mathbf n}\right)\delta(\mathbf z/\mathbf w) \\
 &=\delta_{i,k}\delta_{j,l}\sum_{p=0}^N \sum_{\mathbf n}\kappa_{\mathbf n,p}\mathbf w^{\mathbf n}w_p{\frac{\partial} {\partial w_p}}\delta(\mathbf z/\mathbf w)+\delta_{i,k}\delta_{j,l} \sum_{\mathbf n}\left(\sum_{p=0}^Nn_p\kappa_{\mathbf n,p}\right)\mathbf w^{\mathbf n}
 \delta(\mathbf z/\mathbf w)  \\ 
&=\delta_{i,k}\delta_{j,l}\sum_{p=0}^N \sum_{\mathbf n}\kappa_{\mathbf n,p}\mathbf w^{\mathbf n}w_p{\frac{\partial} {\partial w_p}}\delta(\mathbf z/\mathbf w) 
\end{align*}
by the relation (R0).
Now we take the Fourier $F^{\boldsymbol \lambda}_{\mathbf z, \mathbf w}$ transform of the above, obtaining
\begin{align*}
[a_{ij}{_{\boldsymbol \lambda}}\kappa\cdot Da_{kl}^*]&=[\kappa \cdot Da_{kl}^*{_{\boldsymbol \lambda}}a_{ij} ]=\delta_{i,k}\delta_{j,l}\sum_{p=0}^N\kappa_{p}{\lambda}_p=\delta_{i,k}\delta_{j,l} \kappa\cdot {\boldsymbol \lambda}.
\end{align*}

\end{proof}

In addition, the following consequences of Lemma \ref{prelimlem} are useful
\begin{lem}
The following identities hold
\begin{equation}  \label{c1}
\left[a_{mn} (\mathbf w ){a_{mn}^*}(\mathbf w )_{{\boldsymbol \lambda}}   a_{j s+1}(\mathbf w ) a_{js}^*(\mathbf w )\right] 
= \delta_{mj}\begin{cases} 
-a_{j ,s+1}(\mathbf w ) a_{js}^* (\mathbf w ) \text{if } n = s+1 \\
a_{j ,s+1}(\mathbf w ) a_{js}^*(\mathbf w )\text{  if  }  n = s 
\end{cases}
\end{equation}

\begin{equation}  
[a_{ij} (\mathbf w ) {a_{ij}^*(\mathbf w )}_{\boldsymbol \lambda} \kappa  D  a_{mn}^*(\mathbf w ) ] = \delta_{im}\delta_{jn}  \kappa \cdot  D a_{ij}^* (\mathbf w)   +  
\delta_{im}\delta_{jn}a_{ij}^*(\mathbf w) \kappa \cdot {\boldsymbol \lambda}
\label{eq:kdw}
\end{equation}
\label{lemks}\end{lem}

 \begin{proof}
We prove only the second relation (\ref{eq:kdw}) and leave the other  to the
reader. By Lemma \ref{prelimlem} (c) and by properties of $\delta (\mathbf z/\mathbf w)$   one has
\begin{align*}
[a_{ij}(\mathbf z) a_{ij}^*(\mathbf z) &, \kappa_w D_w a_{mn}^*(\mathbf w) ] = \delta_{im}\delta_{jn}a_{ij}^*(\mathbf z)  \kappa_w D_w 
\delta(\mathbf z/\mathbf w)\\
=&   \delta_{im}\delta_{jn}\kappa_w D_w (  a_{ij }^*(\mathbf z)\delta(\mathbf z/\mathbf w)) \\
 =&  \delta_{im}\delta_{jn}\kappa_w D_w (  a_{ij }^*(\mathbf w)\delta(\mathbf z/\mathbf w)) \\
 = &  \delta_{im}\delta_{jn} (\kappa_w D_w a_{ij}^*(\mathbf w))  \delta(\mathbf z/\mathbf w)  +  \delta_{im}\delta_{jn} a_{ij}^*(\mathbf w) \kappa_w D_w \delta(\mathbf z/\mathbf w)).  
\end{align*}
Note that the formal multivariable in the series $a_{ij} (\mathbf z)$ is not affected by the operator $\kappa (\mathbf w) D_w$ which acts on series in $\mathbf w$.  Applying the transform $F^{{\boldsymbol \lambda}}_{\mathbf z, \mathbf w}$ gives the result.
 \end{proof}

\subsection{Relations involving the $H_i (\mathbf z)$}
The  relations (T1) and (T2) are simpler to verify than those of type (T3) and (T4), so we begin with them. A reader familiar with other free field representations or vertex algebras can verify relation (T1) as an excercise (see also \cite{MR2003g:17034}). 
Because of the definition of the $a_{ij}$ and $a_{ij}^*$ given in \eqnref{normalordering} there are no multiple contractions when computing out the operator product expansion for these terms. To further compress the notation, we sometimes omit the multivariable $ \mathbf w$ in our computations when the variable is clear from the context.

\begin{lem}[T2]  
$$
[\rho(H_r)(\mathbf w )_{\boldsymbol \lambda}\rho(E_s)(\mathbf w )]=A_{rs}\rho(E_s)(\mathbf w ).
$$
\end{lem}
\begin{proof}   
First assume $r,s \neq 0$.  If $|r-s|>1$,  observe that the indices of $a_{ij}$ and $a_{ij}^*$  that appear in
$\rho(H_r)(\mathbf z)$ and
$\rho(E_s)(\mathbf w)$ are disjoint and thus by Lemma \ref{prelimlem} (a) contribute nothing to the
$\boldsymbol \lambda$ bracket $[\rho(H_r)(\mathbf w)_{\boldsymbol \lambda} \rho(E_s)(\mathbf w)]$ (or equivalently to the commutator
$[\rho(H_r)(\mathbf z),\rho(E_s)(\mathbf w)]$). The remaining terms coming from the
$b_j$ have trivial commutator and thus $[\rho(H_r)(\mathbf w)_{\boldsymbol \lambda}
\rho(E_s)(\mathbf w)]=0$.

Now assume $r=s$ (with $r,s \neq 0$).  In this case
$\rho(E_r)(\mathbf w)$ is equal to 
\begin{align*}  &\quad  a_{r,r+1}a^*_{r,r+1}a^*_{r,r+1} +\sum_{j=r+2}^{n+1}
      \left(a_{rj}a^*_{rj}
        -a_{r+1,j}a^*_{r+1,j}\right)a^*_{r,r+1} \\
    &\quad +\sum_{j=1}^{r-1}
      a_{jr}a^*_{j,r+1}
      -\sum_{j=r+2}^{n+1} a_{r+1,j}a_{rj}^* +a_{r,r+1}^* \Phi(b_r)+\kappa\cdot Da_{r,r+1}^*,
\end{align*}
and $\rho(H_r)(\mathbf w)$ expands to
\begin{align*}      &2a_{r,r+1}a^*_{r,r+1} +\sum_{i=1}^{r-1}\left(
        a_{i,r+1}a^*_{i,r+1}-a_{ir}a^*_{ir}\right)
             +\sum_{j=r+2}^{n+1}\left(a_{rj}a_{rj}^*
       - a_{r+1,j}a_{r+1,j}^*\right)+\Phi(b_r),
\end{align*}
(where we have suppressed the variable $ \mathbf w  $).
Now 
\begin{align}\label{eq:HE1}
2\big[a_{r,r+1}&a^*_{r,r+1}{_{\boldsymbol \lambda}}
\rho(E_r)\big]=
  2a_{r,r+1}a^*_{r,r+1}a^*_{r,r+1}+2\sum_{j=r+2}^{n+1}
      \left(a_{r,j}a^*_{r,j}
        -a_{r+1,j}a^*_{r+1,j}\right)a^*_{r,r+1} \\
    &\hskip 100pt+2a_{r,r+1}^* \Phi(b_r)+2\kappa\cdot Da_{r,r+1}^*+2a_{r,r+1}^*\sum_{l=0}^N\kappa_l\lambda_l  .
\notag
\end{align}

The second summation in $\rho(H_r)( \mathbf w  )$ contributes
\begin{align}\label{eq:HE2}
\sum_{i=1}^{r-1}\Big[\left(
        a_{i,r+1}a^*_{i,r+1}- a_{ir}a^*_{ir}\right)
        {_{\boldsymbol \lambda}} \rho(E_r) \Big]
       &=\sum_{i=1}^{r-1}\Big[\left(
        a_{i,r+1}a^*_{i,r+1}- a_{ir}a^*_{ir}\right)_{{\boldsymbol \lambda}}
      a_{ir} a^*_{i,r+1} \Big]\\
       &=2\sum_{i=1}^{r-1}a_{ir} a^*_{i,r+1}.
\notag
\end{align}
Now in the third summation in $\rho(H_r)(\mathbf w)$ the index $j$ is greater than or equal to $r+2$ and so commutes with all but the second and fourth terms of $\rho(E_r) (\mathbf w)$ above, thus
\begin{align}  \label{eq:HE3}
\sum_{j=r+2}^{n+1} &\big[\left(a_{rj}a^*_{rj}
       -a_{r+1,j} a^*_{r+1,j}\right){_{\boldsymbol \lambda}}\rho(E_r)]  \\ \notag
       & =\sum_{j=r+2}^{n+1}\Big(
         \big[a_{rj}a^*_{rj}{_{\boldsymbol \lambda}}
       a_{rj}a^*_{rj}\big]  +\big[a_{r+1,j}a^*_{r+1,j}{_{\boldsymbol \lambda}}
    a_{r+1,j}a^*_{r+1,j}\big]\Big)  \\ \notag
      &\quad -\sum_{i=r+2}^{n+1}\sum_{j=r+2}^{n+1}\big[\left(a_{ri}a^*_{ri}
       -a_{r+1,i} a^*_{r+1,i}\right){_{\boldsymbol \lambda}}a_{r+1,j}a^*_{rj}] \\ \notag
       &=-2\sum_{j=r+2}^{n+1}a_{r+1,j}a^*_{rj}  . \notag
\end{align}
by \lemref{prelimlem} (b) and (g).
The last term in $\rho(H_r)(\mathbf w)$ contributes
\begin{align}\label{eq:HE4}
\big[\Phi(b_r){_{\boldsymbol \lambda}} \rho(E_r)\big]&=\big[\Phi(b_r){_{\boldsymbol \lambda}}a^*_{r,r+1}\Phi(b_{r})]
      =-2a^*_{r,r+1}\sum_{p=0}^N\kappa_p\lambda_p.
\end{align}
The previous four calculations, \ref{eq:HE1}, \ref{eq:HE2}, \ref{eq:HE3} and \ref{eq:HE4}, sum up to  give us the desired  result
$$[\rho(H_r)_{\boldsymbol \lambda} \rho(E_r)]=2\rho(E_r) .$$

Now suppose $s=r+1$ so that 
$\rho(E_{r+1})(\mathbf w  )$ is equal to 
\begin{align*}    &\quad a_{r+1,r+2}(a^*_{r+1,r+2})^2  +\sum_{j=r+3}^{n+1}
      \left(a_{r+1,j}a^*_{r+1,j}
        -a_{r+2,j}a^*_{r+2,j}\right)a^*_{r+1,r+2} \\
    &\quad
    +\sum_{i=1}^{r}
     a_{i,r+1} a^*_{i,r+2}
      -\sum_{j=r+3}^{n+1}a_{r+2,j}a^*_{r+1,j} +a_{r+1,r+2}^* \Phi(b_{r+1})+\kappa\cdot Da_{r+1,r+2}^* . 
\end{align*}
Then the first summand in $\rho(H_r)(\mathbf w)$ contributes
$$
2\big[a_{r,r+1}a^*_{r,r+1}{_{\boldsymbol \lambda}} \rho (E_{r+1})\big]=-
2a_{r,r+1}a^*_{r,r+2}.
$$
The second summation in $H_r(\mathbf w)$ contributes
\begin{align*}\sum_{i=1}^{r-1}\Big[\left(
       a_{i,r+1} a^*_{i,r+1}- a_{ir}a^*_{ir}\right)
        {_{\boldsymbol \lambda}} \rho (E_{r+1})\Big]&=\sum_{i=1}^{r-1}\Big[\left(
        a_{i,r+1}a^*_{i,r+1}- a_{ir}a^*_{ir}\right)_{\boldsymbol \lambda}      a_{i,r+1} a^*_{i,r+2} \Big]\\
       &=-\sum_{i=1}^{r-1}
      a_{i,r+1} a^*_{i,r+2}  .
\end{align*}
The third summand contributes by \lemref{prelimlem}
\begin{align*}\sum_{j=r+2}^{n+1}\big[&\left(a_{rj}a^*_{rj}
       - a_{r+1,j}a^*_{r+1,j},\right)_{\boldsymbol \lambda} \rho (E_{r+1})]  \\
    &=- a_{r+1,r+2}a^*_{r+1,r+2}a^*_{r+1,r+2}\\
    &\quad -\sum_{j=r+3}^{n+1}
      \left(a_{r+1,j}a^*_{r+1,j}
        -a_{r+2,j}a^*_{r+2,j}\right)a^*_{r+1,r+2} \\
    &\quad -a_{r,r+1}a^*_{r,r+2}
    -\sum_{j=r+3}^{n+1}a_{r+2,j} a^*_{r+1,j}\\
    &\quad -a^*_{r+1,r+2}\Phi(b_{r+1}) -\kappa\cdot Da_{r+1,r+2}^*-a_{r+1,r+2}^*\sum_{l=0}^N\kappa_l\lambda_i.
\end{align*}
The last summand in $\rho(H_r)(\mathbf w)$ has ${\boldsymbol \lambda}$-bracket with 
$\rho (E_{r+1})(\mathbf w)$ equal to
\begin{align*}\Big[\Phi(b_r)_{\boldsymbol \lambda}\rho (E_{r+1})\Big] 
     &= a^*_{r+1,r+2}
     [\rho(b_r)_{\boldsymbol \lambda}\rho(b_{r+1})]=a^*_{r+1,r+2}\sum_{p=0}^N\kappa_p\lambda_p.
\end{align*}
Adding the previous four equations up we get $[\rho(H_r)_{\boldsymbol \lambda}\rho
(E_{r+1})] =-\rho (E_{r+1})  $.

The final nontrivial case to consider is when $s=r-1$ (and $rs\neq 0$)  so that 
$\rho(E_{r-1})(\mathbf w)$ is equal to 
\begin{align*}&\quad a_{r-1,r}(a^*_{r-1,r})^2  +\sum_{j=r+1}^{n+1}
      \left(a_{r-1,j}a^*_{r-1,j}
        -a_{rj}a^*_{rj}\right) a^*_{r-1,r}\\
    &\quad
    +\sum_{j=1}^{r-2}
      a_{j,r-1}a^*_{jr}
      -\sum_{j=r+1}^{n+1}a_{r,j}a^*_{r-1,j} +a^*_{r-1,r}\Phi(b_{r-1})
      +\kappa\cdot Da^*_{r-1,r}.  
\end{align*}
Then 
$$
2\big[a_{r,r+1}a^*_{r,r+1}{_{\boldsymbol \lambda}} \rho(E_{r-1})\big]=-2
a^*_{r-1,r+1}a_{r,r+1} .
$$
The second summation in $\rho(H_r)(\mathbf w)$ contributes by \lemref{prelimlem}
\begin{align*}\sum_{i=1}^{r-1}\Big[&\left(
        a_{i,r+1}a^*_{i,r+1}- a_{ir}a^*_{ir}\right)_{{\boldsymbol \lambda}}
        \rho (E_{r-1})\Big]\\
    &=-a_{r-1,r}a^*_{r-1,r}a^*_{r-1,r} 
 -\sum_{j=r+1}^{n+1}
      \left(a_{r-1,j}a^*_{r-1,j}
        -a_{rj}a^*_{rj}\right)a^*_{r-1,r} \\
    &\quad    -\sum_{j=1}^{r-2}
      a_{j,r-1}a^*_{jr}
    +a_{r,r+1}a^*_{r-1,r+1}\\
   &\quad-a^*_{r-1,r}\Phi(b_{r-1})
    -\kappa\cdot Da^*_{r-1,r}-\sum_{p=0}^N\kappa_p\lambda_p.
\end{align*}
The third summand contributes
$$
\sum_{j=r+2}^{n+1}\big[\left(a_{rj}a^*_{rj}
       -a_{r+1,j} a^*_{r+1,j}\right){_{\boldsymbol \lambda}} \rho (E_{r-1}) ]
       =\sum_{j=r+2}^{n+1}a_{r,j}a^*_{r-1,j}.
$$
The last summation in $\rho(H_r)(\mathbf w)$ has $\boldsymbol \lambda$ bracket with
$\rho (E_{r-1})(\mathbf w)$ that reduces to
\begin{align*}\Big[\Phi( b_r)_{\boldsymbol \lambda} \rho (E_{r-1}) \Big] 
     =-\sum_{p=0}^N\kappa_p\lambda_p
\end{align*}
Summing the previous four equations gives
$[\rho(H_r)(\mathbf w),\rho(E_{r-1})(\mathbf w)] =-\rho(E_{r-1})(\mathbf w)$.

We now consider the case of $s=0$ and $r\neq 0$:  Then
$\rho(E_{s})(\mathbf w)=\rho(E_{0})(\mathbf w)=-a_{1,n+1}$ and hence 
$$
2\big[a_{r,r+1}a^*_{r,r+1}{_{\boldsymbol \lambda}} \rho(E_{0})\big]=0.
$$
The second summation in $\rho(H_r)(\mathbf w)$ contributes  
\begin{align*}
-\sum_{i=1}^{r-1}\Big[&\left(
        a_{i,r+1}a^*_{i,r+1}- a_{ir}a^*_{ir}\right)_{{\boldsymbol \lambda}}
        a_{1,n+1}\Big]=\delta_{r,n}a_{1,n+1}.
\end{align*}
The third summand contributes
$$
-\sum_{j=r+2}^{n+1}\big[\left(a_{rj}a^*_{rj}
       -a_{r+1,j} a^*_{r+1,j}\right){_{\boldsymbol \lambda}} a_{1,n+1} ]
       =\delta_{r,1}a_{1,n+1}.
$$
The last summation in $\rho(H_r)(\mathbf z)$ has commutator with
$\rho (E_{0})(\mathbf w)$ equal to $0$, and hence does not contribute to the $\boldsymbol \lambda$-bracket. Summing the previous three equations we get
$[\rho(H_r)(\mathbf w)_{\boldsymbol \lambda}\rho(E_{0})(\mathbf w)] =-\rho(E_{0})(\mathbf w)$.

If $r=0$, then since $\rho(H_0)=-\sum_{r=1}^n\rho(H_r)$, we get
\begin{align*}
[\rho(H_0)(\mathbf w)_{\boldsymbol \lambda}\rho(E_s)(\mathbf w)]&=-\sum_{r=1}^n[\rho(H_r)(\mathbf w)_{\boldsymbol \lambda}\rho(E_s)(\mathbf w)]  \\
&=-\sum_{r=1}^nA_{r,s}\rho(E_s)(\mathbf w)=A_{0,s}\rho(E_s)(\mathbf w)\\
\end{align*}
which holds for any $s$.  This completes the proof of the Lemma.
\end{proof}
Since our expression for $F_0(\mathbf w)$ is quite different from that of $F_i (\mathbf w)$ if $i \neq 0$ we shall prove that case separately. First we consider the case 
\begin{lem}[T2]  For $r,s \neq 0$,  
$$
[\rho(H_r)(\mathbf w)_{\boldsymbol \lambda} \rho(F_s)(\mathbf w)]=-A_{rs}\rho(F_s)(\mathbf w)
$$
\end{lem}

\begin{proof}  
We assume $s,  r \neq 0 $ in       
\begin{align*}
\left[{H_{r} }_{\boldsymbol \lambda}  F_{s} \right]  = \left[\sum_{i=1}^{r }  \right.  &  a_{i,r+1}a_{i,r+1}^*-\sum_{i=1}^{r-1} a_{ir} a_{ir}^*   \\
           & \left. +\sum_{j=r+1}^{n+1} a_{rj}a_{rj}^*
       - \sum_{j=r+2}^{n+1}a_{r+1,j}a_{r+1,j}^* +\Phi(b_r)_{{\boldsymbol \lambda}} 
     a_{s,s+1}-\sum_{j=1}^{s-1}a_{j,s+1}a_{js}^*\right] 
\end{align*}   
        
(Omitting the multivariable $\mathbf w$ as before). Using  Lemma \ref{lemks} and the fact that $\Phi(b_r) $ commutes with the $a_{ij, \mathbf m} $ and $a_{ij, \mathbf m}^* $ gives
 
\begin{align}\label{ksc1}
 \left[\sum_{i=1}^{r }   a_{i,r+1}a_{i,r+1}^*\right. & \left. -\sum_{i=1}^{r-1} a_{ir} a_{ir}^*
       +\sum_{j=r+1}^{n+1} a_{rj}a_{rj}^*
       - \sum_{j=r+2}^{n+1}a_{r+1,j}a_{r+1,j}^* +\Phi(b_r)_{{\boldsymbol \lambda}}
       a_{s,s+1}\right] 
       \\
       =&-\delta_{r+1, s+1} a_{s, s+1}   + \delta_{r, s+1}  a_{s, s+1}  - \delta_{r,s}  a_{s, s+1}   + \delta_{r+1, s}  a_{s, s+1} \notag  \\
       =& -A_{rs} a_{s,s+1} \notag
\end{align}
 
For the remaining component we must show
\begin{align}
 \left[\right( \sum_{i=1}^{r }  a_{i,r+1}a_{i,r+1}^*-\sum_{i=1}^{r-1} a_{ir} a_{ir}^*\left) \right.
        & \left. +\right( \sum_{j=r+1}^{n+1} a_{rj}a_{rj}^*
       - \sum_{j=r+2}^{n+1}a_{r+1,j}{a_{r+1,j}^* }\left)_{{\boldsymbol \lambda}} -\sum_{j=1}^{s-1}a_{j,s+1}a_{js}^*\right]  \\
       & = A_{rs} \sum_{j=1}^{s-1}a_{j,s+1}a_{js}^*
\label{ksc2}\end{align}
First note that by 
Lemma \ref{prelimlem} 
\begin{equation}
\left[\sum_{i=1}^{r }  a_{i,r+1}a_{i,r+1}^*-\sum_{i=1}^{r-1} a_{ir} {a_{ir}^*}_{\boldsymbol \lambda} -\sum_{j=1}^{s-1}a_{j,s+1}a_{js}^*\right]
= 0
\label{t2cases}\end{equation}
unless $r=s$, $r+1=s$, or $r = s+1$.
 
Equation (\ref{c1}) of Lemma \ref{lemks} allows us to compute each case.
Suppose $r=s$,  then
$$
\left[\sum_{i=1}^{r }  a_{i,r+1}a_{i,r+1}^*-\sum_{i=1}^{r-1} a_{ir} {a_{ir}^*}_{\boldsymbol \lambda} -\sum_{j=1}^{s-1}a_{j,s+1}a_{js}^*\right]\\
=\sum_{j=1}^{s-1}a_{j,s+1}a_{js}^*   +\sum_{j=1}^{s-1}a_{j,s+1}a_{js}^* 
= 2 \sum_{j=1}^{s-1}a_{j,s+1}a_{js}^* .
\label{rs}$$
Suppose $r = s-1$, then  
$$\left[-\sum_{i=1}^{r-1} a_{ir} {a_{ir}^*}_{\boldsymbol \lambda} -\sum_{j=1}^{s-1}a_{j,s+1}a_{js}^*\right]=0$$
and
$$\left[\sum_{i=1}^{r }  a_{i,r+1} {a_{i,r+1}^*}_{\boldsymbol \lambda}  -\sum_{j=1}^{s-1}a_{j,s+1}a_{js}^*\right]=- \sum_{j=1}^{s-1} a_{j, s+1} a_{js}^* .
\label{rs+1}
$$
Similarly, if $r = s+1$ 
$$\left[\sum_{i=1}^{r }  a_{i,r+1} {a_{i,r+1}^*}_{\boldsymbol \lambda}  -\sum_{j=1}^{s-1}a_{j,s+1}a_{js}^*\right]=0
$$
and
$$
\left[-\sum_{i=1}^{r-1} a_{ir} {a_{ir}^*}_{\boldsymbol \lambda} -\sum_{j=1}^{s-1}a_{j,s+1}a_{js}^*\right]=- \sum_{j=1}^{s-1} a_{j, s+1} a_{js}^*  
\label{s+1r} 
$$
We have shown
\begin{equation}
 \left[\left( \sum_{i=1}^{r }  a_{i,r+1}a_{i,r+1}^*-\sum_{i=1}^{r-1} a_{ir} a_{ir}^*\right)_{{\boldsymbol \lambda}} -\sum_{j=1}^{s-1}a_{j,s+1}a_{js}^*\right] = A_{rs} \sum_{j=1}^{s-1}a_{j,s+1}a_{js}^*
\label{kspart1}\end{equation}

Applying  Lemma \ref{prelimlem} and Lemma \ref{lemks} and splitting into cases $r > s-1, r = s-1,$ and $ 1\leq r \leq s-1$  a straightforward computation shows the remaining component satisfies
\begin{equation}
 \left[ \sum_{i=r+1}^{n+1} a_{ri}a_{ri}^*
       - \sum_{i=r+2}^{n+1}a_{r+1,i} {a_{r+1,i}^*}_{\boldsymbol \lambda} -\sum_{j=1}^{s-1}a_{j,s+1}a_{js}^*\right] = 0
\label{kspart2}\end{equation}
 for all $r,s \neq 0$. Equations (\ref{kspart1}) and (\ref{kspart2}) give (\ref{ksc2}) and the desired result. 
\end{proof}

The case of $[H_0(\mathbf w)_{\boldsymbol \lambda} F_i(\mathbf w)]$ is similar to the above and is left to the reader. 
Next we  consider the case of $F_0 (\mathbf w)$. 

\begin{lem} [T2] For all $0\leq k \leq n$, $[\rho( H_k)(\mathbf w)_{\boldsymbol \lambda} \rho( F_0)  (\mathbf w)] = -A_{k0} F_0 (\mathbf w)$.
\end{lem}
 
\begin{proof}To simplify the computation for $\rho(F_0) (\mathbf w)$ one should note that
for all positive integers $s, t$, $i$, and fixed $i+1$-tuple $\mathbf q= (q_1,q_2,q_3, \dots q_{i+1}) \in \mathbf Z^{i+1}$ with $1=q_1<q_2 < \cdots <q_i$ it follows immediately from Lemma \ref{prelimlem} that
\begin{equation}[{a_{st}a_{st}^*}_{\boldsymbol \lambda} 
 \prod_{l=1}^{i } a_{q_l q_{l+1}}^*] =   \prod_{l=1}^{i } \delta_{s q_l} \delta_{t q_{l+1} }a_{q_l q_{l+1}}^* 
\label{equ:ID} \end{equation} 
In other  words, the expression $[{a_{st}a_{st}^*}_\lambda  
 \prod_{l=1}^{i } a_{q_l q_{l+1}}^*]  $ is zero unless $s,t$ appear as consecutive integers in the increasing sequence $q_1<q_2 < \cdots <q_i$ in which case $ [{a_{st}a_{st}^*}_\lambda \cdot ]$ acts as an identity operator.   From this observation and noting that the expression $\rho (F_0)(\mathbf w)$ contains sums of strings of such products, one is motivated to arrange terms of $\rho(H_k) (\mathbf w)$ to promote cancellation, writing (where we suppress the multi-variable):
\begin{equation*}
\rho(H_k )=  \sum_{i=1}^{k }  a_{i,k+1}a_{i,k+1}^*
  - \sum_{j=k+2}^{n+1}a_{k+1,j}a_{k+1,j}^* 
  -\sum_{i=1}^{k-1} a_{ik} a_{ik}^*
       +\sum_{j=k+1}^{n+1} a_{kj}a_{kj}^*
     +\Phi(b_k).
\end{equation*}
We now consider the $\boldsymbol \lambda$ (or equivalently the commutators) of components of $\rho(H_k )(\mathbf w) $ and $\rho(F_0)(\mathbf w)$, and will show that  $[\rho(H_k)(\mathbf w)_{\boldsymbol \lambda}\rho(F_0)(\mathbf w)]$ is zero except in cases of $k = 0, 1, n$. 
Let
$$
A:= -\sum_{1\leq r < n+1} \sum_{m= r+1}^{n+1} a_{rm} \sum_{\stack{\mathbf q }{m= q_i >q_{i-1} \dots >q_1 =1 }} 
 \prod_{l=1}^{i -1} a_{q_l q_{l+1}}^*  a_{r, n+1}^* 
$$
$$
B:=- \sum_{1\leq r < n+1} \Phi(b_r) \sum_{\stack{\mathbf q}{ r\geq q_i >q_{i-1} \cdots >q_1 =1} } 
 \prod_{l=1}^{i } a_{q_l q_{l+1}}^*
$$
$$
C:= -\sum_{\stack{\mathbf q }{ r= q_i >q_{i-1} \cdots >q_1 =1} } 
 \prod_{l=1}^{i -1} a_{q_l q_{l+1}}^*  \kappa \cdot D a_{r, n+1}^*,
$$
so $\rho(F_0) (\mathbf w) = A + B + C$. 
Because it is simpler, we first consider the second component $B$.
Fix $r$ with $1\leq r < n+1$, recall $q_{i+1} = n+1$, and fix $k$ with $1\leq k  < n $. In a fixed $\mathbf q \in \mathbf Z^{i+1}$ as above, if none of the $q_j = k +1$ for $1\leq j \leq i$ then 
by (\ref{equ:ID})
\begin{equation*}
\left[\left( \sum_{i=1}^{k }
        a_{i,k+1}a_{i,k+1}^*-\sum_{j= k+2}^{n+1} a_{k+1,j}a_{k+1,j}^* \right)_{\boldsymbol \lambda}
 \prod_{l=1}^{i } a_{q_l q_{l+1}}^*
   \right] 
 =0
\end{equation*}
On the other hand, 
if $q_t = k+1$ for some (unique because of the conditions on $\mathbf q$) $ 1<t< i+1$ then (\ref{equ:ID}) shows
\begin{equation*}
\left[\left( \sum_{i=1}^{k }
        a_{i,k+1}a_{i,k+1}^*-\sum_{j= k+2}^{n+1} a_{k+1,j}a_{k+1,j}^* \right)_{\lambda}
 \prod_{l=1}^{i } a_{q_l q_{l+1}}^*
   \right] \\
   = 
 \prod_{l=1}^{i } a_{q_l q_{l+1}}^*  -    \prod_{l=1}^{i } a_{q_l q_{l+1}}^*  
  =0
\label{eq:template} \end{equation*}
Since $\Phi(b_r)(\mathbf z)$ commutes with the operators $a_{ij}(\mathbf w)$ and $a_{ij}^*(\mathbf w)$, we have shown, summing over $r, \mathbf q$ with $r\geq q_i $:
For $1\leq k< n$ 
\begin{equation} \left[\sum_{i=1}^{k }
        a_{i,k+1}a_{i,k+1}^*-\sum_{j= k+2}^{n+1} a_{k+1,j} {a_{k+1,j}^* }_{\boldsymbol \lambda} B
   \right] =0 .
   \label{eq:F0bk1}
\end{equation}
A similar argument shows that for $1< k \leq n$ 
\begin{equation}\left[ \sum_{j=k+1}^{n+1}a_{kj}a_{kj}^*
       - \sum_{i=1}^{k-1} a_{ik} {a_{ik}^*}_{\lambda}  B
   \right] =0 .
   \label{eq:F0bk2}
\end{equation}
Adding equations (\ref{eq:F0bk1}) and (\ref{eq:F0bk2}) shows, for $1<  k <  n$
\begin{equation} 
\left[\rho(H_k) (\mathbf w)_{\lambda} 
   B  \right] = 0 
\end{equation}
Now consider the first component $A$ of our realization of $F_0 (\mathbf w)$.
Note that the last term, $a_{r, n+1}^* $ may also appear in the product $ \prod_{l=1}^{i -1} a_{q_l q_{l+1}}^* $. 
Assume $k \neq n$,  apply (\ref{equ:ID})  and (\ref{rule2}) for manipulating the brackets:
\begin{align}
 \sum_{i,j} &  \left[ a_{i,k+1}a_{i,k+1}^*- a_{k+1, j} {a_{k+1, j}^*}_{\boldsymbol{ \lambda}} A \right]   \label{eq1}  \\
   &=   \sum_{i,j}  \sum_{r<m} \left[ a_{i,k+1}a_{i,k+1}^*- a_{k+1, j} {a_{k+1, j}^* }_{\boldsymbol{ \lambda}}  -a_{rm}\right] 
  \sum_{\stack{\mathbf q }{ m= q_i > \cdots >q_1 =1 }}  \prod_{l=1}^{i -1} a_{q_l q_{l+1}}^*  a_{r, n+1}^*  \notag \\ 
&\  \  -a_{rm}   \left(   \left[  a_{i,k+1}a_{i,k+1}^*- a_{k+1, j} {a_{k+1, j}^*}_{\boldsymbol{ \lambda}}  \sum_{\stack{\mathbf q }{ m= q_i >q_{i-1} \cdots >q_1 =1} }  \prod_{l=1}^{i -1} a_{q_l q_{l+1}}^*\right]  a_{r, n+1}^*   \right.  \notag \\
 &\left. \ \ + \sum_{\stack{\mathbf q }{ m= q_i >q_{i-1} \cdots >q_1 =1} }  \prod_{l=1}^{i -1} a_{q_l q_{l+1}}^*   \left[  a_{i,k+1}a_{i,k+1}^*- a_{k+1, j} {a_{k+1, j}^*}_{\boldsymbol{ \lambda}}  a_{r, n+1}^*
\right]   \right)    \notag \\
&= \sum_{i,j} \sum_{r<m}  (a_{i, k+1} \delta_{i,r} \delta_{k+1, m}  \sum_{\stack{\mathbf q}{ m= q_i r \geq q_{i-1} \cdots >q_1 =1 } }
 \prod_{l=1}^{i -1} a_{q_l q_{l+1}}^*  a_{r, n+1}^*  \notag \\  
 &\  \  -  a_{k+1, j} \delta_{k+1, r} \delta_{j,m} \sum_{\stack{\mathbf q }{m= q_i r \geq q_{i-1} \cdots >q_1 =1 } }
 \prod_{l=1}^{i -1} a_{q_l q_{l+1}}^*  a_{r, n+1}^*  ) \notag \\
&  \ \ +\sum_{r<m} ( -a_{rm} \delta_{k+1,m}    \sum_{\stack {\mathbf q }{  m= q_i >q_{i-1} \cdots >q_1 =1} }  \prod_{l=1}^{i -1} a_{q_l q_{l+1}}^*  a_{r, n+1}^*  \notag \\
 &  \ \ + a_{r,m} \sum_{\stack{\mathbf q }{ m= q_i >q_{i-1} \cdots >q_1 =1} }  \prod_{l=1}^{i -1} a_{q_l q_{l+1}}^*   \left[  a_{k+1, j} {a_{k+1, j}^*} _{\boldsymbol{ \lambda}}   a_{r, n+1}^*
\right] )  \notag  \\
&= \sum_{r<m}
(a_{r, k+1 } \delta_{m k+1}  \sum_{\stack{\mathbf q,  m= q_i }{ r \geq q_{i-1} \cdots >q_1 =1 } }
 \prod_{l=1}^{i -1} a_{q_l q_{l+1}}^*  a_{r, n+1}^*   -  a_{k+1 , m } \delta_{k+1, r}   \sum_{\stack{\mathbf q,   m= q_i }{ r \geq q_{i-1} \cdots >q_1 =1 } }
 \prod_{l=1}^{i -1} a_{q_l q_{l+1}}^*  a_{r, n+1}^*  ) \notag \\
&+ \sum_{r<m} ( -a_{rm} \delta_{k+1,m}    \sum_{\stack{\mathbf q }{ m= q_i >q_{i-1} \cdots >q_1 =1 } } \prod_{l=1}^{i -1} a_{q_l q_{l+1}}^*  a_{r, n+1}^*   
  + \delta_{r, k+1} a_{rm}  \sum_{\stack{\mathbf q }{m= q_i >q_{i-1} \cdots >q_1 =1} }  \prod_{l=1}^{i -1} a_{q_l q_{l+1}}^*   a_{r, n+1}^* ) \notag \\
&= 0 \notag
\end{align}
Thus we have shown for $1\leq k<n$ 
\begin{equation} \label{HkF0part1}
\left[  \sum_{i=1}^k a_{i,k+1}a_{i,k+1}^* - \sum_{j= k+2}^{n+1} a_{k+1, j} {a_{k+1, j}^* }_{\boldsymbol{ \lambda}} 
A \right]  = 0
\end{equation}
Using a similar argument one can also show
for $1<k\leq n$ 
\begin{equation}
\left[\sum_{i= 1}^{k-1}  -a_{ik} a_{ik}^* + \sum_{j= k+1}^{n+1}  a_{kj}{a_{kj}^*}_{\boldsymbol{ \lambda}} A \right] =0
\label{HkF0part2}
\end{equation}
 So for $1<k<n$
 $$ [H_k(\mathbf w)_{\boldsymbol \lambda} A] = 0$$
Using Lemma \ref{lemks} and equation (\ref{equ:ID}) as above, one obtains for $1\leq k <n$
\begin{equation}
\sum_{r=1}^n \left[  \sum_{i=1}^k a_{i,k+1}a_{i,k+1}^* - \sum_{j= k+2}^{n+1} a_{k+1, j} {a_{k+1, j}^* }_{\boldsymbol{ \lambda}} 
 C \right]  = 
\sum_{\stack{\mathbf q }{ k+1= q_i >q_{i-1} \cdots >q_1 =1 } }
 \prod_{l=1}^{i -1} a_{q_l q_{l+1}}^*   a_{k+1, n+1}^* \kappa \cdot  {\boldsymbol{ \lambda}}
\label{tailend1} \end{equation}
and for $k$ with $1 < k \leq  n$:
\begin{equation}
\sum_{r=1}^n\left[\sum_{i= 1}^{k-1}  -a_{ik} a_{ik}^* + \sum_{j= k+1}^{n+1}  a_{kj} {a_{kj}^*}_{\boldsymbol{ \lambda}}  
C\right] =
- \sum_{\stack{\mathbf q }{ k= q_i >q_{i-1} \cdots >q_1 =1 } }
 \prod_{l=1}^{i -1} a_{q_l q_{l+1}}^* a_{k, n+1}^*  \kappa \cdot  {\boldsymbol{ \lambda}} ,
\label{tailend2} \end{equation}
The last term $\Phi(b_k)(\mathbf z)$ appearing in $H_k(\mathbf z)$ commutes with all of the operators $a_{ij}(\mathbf w), a_{ij}^* (\mathbf w)$, so all that remains is to compute for $1\leq k <n$
\begin{align}\label{tailend3} 
  [      &  \Phi    (b_k) (\mathbf w)_{\boldsymbol{ \lambda}} 
        -  \sum_{1\leq r < n+1}  \Phi(b_r) (\mathbf w) \sum_{\stack{\mathbf q }{  r\geq q_i >\cdots >q_1 =1 } }
 \prod_{l=1}^{i } a_{q_l q_{l+1}}^*
   ] \\
  & =  \sum_{1\leq r < n+1}  A_{kr} \sum_{\stack{\mathbf q}{ r\geq q_i >\cdots >q_1 =1 } }
 \prod_{l=1}^{i } a_{q_l q_{l+1}}^*\kappa\cdot  {\boldsymbol{ \lambda}} \notag
 \\
 &=   (-1)   \sum_{\stack{\mathbf q}{ k-1\geq q_i > \cdots >q_1 =1 } }
 \prod_{l=1}^{i } a_{q_l q_{l+1}}^*\kappa\cdot  {\boldsymbol{ \lambda}}
+ 2 \sum_{\stack{\mathbf q }{ k\geq q_i > \cdots >q_1 =1 } }
 \prod_{l=1}^{i } a_{q_l q_{l+1}}^*\kappa\cdot  {\boldsymbol{ \lambda}} 
+(-1) \sum_{\stack{\mathbf q}{ k+1 \geq q_i > \cdots >q_1 =1 }} 
 \prod_{l=1}^{i } a_{q_l q_{l+1}}^*\kappa\cdot  {\boldsymbol{ \lambda}} \notag \\
  & = 
 \sum_{\stack{\mathbf q}{ k = q_i >  \cdots >q_1 =1 } }
 \prod_{l=1}^{i } a_{q_l q_{l+1}}^*\kappa\cdot {\boldsymbol{ \lambda}}
+(-1) \sum_{\stack{\mathbf q}{k+1 = q_i >  \cdots >q_1 =1 } }
 \prod_{l=1}^{i } a_{q_l q_{l+1}}^*\kappa\cdot  {\boldsymbol{ \lambda}} . \notag
\end{align}

Where we collect partitions in the last equality.
 
Now for $k\neq 0, 1, n$ we have shown
\begin{equation}
[\rho( H_k) {(\mathbf w) }_{\boldsymbol{ \lambda}} \rho(F_0)(\mathbf w) ] =  \ref{HkF0part1} + \ref{HkF0part2} + \ref{eq:F0bk1} + \ref{eq:F0bk2}+ \ref{tailend1} + \ref{tailend2} + \ref{tailend3} = 0
\end{equation}

\end{proof}

Now we consider the case $k=1$ where

\begin{equation}
 H_1 =  a_{1,2} a_{1,2}^ *  - \sum_{j=3}^{n+1} a_{2,j}a_{2,j}^*
       +\sum_{j=2}^{n+1}  a_{1j}a_{1j}^* +\Phi(b_1)
\end{equation}

Equations (\ref{eq:F0bk1}) and (\ref{HkF0part1}) hold for $k=1$ as does equation (\ref{tailend1}), so 
\begin{align}\label{eq:H12}
 \left[ a_{1,2} a_{1,2}^ *  - \sum_{j=3}^{n+1} a_{2,j}{a_{2,j}^*}_{\boldsymbol \lambda}
       F_0 \right] &=  -\sum_{\stack{\mathbf q}{ 2= q_i >q_{i-1} \cdots >q_1 =1 } }
 \prod_{l=1}^{i -1} a_{q_l q_{l+1}}^*   a_{2, n+1}^* \kappa \cdot  {\boldsymbol \lambda})\\ 
& = a_{12}^*a_{2,n+1}^* \kappa \cdot {\boldsymbol \lambda} \notag
  \end{align}
Furthermore, since in all of our $\mathbf q$, $q_1 = 1$ equations (\ref{equ:ID}) and (\ref{eq:kdw}) give
\begin{align}\label{eq:H13}
\left[ \sum_{j=2}^{n+1}  a_{1j}{a_{1j}^*}_{\boldsymbol \lambda} F_0\right] = &A +B +
\left[ \sum_{j=2}^{n+1}  a_{1j}{a_{1j}^*}_{\boldsymbol \lambda}  -  \sum_r \sum_{\stack{\mathbf q }{ r = q_i >q_{i-1} \cdots >q_1 =1 } } \prod_{l=1}^{i -1}a_{q_{l}q_{l +1} }^* \kappa\cdot Da_{r ,n+1}^*  \right]  \\
=& F_0 -a_{1,n+1}^* \kappa \cdot {\boldsymbol \lambda} \notag
\end{align}
Finally,
\begin{align}\label{eq:H11}
\left[\Phi (b_1)_{\boldsymbol \lambda}
       F_0 \right] &= 2 a_{1,n+1}^*\kappa\cdot{\boldsymbol \lambda}- 1 (a_{12}^*a_{2,n+1}^* + a_{1,n+1}^*)\kappa\cdot{\boldsymbol \lambda} \\
      & = -a_{12}^*a_{2,n+1}^*\kappa\cdot {\boldsymbol \lambda} + a_{1,n+1}^*\kappa\cdot {\boldsymbol \lambda} \notag
 \end{align}

Summing equations \ref{eq:H11}, \ref{eq:H12} and \ref{eq:H13} yields $[\rho( H_1)(\mathbf w)_{\boldsymbol \lambda} \rho(F_0)(\mathbf w)] = \rho( F_0) (\mathbf w)$ as desired.

Now consider $H_n= 
       \sum_{i=1}^{n }  a_{i,n+1}a_{i,n+1}^*-\sum_{i=1}^{n-1} a_{in} a_{in}^*
       +  a_{n,n+1}a_{n, n+1}^*
        +\Phi(b_n), 
$
writing $F_0(\mathbf w) = A  + B + C$ as above, and recalling our assumption that in all the $i+1$-tuples $\mathbf q$ the term $\mathbf q_{i+1} = n+1$ a straightforward computation using (\ref{equ:ID}) shows
\begin{equation} \label{eqA}
[  \sum_{i=1}^{n }  a_{i,n+1}{a_{i,n+1}^*}_{\boldsymbol \lambda} A] = A
\end{equation}
\begin{equation} \label{eqB}
[  \sum_{i=1}^{n }  a_{i,n+1}{a_{i,n+1}^*}_{\boldsymbol \lambda} B] = B.
\end{equation}
Furthermore
\begin{align} \label{eqC}
[  \sum_{i=1}^{n }  a_{i,n+1}{a_{i,n+1}^*}_{\boldsymbol \lambda} C]  = & -\sum_{i=1}^{n }  \sum_{1\leq r < n+1}\  \sum_{\stack{\mathbf q }{ r = q_i >q_{i-1} \cdots >q_1 =1 }  }
\prod_{l=1}^{i -1}a_{q_{l}q_{l +1} }^* [  a_{i,n+1} {a_{i,n+1}^*}_{\boldsymbol \lambda} \kappa\cdot Da_{r ,n+1}^*] \\
&= -  \sum_{1\leq r < n+1}\  \sum_{\stack{\mathbf q}{  r = q_i >q_{i-1} \cdots >q_1 =1 }}  
\prod_{l=1}^{i -1}a_{q_{l}q_{l +1} }^* [  a_{r,n+1} {a_{r,n+1}^* }_{\boldsymbol \lambda} \kappa\cdot Da_{r ,n+1}^*] \notag \\
&= C - \sum_{1\leq r < n+1}\  \sum_{\stack{\mathbf q}{  r = q_i >q_{i-1} \cdots >q_1 =1 }  }
\prod_{l=1}^{i }a_{q_{l}q_{l +1} }^*  \kappa\cdot {\boldsymbol \lambda}
 \notag \\
&= C - \sum_{\stack{\mathbf q}{ n \geq q_i >q_{i-1} \cdots >q_1 =1 }  }
\prod_{l=1}^{i }a_{q_{l}q_{l +1} }^*  \kappa\cdot {\boldsymbol \lambda} \notag
\end{align}
Equations (\ref{eq:F0bk2}) (\ref{HkF0part2}) and (\ref{tailend2}) hold for $k=n$ and show 
\begin{equation}
[-\sum_{i=1}^{n-1} a_{in} a_{in}^*
       +  a_{n,n+1}{a_{n, n+1}^*}_{\boldsymbol \lambda} A] =  0\end{equation}
\begin{equation}
[-\sum_{i=1}^{n-1} a_{in} a_{in}^*
       +  a_{n,n+1}{a_{n, n+1}^*}_{\boldsymbol \lambda} B] =  0\end{equation}
     \begin{equation} \label{eqCC}
[-\sum_{i=1}^{n-1} a_{in} a_{in}^*
       +  a_{n,n+1}{a_{n, n+1}^*}_{\boldsymbol \lambda} C] = -\sum_{n=q_i> q_{i-1} > \cdots q_1}  \prod_{l=1}^{i  }a_{q_{l}q_{l +1} }^* \kappa\cdot {\boldsymbol \lambda}
       \end{equation}

\begin{equation} \label{eqPHI}
[\Phi(b_n)_{\boldsymbol \lambda} \rho(F_0)] = (-1 \sum_{\stack{\mathbf q}{ n-1 \geq q_i >q_{i-1} \cdots >q_1 =1 }}  \prod_{l=1}^{i }a_{q_{l}q_{l +1} }^*  \kappa\cdot {\boldsymbol \lambda} 
+ 2 \sum_{\stack{\mathbf q}{  n \geq q_i >q_{i-1} \cdots >q_1 =1 } } \prod_{l=1}^{i }a_{q_{l}q_{l +1} }^*  \kappa\cdot {\boldsymbol \lambda}
\end{equation}

Summing equations (\ref{eqA})-(\ref{eqPHI}), we have 
\begin{align*}[\rho(H_n) & (\mathbf w)_{\boldsymbol \lambda} \rho(F_0)(\mathbf w)] =  A +B +C  
  - \sum_{\stack{\mathbf q}{ n \geq q_i > \cdots >q_1 =1 }  }\prod_{l=1}^{i }a_{q_{l}q_{l +1} }^*  \kappa\cdot {\boldsymbol \lambda} + 2 \sum_{\stack{\mathbf q}{  n \geq q_i >  \cdots >q_1 =1 } } 
\prod_{l=1}^{i }a_{q_{l}q_{l +1} }^*  \kappa\cdot {\boldsymbol \lambda} \\
&- (  \sum_{\stack{\mathbf q}{n=q_i>  \cdots > q_1} } \prod_{l=1}^{i  }a_{q_{l}q_{l +1} }^* \kappa\cdot {\boldsymbol \lambda} + \sum_{\stack{\mathbf q}{ n-1 \geq q_i > \cdots >q_1 =1 }}  \prod_{l=1}^{i }a_{q_{l}q_{l +1} }^*  \kappa\cdot {\boldsymbol \lambda})  \\
&= F_0 (\mathbf w)  + \sum_{\stack{\mathbf q}{  n \geq q_i >  \cdots >q_1 =1 } } 
\prod_{l=1}^{i }a_{q_{l}q_{l +1} }^*  \kappa\cdot {\boldsymbol \lambda} -\sum_{\stack{\mathbf q}{  n \geq q_i >  \cdots >q_1 =1 } } 
\prod_{l=1}^{i }a_{q_{l}q_{l +1} }^*  \kappa\cdot {\boldsymbol \lambda}\\
&= F_0 (\mathbf w)
\end{align*}
The case of $k = 0$ follows from the above, and is left to the reader.

\begin{lem}[T3]
$\displaystyle{\left[ \rho(E_{s}) (\mathbf w){_{\boldsymbol \lambda}}\rho(F_{r}) (\mathbf w)\right]=-\delta_{r,s}\left(\rho(H_{r}) (\mathbf w)+\frac{2}{A_{rs}}\sum_{l=0}^N\rho(K_{l}) (\mathbf w)\lambda_l\right)}$.
\end{lem}

\begin{proof}
For $r\neq 0$ and $s\neq 0$, the proof is nearly the same as those in  \cite[Lemma 3.4]{MR2003g:17034} where
$$
-\gamma b_r(z)-\frac{1}{2}\left(b_{r-1}^+(z)+b_{r+1}^+(z)\right)
$$
 is replaced by $\Phi(b_r)$ and 
$$
-\frac{\gamma}{2}\dot a_{r,r+1}^*(z)
$$
is replace by $\kappa\cdot D a_{r,r+1}^*$. 
We refer the interested reader to that paper for the proof.

It is also straightforward to check that 
$$
 [\rho(E_0){_{\boldsymbol \lambda}}\rho(F_r)]=-\delta_{0,r}\left(\rho(H_{r})+\frac{2}{A_{r0}}\sum_{l=0}^N\rho(K_{l})\lambda_l.\right)
$$

Now we consider $[\rho(E_s){_{\boldsymbol \lambda}}\rho(F_0)]$ with $s>0$:   We break this up into pieces:
\begin{align*}
 -[a_{s,s+1}(a_{s,s+1}^*)^2&{_{\boldsymbol \lambda}}\sum_{1\leq r<j  \leq n+1} a_{rj}  \sum_{\mathbf q;j= q_i;r\geq q_{i-1}} \prod_{l=1}^{i-1} a_{q_l q_{l+1}}^* a_{r,n+1}^*  ] \\
&=-\delta_{s,n}a_{n,n+1}(a_{n,n+1}^*)^2   \sum_{\mathbf q;n+1= q_i} \prod_{l=1}^{i-1} a_{q_l q_{l+1}}^*       \\  
&\quad -\sum_{1\leq r<j  \leq n+1} a_{rj}  \sum_{\stack{\mathbf q;j= q_i;r\geq q_{i-1}}{ \exists t:(q_t,q_{t+1})=(s,s+1)}} \prod_{l=1,l\neq t}^{i-1} a_{q_l q_{l+1}}^*  (a_{q_t,q_{t+1}}^*)^2  
    a_{r,n+1}^*  \\ 
&\quad +2a_{s,s+1}a_{s,s+1}^*  \sum_{\mathbf q;s+1= q_i;s\geq q_{i-1}} \prod_{l=1}^{i-1} a_{q_l q_{l+1}}^* 
    a_{s,n+1}^*,
\end{align*}
\begin{align*}
    -[a_{s,s+1}(a_{s,s+1}^*)^2&{_{\boldsymbol \lambda}} \sum_{1 \leq r <n+1}\sum_{\mathbf q; r \geq  q_i} 
    \prod_{j=1}^{i} a_{q_{j}q{_{j +1} }}^*  \Phi(b_r)]  \\
 & = -\sum_{1 \leq r <n+1}\sum_{\stack{\mathbf q; r \geq  q_i}{\exists t:(q_t,q_{t+1})=(s,s+1)} }
    \prod_{j=1,j\neq t}^{i} a_{q_{j}q{_{j +1} }}^*   (a_{s,s+1}^*)^2 \Phi(b_r),  
\end{align*}
\begin{align*}
 -[a_{s,s+1}(a_{s,s+1}^*)^2&{_{\boldsymbol \lambda}}\sum_{1\leq r <n+1}\sum_{\mathbf q; r=q_i} 
   \prod_{j=1}^{i -1}a_{q_{j}q{_{j +1} }}^* \kappa \cdot Da_{r,n+1}^* ] \\
   & = -\sum_{1\leq r <n+1}\sum_{\stack{\mathbf q; r=q_i }{\exists t:(q_t,q_{t+1})=(s,s+1)} }
   \prod_{j=1,j\neq t}^{i -1}a_{q_{j}q{_{j +1} }}^*(a_{s,s+1}^*)^2 \kappa \cdot Da_{r,n+1}^*   \\ 
  & \quad  -\delta_{s,n} \sum_{\mathbf q; n=q_i} 
   \prod_{j=1}^{i -1}a_{q_{j}q{_{j +1} }}^*\kappa\cdot ({\boldsymbol \lambda} +D)(a_{n,n+1}^*)^2 .
  \end{align*} 
                      
The second summand in  $\rho(E_s)$ contributes the following:
\begin{align*}
 [ \sum_{k=s+2}^{n+1} a_{s+1,k}a_{sk}^*&{_{\boldsymbol \lambda}}\sum_{1\leq r<j  \leq n+1} a_{rj}  \sum_{\mathbf q;j= q_i;r\geq q_{i-1}} \left(\prod_{l=1}^{i-1} a_{q_l q_{l+1}}^*\right) a_{r,n+1}^*  ] \\
&=  -\sum_{k=s+2}^{n+1}a_{s+1,k}\sum_{\stack{\mathbf q: k=q_i } {s\geq q_{i-1} }  } \left(\prod_{l=1}^{i-1}a_{q_l,q_{l+1}}^*\right)a_{s,n+1}^*\\ 
&\quad  +\sum_{k=s+2}^{n+1}\sum_{1\leq r<j\leq n+1}a_{rj}\sum_{\stack{ \mathbf q: j=q_i }{r \geq q_{i-1},\exists t:(q_t,q_{t+1})=(s+1,k)} }
\left(\prod_{\stack{l=1}{l\neq t}}^{i-1} a_{q_l,q_{l+1}}^* \right)a_{s,k}^* a_{r,n+1}^*\\ 
&\quad  + \sum_{ s+1<j\leq n+1}a_{s+1,j}\sum_{\stack{\mathbf q: j=q_i } { s+1\geq q_{i-1}}}\left(\prod_{l=1}^{i-1}a_{q_l,q_{l+1}}^*\right)a_{s,n+1}^* ,
\end{align*}
\begin{align*}
   [ \sum_{k=s+2}^{n+1} a_{s+1,k}a_{sk}^*&{_{\boldsymbol \lambda}} \sum_{1 \leq r <n+1}\sum_{\mathbf q; r \geq  q_i} 
    \prod_{j=1}^{i} a_{q_{j}q{_{j +1} }}^*  \Phi(b_r)]  \\
 & = \sum_{k=s+2}^{n+1} \sum_{1 \leq r <n+1}\sum_{\stack{\mathbf q; r \geq  q_i }{ \exists t:(q_t,q_{t+1})=(s+1,k)} }
    \prod_{j=1,j\neq t}^{i} a_{q_{j}q{_{j +1} }}^*   a_{s,k}^* \Phi(b_r),
\end{align*}

\begin{align*}
 [ \sum_{k=s+2}^{n+1} a_{s+1,k}a_{sk}^*&{_{\boldsymbol \lambda}}\sum_{1\leq r <n+1}\sum_{\mathbf q; r=q_i} 
   \prod_{j=1}^{i -1}a_{q_{j}q{_{j +1} }}^* \kappa \cdot Da_{r,n+1}^* ] \\
   & = \sum_{k=s+2}^{n+1} \sum_{1\leq r <n+1}\sum_{\stack{\mathbf q; r=q_i }{ \exists t:(q_t,q_{t+1})=(s+1,k)} }
   \prod_{j=1,j\neq t}^{i -1}a_{q_{j}q{_{j +1} }}^*a_{sk}^* \kappa \cdot Da_{r,n+1}^*   \\ 
  & \quad  + (1-\delta_{s,n})\sum_{\mathbf q; s+1=q_i} 
   \prod_{j=1}^{i -1}a_{q_{j}q{_{j +1} }}^*\kappa\cdot (\boldsymbol \lambda +D)a_{s,n+1}^* .
\end{align*} 

Next we consider the third summation in $\rho(E_s)$:

\begin{align*}
 -&[\sum_{k=1}^{s-1}
      a_{ks}a^*_{k,s+1}{_{\boldsymbol \lambda}}\sum_{1\leq r<j  \leq n+1} a_{rj}  \sum_{\stack{\mathbf q;j= q_i}{r\geq q_{i-1}} }\left(\prod_{l=1}^{i-1} a_{q_l q_{l+1}}^*\right) a_{r,n+1}^*  ] \\   
& = \sum_{k=1}^{s-1}a_{ks}  \sum_{\stack{\mathbf q;s+1= q_i}{k\geq q_{i-1}} }\left(\prod_{l=1}^{i-1} a_{q_l q_{l+1}}^*\right) a_{k,n+1}^*\\ 
& -\sum_{k=1}^{s-1}\sum_{1\leq r<j  \leq n+1} a_{rj}  \sum_{\stack{\mathbf q;j= q_i;r\geq q_{i-1} }{
\exists t\leq i-1:(q_t,q_{t+1})=(k,s)} } \left(\prod_{l=1,l\neq t}^{i-1} a_{q_l q_{l+1}}^*\right) a_{k,s+1}^*a_{r,n+1}^*,
\end{align*}
 
\begin{align*}
  - &[ \sum_{k=1}^{s-1}
      a_{ks}a^*_{k,s+1}{_{\boldsymbol \lambda}} \sum_{1 \leq r <n+1}\sum_{\mathbf q; r \geq  q_i} 
    \prod_{j=1}^{i} a_{q_{j}q{_{j +1} }}^*  \Phi(b_r)]  
    \\
   &=- \sum_{k=1}^{s-1}
        \sum_{1 \leq r <n+1}\sum_{ \stack{\mathbf q;\, r \geq  q_i}{ \exists t\,:\,
       (q_t,q_{t+1})=(k,s) } }
   \left( \prod_{j=1,j\neq t}^{i} a_{q_{j}q{_{j +1} }}^* \right) a^*_{k,s+1} \Phi(b_r) ,
\end{align*}
 
\begin{align*}
- &[\sum_{k=1}^{s-1}
      a_{ks}a^*_{k,s+1} {_{\boldsymbol \lambda}}\sum_{1\leq r <n+1}\sum_{\stack{\mathbf q}{ r=q_i} }
   \prod_{j=1}^{i -1}a_{q_{j}q{_{j +1} }}^* \kappa \cdot Da_{r,n+1}^* ] \\
  = &-\sum_{k=1}^{s-1}
      \sum_{1\leq r <n+1}\sum_{ \stack{\mathbf q; r=q_i}{\exists t\,:\,
       (q_t,q_{t+1})=(k,s)} }
  \left( \prod_{j=1,j\neq t}^{i -1}a_{q_{j}q{_{j +1} }}^*  \right)a^*_{k,s+1} \kappa \cdot Da_{r,n+1}^*.
  \end{align*} 
  
The fourth summation in $\rho(E_s)$ ${\boldsymbol \lambda}$-brackets with the summands of $\rho(F_0)$ as follows:
\begin{align*}
 -& \sum_{k=s+2}^{n+1} \Big[\left(
    a_{sk}a_{sk}^*
        -a_{s+1,k}a_{s+1,k}^* \right)a_{s,s+1}^*{_{\boldsymbol \lambda}}   \sum_{1\leq r<j  \leq n+1} a_{rj}  \sum_{\stack{\mathbf q}{j= q_i;r\geq q_{i-1}} }\prod_{l=1}^{i-1} a_{q_l q_{l+1}}^* a_{r,n+1}^*  \Big]   \\
 = &\sum_{j=s+2}^{n+1}  a_{sj}  a_{s,s+1}^* 
        \sum_{\stack{\mathbf q}{j= q_i;s\geq q_{i-1}}} \prod_{l=1}^{i-1} a_{q_l q_{l+1}}^* a_{s,n+1}^*   
 -\sum_{j=s+2}^{n+1}   a_{s+1,j}   a_{s,s+1}^* 
        \sum_{\stack{\mathbf q}{j= q_i;s+1\geq q_{i-1}} }\prod_{l=1}^{i-1} a_{q_l q_{l+1}}^* a_{s+1,n+1}^*  \\ 
 & +\sum_{k=s+2}^{n+1}  \left(
    a_{sk}a_{sk}^*
        -a_{s+1,k}a_{s+1,k}^* \right) 
         \sum_{\stack{\mathbf q}{s+1= q_i }} \prod_{l=1}^{i-1} a_{q_l q_{l+1}}^* a_{s,n+1}^*   \\
& -\sum_{k=s+2}^{n+1} \sum_{1\leq r<j  \leq n+1} a_{rj} 
         \sum_{\stack{\mathbf q, j= q_i;r\geq q_{i-1}  }{ \exists t;\, (q_t,q_{t+1})=(s,k)}  }
           \prod_{l=1}^{i-1} a_{q_l q_{l+1}}^*  
        a_{s,s+1}^* a_{r,n+1}^*  \\ 
 &+\sum_{k=s+2}^{n+1} \sum_{1\leq r<j  \leq n+1} a_{rj} 
         \sum_{\stack{\mathbf q;j= q_i;r\geq q_{i-1}}{ \exists t;\, (q_t,q_{t+1})=(s+1,k)} }
          \prod_{l=1}^{i-1} a_{q_l q_{l+1}}^*  
        a_{s,s+1}^* a_{r,n+1}^*  \\ 
&  -   (1-\delta_{s,n})\sum_{ s<j  \leq n+1} a_{sj}  
        \sum_{\stack{ \mathbf q;j= q_i}{s\geq q_{i-1}} }\prod_{l=1}^{i-1} a_{q_l q_{l+1}}^*a_{s,n+1}^*a_{s,s+1}^*  \\ 
  &+   \sum_{s+1<j  \leq n+1} a_{s+1,j}  
        \sum_{\stack{\mathbf q; j= q_i}{s+1\geq q_{i-1}}} \prod_{l=1}^{i-1} a_{q_l q_{l+1}}^* a_{s+1,n+1}^*a_{s,s+1}^* .
\end{align*}
 \begin{align*} 
-\sum_{k=s+2}^{n+1} &\Big[\left(
    a_{sk}a_{sk}^*
        -a_{s+1,k}a_{s+1,k}^* \right)a_{s,s+1}^*{_{\boldsymbol \lambda}}   \sum_{1 \leq r <n+1}
        \sum_{\mathbf q; r \geq  q_i} 
    \prod_{j=1}^{i} a_{q_{j}q{_{j +1} }}^*  \Phi(b_r)  \Big]  \\
    &=-\sum_{k=s+2}^{n+1}    \sum_{1 \leq r <n+1}\sum_{\stack{\mathbf q; r \geq  q_i}{ \exists t;\, 
    (q_t,q_{t+1})=(s,k)   } }
    \prod_{j=1}^{i} a_{q_{j}q{_{j +1} }}^* a_{s,s+1}^*  \Phi(b_r)   \\
&+\sum_{k=s+2}^{n+1} \sum_{1 \leq r <n+1}\sum_{\mathbf q; r\geq q_i\atop \exists t;\, (q_t,q_{t+1})=(s+1,k)  } 
    \prod_{j=1}^{i} a_{q_{j}q{_{j +1} }}^*  a_{s,s+1}^* \Phi(b_r) .
\end{align*}
and
\begin{align*}
-&\sum_{k=s+2}^{n+1} \Big[\left(
    a_{sk}a_{sk}^*
        -a_{s+1,k}a_{s+1,k}^* \right)a_{s,s+1}^*{_{\boldsymbol \lambda}}  \sum_{1\leq r <n+1}\sum_{\mathbf q; r=q_i} 
   \prod_{j=1}^{i -1}a_{q_{j}q{_{j +1} }}^* \kappa \cdot Da_{r,n+1}^* \Big]  \\
&= -   \sum_{s\leq r <n+1}\sum_{k=s+2}^{r}\sum_{\mathbf q\ ; r=q_i, \exists t\leq i-1, (q_t,q_{t+1})=(s,k)} 
   \prod_{j=1}^{i -1}a_{q_{j}q{_{j +1} }}^*a_{s,s+1}^* \kappa \cdot Da_{r,n+1}^*  \\
&\quad+ \sum_{s+1\leq r <n+1}\sum_{k=s+2}^{r}\sum_{\mathbf q\ ; r=q_i, \exists t\leq i-1, (q_t,q_{t+1})=(s+1,k)} 
   \prod_{j=1}^{i -1}a_{q_{j}q{_{j +1} }}^*a_{s,s+1}^* \kappa \cdot Da_{r,n+1}^*  \\
&\quad -(1-\delta_{s,n})  \sum_{ \mathbf q; s=q_i} 
   \prod_{j=1}^{i -1}a_{q_{j}q{_{j +1} }}^* \kappa\cdot ({\boldsymbol \lambda} +D)(a_{s,n+1}^*
        a_{s,s+1}^*) \\
&\quad +(1-\delta_{s,n})  \sum_{ \mathbf q; s+1=q_i} 
   \prod_{j=1}^{i -1}a_{q_{j}q{_{j +1} }}^*\kappa\cdot( {\boldsymbol \lambda} +D)(a_{s+1,n+1}^*  a_{s,s+1}^*).
\end{align*}
Now we consider the last two summands in $\rho(E_s)$:
\begin{align*}
- &[(a_{s,s+1}^* \Phi(b_s)+\kappa\cdot Da_{s,s+1}^*){_{\boldsymbol \lambda}}\sum_{1\leq r<j  \leq n+1} a_{rj}  \sum_{\mathbf q;j= q_i;r\geq q_{i-1}} \left(\prod_{l=1}^{i-1} a_{q_l q_{l+1}}^*\right) a_{r,n+1}^*  ] \\ 
 &=   \sum_{\mathbf q;s+1= q_i;s\geq q_{i-1}} \left(\prod_{l=1}^{i-1} a_{q_l q_{l+1}}^*\right) a_{s,n+1}^* \Phi(b_s)  - \sum_{\mathbf q;s+1= q_i;s\geq q_{i-1}} \left(\prod_{l=1}^{i-1} a_{q_l q_{l+1}}^*\right) a_{s,n+1}^* \kappa\cdot {\boldsymbol \lambda} 
 \end{align*}
 and
 \begin{align*}
- [(a_{s,s+1}^* \Phi(b_s)+\kappa\cdot Da_{s,s+1}^*)&{_{\boldsymbol \lambda}}   \sum_{1 \leq r <n+1}
        \sum_{\mathbf q; r \geq  q_i} 
    \prod_{j=1}^{i} a_{q_{j}q{_{j +1} }}^*  \Phi(b_r)    ] \\ 
 &=   \sum_{1 \leq r <n+1}
        \sum_{\mathbf q; r \geq  q_i} 
    \prod_{j=1}^{i} a_{q_{j}q{_{j +1} }}^* A_{sr}\kappa\cdot ({\boldsymbol \lambda}+D)a_{s,s+1}^*   .
 \end{align*}
 
 Summing up the terms with $\Phi(b_k)$'s in them we get
 \begin{align*}
&-\sum_{1 \leq r <n+1}\sum_{\mathbf q; r \geq  q_i\atop \exists t:(q_t,q_{t+1})=(s,s+1)} 
    \prod_{j=1,j\neq t}^{i} a_{q_{j}q{_{j +1} }}^*   (a_{s,s+1}^*)^2 \Phi(b_r),  \\
&+\sum_{k=s+2}^{n+1} \sum_{1 \leq r <n+1}\sum_{\mathbf q; r \geq  q_i\atop \exists t:(q_t,q_{t+1})=(s+1,k)} 
    \prod_{j=1,j\neq t}^{i} a_{q_{j}q{_{j +1} }}^*   a_{s,k}^* \Phi(b_r) \\ 
&- \sum_{k=1}^{s-1}
        \sum_{1 \leq r <n+1}\sum_{ \mathbf q;\, r \geq  q_i\atop \exists t\,:\,
       (q_t,q_{t+1})=(k,s) } 
   \left( \prod_{j=1,j\neq t}^{i} a_{q_{j}q{_{j +1} }}^* \right) a^*_{k,s+1} \Phi(b_r) \\ 
  & -\sum_{k=s+2}^{n+1}    \sum_{1 \leq r <n+1}\sum_{\mathbf q; r \geq  q_i\atop \exists t;\enspace 
    (q_t,q_{t+1})=(s,k)   } 
    \prod_{j=1}^{i} a_{q_{j}q{_{j +1} }}^* a_{s,s+1}^*  \Phi(b_r)   \\
&+\sum_{k=s+2}^{n+1} \sum_{1 \leq r <n+1}\sum_{\mathbf q; r\geq q_i\atop \exists t;\enspace (q_t,q_{t+1})=(s+1,k)  } 
    \prod_{j=1}^{i} a_{q_{j}q{_{j +1} }}^*  a_{s,s+1}^* \Phi(b_r)  \\ 
 &+\sum_{\mathbf q;s+1= q_i;s\geq q_{i-1}} \left(\prod_{l=1}^{i-1} a_{q_l q_{l+1}}^*\right) a_{s,n+1}^* \Phi(b_s) \\  \\ 
&=-\sum_{r=s+1}^n\sum_{\mathbf q; r \geq  q_i\atop \exists t:(q_t,q_{t+1})=(s,s+1)} 
    \prod_{j=1 }^{i} a_{q_{j}q{_{j +1} }}^*   a_{s,s+1}^* \Phi(b_r),  \\
&\quad+\sum_{k=s+2}^{n+1} \sum_{r=s+1}^n\sum_{\mathbf q; r \geq  q_i\atop \exists t:(q_t,q_{t+1})=(s+1,k)} 
    \prod_{j=1,j\neq t}^{i} a_{q_{j}q{_{j +1} }}^*   a_{s,k}^* \Phi(b_r) \\ 
& \quad- \sum_{k=1}^{s-1}\sum_{r=s+1}^n\sum_{ \mathbf q;\, r \geq  q_i\atop \exists t\,:\,
       (q_t,q_{t+1})=(k,s) } 
   \left( \prod_{j=1,j\neq t}^{i} a_{q_{j}q{_{j +1} }}^* \right) a^*_{k,s+1} \Phi(b_r) \\ 
& \quad- \sum_{k=1}^{s-1}
         \sum_{ \mathbf q;\, s = q_i } 
   \left( \prod_{j=1}^{i-1} a_{q_{j}q{_{j +1} }}^* \right) a^*_{k,s+1} \Phi(b_s)
    \\ 
  & \quad -  \sum_{  \mathbf q,s=  q_i ; 
    (q_i,q_{i+1})=(s,n+1)   } 
    \prod_{j=1}^{i} a_{q_{j}q{_{j +1} }}^* a_{s,s+1}^*  \Phi(b_s)   \\
 & \quad -\sum_{k=s+2}^{n+1}    \sum_{r=s+1}^n\sum_{\mathbf q; r \geq  q_i\atop \exists t;\enspace 
    (q_t,q_{t+1})=(s,k)   } 
    \prod_{j=1}^{i} a_{q_{j}q{_{j +1} }}^* a_{s,s+1}^*  \Phi(b_r)   \\
& \quad +\sum_{k=s+2}^{n+1} \sum_{r =s+1}^n\sum_{\mathbf q; r\geq q_i\atop \exists t;\enspace (q_t,q_{t+1})=(s+1,k)  } 
    \prod_{j=1}^{i} a_{q_{j}q{_{j +1} }}^*  a_{s,s+1}^* \Phi(b_r)  \\ 
 & \quad +\sum_{\mathbf q;s+1= q_i;s\geq q_{i-1}} \left(\prod_{l=1}^{i-1} a_{q_l q_{l+1}}^*\right) a_{s,n+1}^* \Phi(b_s) \\
\end{align*}
\begin{align*}
&=-\sum_{r=s+1}^n\sum_{\mathbf q; r \geq  q_i\atop \exists t:(q_t,q_{t+1})=(s,s+1)} 
    \prod_{j=1 }^{i} a_{q_{j}q{_{j +1} }}^*   a_{s,s+1}^* \Phi(b_r),  \\
& \quad +\sum_{k=s+2}^{n+1} \sum_{r=s+1}^n\sum_{\mathbf q; r \geq  q_i\atop \exists t:(q_t,q_{t+1})=(s+1,k)} 
    \prod_{j=1,j\neq t}^{i} a_{q_{j}q{_{j +1} }}^*   a_{s,k}^* \Phi(b_r) \\ 
& \quad- \sum_{k=1}^{s-1}\sum_{r=s+1}^n\sum_{ \mathbf q;\, r \geq  q_i\atop \exists t\,:\,
       (q_t,q_{t+1})=(k,s) } 
   \left( \prod_{j=1,j\neq t}^{i} a_{q_{j}q{_{j +1} }}^* \right) a^*_{k,s+1} \Phi(b_r) \\ 
 & \quad -\sum_{k=s+2}^{n+1}    \sum_{r=s+1}^n\sum_{\mathbf q; r \geq  q_i\atop \exists t;\enspace 
    (q_t,q_{t+1})=(s,k)   } 
    \prod_{j=1}^{i} a_{q_{j}q{_{j +1} }}^* a_{s,s+1}^*  \Phi(b_r)   \\
& \quad +\sum_{k=s+2}^{n+1} \sum_{r =s+1}^n\sum_{\mathbf q; r\geq q_i\atop \exists t;\enspace (q_t,q_{t+1})=(s+1,k)  } 
    \prod_{j=1}^{i} a_{q_{j}q{_{j +1} }}^*  a_{s,s+1}^* \Phi(b_r)  \\ 
& \quad- \sum_{k=1}^{s-1}
         \sum_{ \mathbf q;\, s = q_i } 
   \left( \prod_{j=1}^{i-1} a_{q_{j}q{_{j +1} }}^* \right) a^*_{k,s+1} \Phi(b_s)
    \\ 
& \quad-   \sum_{ \mathbf q,s=  q_i  ; 
    (q_i,q_{i+1})=(s,n+1)   } 
    \prod_{j=1}^{i} a_{q_{j}q{_{j +1} }}^* a_{s,s+1}^*  \Phi(b_s)   \\
& \quad +\sum_{\mathbf q;s+1= q_i;s\geq q_{i-1}} \left(\prod_{l=1}^{i-1} a_{q_l q_{l+1}}^*\right) a_{s,n+1}^* \Phi(b_s) \\  \\
&=-\sum_{r=s+1}^n\sum_{\mathbf q; r \geq  q_i } a_{1,q_1}^*
    \cdots a_{q_{t-1}s}^* a_{s,s+1 }^* a_{s+1,q{_{j +1} }}^* \cdots a_{q_{i},n+1}^*  a_{s,s+1}^* \Phi(b_r),  \\
& \quad + \sum_{r=s+1}^n \sum_{k=s+2}^{n+1}\sum_{\mathbf q; r \geq  q_i } 
     a_{1,q_1}^*\cdots a_{q_{t-1}s+1}^*   a_{k,q{_{t +2} }}^* \cdots a_{q_{i},n+1}^*   a_{s,k}^* \Phi(b_r) \\ 
& \quad- \sum_{r=s+1}^n\sum_{k=1}^{s-1}\sum_{ \mathbf q;\, r \geq  q_i\atop \exists t\,:\,
       (q_t,q_{t+1})=(k,s) } 
     a_{1,q_1}^*\cdots a_{q_{t-1}k}^*   a_{s,q{_{t +2} }}^* \cdots a_{q_{i},n+1}^*a^*_{k,s+1} \Phi(b_r) \\ 
 & \quad-  \sum_{r=s+1}^n\sum_{k=s+2}^{n+1}  \sum_{\mathbf q; r \geq  q_i } 
   	a_{1,q_1}^*\cdots a_{s,k}^*    \cdots a_{q_{i},n+1}^* a_{s,s+1}^*  \Phi(b_r)   \\
& \quad + \sum_{r =s+1}^n\sum_{k=s+2}^{n+1}\sum_{\mathbf q; r\geq q_i  } 
     	a_{1,q_1}^*\cdots a_{s+1,k}^*    \cdots
    a_{q_{i},n+1}^*  a_{s,s+1}^* \Phi(b_r)  
\end{align*}
\begin{align*}
&= \sum_{r=s+1}^n \sum_{k=s+2}^{n+1}\sum_{\mathbf q; r \geq  q_i } 
     a_{1,q_1}^*\cdots a_{q_{t-1}s+1}^*   a_{k,q{_{t +2} }}^* \cdots a_{q_{i},n+1}^*   a_{s,k}^* \Phi(b_r) \\ 
& \quad- \sum_{r=s+1}^n\sum_{k=1}^{s-1}\sum_{ \mathbf q;\, r \geq  q_i } 
     a_{1,q_1}^*\cdots a_{q_{t-1}k}^*   a_{s,q{_{t +2} }}^* \cdots a_{q_{i},n+1}^*a^*_{k,s+1} \Phi(b_r) \\ 
 & \quad -  \sum_{r=s+1}^n\sum_{k=s+2}^{n+1}  \sum_{\mathbf q; r \geq  q_i } 
   	a_{1,q_1}^*\cdots a_{s,k}^*    \cdots a_{q_{i},n+1}^* a_{s,s+1}^*  \Phi(b_r)   \\
& \quad + \sum_{r =s+1}^n\sum_{k=1}^{s-1}\sum_{\mathbf q; r\geq q_i  } 
     	a_{1,q_1}^*\cdots a_{q_{t-1},k}^*a_{k,s+1}^*     \cdots
    a_{q_{i},n+1}^*  a_{s,s+1}^* \Phi(b_r)   \\  \\
&= \sum_{r=s+1}^n \sum_{k=s+2}^{n+1}\sum_{\mathbf q; r \geq  q_i } 
     a_{1,q_1}^*\cdots a_{q_{t-1}s+1}^*   a_{k,q{_{t +2} }}^* \cdots a_{q_{i},n+1}^*   a_{s,k}^* \Phi(b_r) \\ 
&- \sum_{r=s+1}^n\sum_{k=1}^{s-1}\sum_{l=k+2}^{n+1}\sum_{ \mathbf q;\, r \geq  q_i } 
     a_{1,q_1}^*\cdots a_{q_{t-1}k}^*   a_{s,l}^* \cdots a_{q_{i},n+1}^*a^*_{k,s+1} \Phi(b_r) \\ 
&- \sum_{r=s+1}^n\sum_{k=1}^{s-1} \sum_{ \mathbf q;\, r \geq  q_i } 
     a_{1,q_1}^*\cdots a_{q_{t-1}k}^*   a_{s,s+1}^* a_{s+1,q_{t+2}}^*\cdots a_{q_{i},n+1}^*a^*_{k,s+1} \Phi(b_r) \\ 
 & -  \sum_{r=s+1}^n\sum_{k=s+2}^{n+1}  \sum_{\mathbf q; r \geq  q_i } 
   	a_{1,q_1}^*\cdots a_{s,k}^*    \cdots a_{q_{i},n+1}^* a_{s,s+1}^*  \Phi(b_r)   \\
&+ \sum_{r =s+1}^n\sum_{k=1}^{s-1}\sum_{\mathbf q; r\geq q_i  } 
     	a_{1,q_1}^*\cdots a_{q_{t-1},k}^*a_{k,s+1}^*     \cdots
    a_{q_{i},n+1}^*  a_{s,s+1}^* \Phi(b_r)   \\ 
    \end{align*}
    \begin{align*}
&= \sum_{r=s+1}^n \sum_{k=s+2}^{n+1} \sum_{\mathbf q; r \geq  q_i } 
     a_{1,q_1}^*\cdots a_{s,s+1}^*   a_{k,q{_{t +2} }}^* \cdots a_{q_{i},n+1}^*   a_{s,k}^* \Phi(b_r) \\ 
&+ \sum_{r=s+1}^n \sum_{k=s+2}^{n+1}\sum_{l=1}^{s-1}\sum_{\mathbf q; r \geq  q_i } 
     a_{1,q_1}^*\cdots a_{ls+1}^*   a_{k,q{_{t +2} }}^* \cdots a_{q_{i},n+1}^*   a_{s,k}^* \Phi(b_r) \\ 
&- \sum_{r=s+1}^n\sum_{k=1}^{s-1}\sum_{l=s+2}^{n+1}\sum_{ \mathbf q;\, r \geq  q_i } 
     a_{1,q_1}^*\cdots a_{q_{t-1}k}^*   a_{s,l}^* \cdots a_{q_{i},n+1}^*a^*_{k,s+1} \Phi(b_r) \\ 
 & -  \sum_{r=s+1}^n\sum_{k=s+2}^{n+1}  \sum_{\mathbf q; r \geq  q_i } 
   	a_{1,q_1}^*\cdots a_{s,k}^*    \cdots a_{q_{i},n+1}^* a_{s,s+1}^*  \Phi(b_r)   \\ 
&= \sum_{r=s+1}^n \sum_{k=s+2}^{n+1} \sum_{l=1}^{s-1}\sum_{\mathbf q; r \geq  q_i } 
     a_{1,q_1}^*\cdots a_{l,s}^*a_{s,s+1}^*   a_{k,q{_{t +2} }}^* \cdots a_{q_{i},n+1}^*   a_{s,k}^* \Phi(b_r) \\ 
&+ \sum_{r=s+1}^n \sum_{k=s+2}^{n+1}\sum_{l=1}^{s-1}\sum_{\mathbf q; r \geq  q_i } 
     a_{1,q_1}^*\cdots  a_{q_{t-1}l}^* a_{ls+1}^*   a_{k,q{_{t +2} }}^* \cdots a_{q_{i},n+1}^*   a_{s,k}^* \Phi(b_r) \\ 
&- \sum_{r=s+1}^n\sum_{l=1}^{s-1}\sum_{k=s+2}^{n+1}\sum_{ \mathbf q;\, r \geq  q_i } 
     a_{1,q_1}^*\cdots a_{q_{t-1}l}^*   a_{s,k}^*a_{k,q{_{t +2} }}^* \cdots a_{q_{i},n+1}^*a^*_{l,s+1} \Phi(b_r) \\ 
 & -  \sum_{r=s+1}^n\sum_{k=s+2}^{n+1}  \sum_{l=1}^{s-1}\sum_{\mathbf q; r \geq  q_i } 
   	a_{1,q_1}^*\cdots a_{l,s}^*a_{s,k}^*  a_{k,q_{t+2}}^*    \cdots a_{q_{i},n+1}^* a_{s,s+1}^*  \Phi(b_r)  =0. 
\end{align*}

Next we consider the terms with the $\kappa\cdot {\boldsymbol \lambda}$'s and $D$'s in them (here $A_{1,0}:=0=A_{n+1,n+2}:=0$):
\begin{align*}
 &  -\sum_{s+1\leq r <n+1}\sum_{\mathbf q; r=q_i\atop \exists t:(q_t,q_{t+1})=(s,s+1)} 
   	\prod_{j=1,j\neq t}^{i -1}a_{q_{j}q{_{j +1} }}^*(a_{s,s+1}^*)^2 \kappa \cdot Da_{r,n+1}^*   \\ 
& \quad  -\delta_{s,n} \sum_{\mathbf q; n=q_i} 
   	\prod_{j=1}^{i -1}a_{q_{j}q{_{j +1} }}^*\kappa\cdot ({\boldsymbol \lambda} +D)(a_{n,n+1}^*)^2  \\
& \quad +\sum_{s+2\leq r <n+1} \sum_{k=s+2}^{r} 
	\sum_{\mathbf q; r=q_i\atop \exists t:(q_t,q_{t+1})=(s+1,k)} 
   	\prod_{j=1,j\neq t}^{i -1}a_{q_{j}q{_{j +1} }}^*a_{sk}^* \kappa \cdot Da_{r,n+1}^*   \\ 
& \quad  +  (1-\delta_{s,n}) \sum_{\mathbf q; s+1=q_i} 
  	 \prod_{j=1}^{i -1}a_{q_{j}q{_{j +1} }}^*\kappa\cdot ( {\boldsymbol \lambda} +D)a_{s,n+1}^*  \\ 
& \quad-
      \sum_{s\leq r \leq n+1}\sum_{k=1}^{s-1}\sum_{ \mathbf q; r=q_i,\atop \exists t\leq i-1\,:\,
       (q_t,q_{t+1})=(k,s)} 
  \left( \prod_{j=1,j\neq t}^{i -1}a_{q_{j}q{_{j +1} }}^*  \right)a^*_{k,s+1} \kappa \cdot Da_{r,n+1}^* \\
&\quad- \sum_{s+2\leq r <n+1} \sum_{k=s+2}^{r}\sum_{\mathbf q\ ; r=q_i, \exists t\leq i-1, (q_t,q_{t+1})=(s,k)} 
   \prod_{j=1}^{i -1}a_{q_{j}q{_{j +1} }}^*a_{s,s+1}^* \kappa \cdot Da_{r,n+1}^*  \\
&\quad+  \sum_{s+2\leq r <n+1} \sum_{k=s+2}^{r}\sum_{\mathbf q\ ; r=q_i, \exists t\leq i-1, (q_t,q_{t+1})=(s+1,k)} 
   \prod_{j=1}^{i -1}a_{q_{j}q{_{j +1} }}^*a_{s,s+1}^* \kappa \cdot Da_{r,n+1}^*  \\
&\quad -  (1-\delta_{s,n}) \sum_{ \mathbf q; s=q_i} 
  	 \prod_{j=1}^{i -1}a_{q_{j}q{_{j +1} }}^* \kappa\cdot( {\boldsymbol \lambda} +D)(a_{s,n+1}^* a_{s,s+1}^*) \\
&\quad  +(1-\delta_{s,n})   \sum_{ \mathbf q; s+1=q_i} 
   	\prod_{j=1}^{i -1}a_{q_{j}q{_{j +1} }}^*\kappa\cdot( {\boldsymbol \lambda} +D)(a_{s+1,n+1}^*  a_{s,s+1}^*) \\ 
&\quad  - \sum_{\mathbf q;s+1= q_i;s\geq q_{i-1}} 
	\left(\prod_{l=1}^{i-1} a_{q_l q_{l+1}}^*\right) a_{s,n+1}^* \kappa\cdot {\boldsymbol \lambda}   \\
&\quad+  
        \sum_{\mathbf q; s-1 \geq  q_i } 
    \prod_{j=1}^{i} a_{q_{j}q{_{j +1} }}^* A_{s,s-1}\kappa\cdot({\boldsymbol \lambda}+D)a_{s,s+1}^*    \\
&\quad+2 
        \sum_{ \mathbf q; s \geq  q_i } 
    \prod_{j=1}^{i} a_{q_{j}q{_{j +1} }}^* \kappa\cdot({\boldsymbol \lambda}+D)a_{s,s+1}^*    \\
&\quad+ 
        \sum_{ s+1 \geq  q_i ,\mathbf q} 
    \prod_{j=1}^{i} a_{q_{j}q{_{j +1} }}^* A_{s,s+1}\kappa\cdot({\boldsymbol \lambda}+D)a_{s,s+1}^*   .
\end{align*} 
\begin{align*}
 & = - \sum_{\mathbf q; s+1=q_i }
   	 a_{1,q_2}^*\cdots a_{q_{i-2},s}^*(a_{s,s+1}^*)^2 \kappa \cdot Da_{s+1,n+1}^*   \\ 
&\quad -\sum_{s+2\leq r <n+1}\sum_{\mathbf q; r=q_i\atop \exists t:(q_t,q_{t+1})=(s,s+1)} 
   	 a_{1,q_2}^*\cdots a_{q_{i-2},s}^*a_{s,s+1}^*\cdots a_{q_{i-1}r}^*a_{s,s+1}^* \kappa \cdot Da_{r,n+1}^*   \\ 
& \quad  -\delta_{s,n} \sum_{\mathbf q; n=q_i} 
   	\prod_{j=1}^{i -1}a_{q_{j}q{_{j +1} }}^*\kappa\cdot (\mathbf {\boldsymbol \lambda} +D)(a_{n,n+1}^*)^2  \\
& \quad +\sum_{s+2\leq r <n+1} \sum_{k=s+2}^{r} \sum_{l=1}^{s-1}
	\sum_{\mathbf q; r=q_i\atop \exists t:(q_t,q_{t+1})=(s+1,k)} a_{1,q_2}^*\cdots a_{l,s+1}^*
	a_{k,q_{t+2}}^*\cdots a_{q_{i-1},r}^* a_{sk}^* \kappa \cdot Da_{r,n+1}^*   \\  
& \quad +\sum_{s+2\leq r <n+1} \sum_{k=s+2}^{r} 
	\sum_{\mathbf q; r=q_i\atop \exists t:(q_t,q_{t+1})=(s+1,k)} a_{1,q_2}^*\cdots a_{q_{t-2},s}^*
	a_{s,s+1}^*
	a_{k,q_{t+2}}^*\cdots a_{q_{i-1},r}^* a_{sk}^* \kappa \cdot Da_{r,n+1}^*   \\  
& \quad  +  (1-\delta_{s,n}) \sum_{\mathbf q; s+1=q_i} 
  	 \prod_{j=1}^{i -1}a_{q_{j}q{_{j +1} }}^*\kappa\cdot (   {\boldsymbol \lambda} +D)a_{s,n+1}^*  \\  
& \quad-
      \sum_{s+2\leq r \leq n+1}\sum_{k=1}^{s-1}\sum_{ \mathbf q; r=q_i,\atop \exists t\leq i-1\,:\,
       (q_t,q_{t+1})=(k,s)} 
  \left( \prod_{j=1,j\neq t}^{i -1}a_{q_{j}q{_{j +1} }}^*  \right)a^*_{k,s+1} \kappa \cdot Da_{r,n+1}^* \\
& \quad-
     \sum_{k=1}^{s-1}\sum_{ \mathbf q; s=q_i,\atop \exists t\leq i-1\,:\,
       (q_t,q_{t+1})=(k,s)} 
  \left( \prod_{j=1,j\neq t}^{i -1}a_{q_{j}q{_{j +1} }}^*  \right)a^*_{k,s+1} \kappa \cdot Da_{s,n+1}^* \\
& \quad-
      \sum_{k=1}^{s-1}\sum_{ \mathbf q; s+1=q_i,\atop \exists t\leq i-1\,:\,
       (q_t,q_{t+1})=(k,s)} 
  \left( \prod_{j=1,j\neq t}^{i -1}a_{q_{j}q{_{j +1} }}^*  \right)a^*_{k,s+1} \kappa \cdot Da_{s+1,n+1}^* \\
&\quad- \sum_{s+2\leq r <n+1} \sum_{k=s+2}^{r}
	\sum_{\mathbf q\ ; r=q_i, \exists t\leq i-1, (q_t,q_{t+1})=(s,k)} 
   	\prod_{j=1}^{i -1}a_{q_{j}q{_{j +1} }}^*a_{s,s+1}^* \kappa \cdot Da_{r,n+1}^*  \\ \\
&\quad+  \sum_{s+2\leq r <n+1}
	 \sum_{k=s+2}^{r}a_{1q_2}^*\cdots 
	a_{s+1,k}^*a_{k,q{_{t=2} }}^*\cdots 
	a_{q_{i-1},r}^*a_{s,s+1}^* \kappa \cdot Da_{r,n+1}^*  \\
&\quad -  (1-\delta_{s,n}) \sum_{ \mathbf q; s=q_i} 
  	 \prod_{j=1}^{i -1}a_{q_{j}q{_{j +1} }}^* a_{s,n+1}^*\kappa\cdot( {\boldsymbol \lambda} +D)(
   	a_{s,s+1}^*) \\
	& \quad -  (1-\delta_{s,n})  \sum_{ \mathbf q; s=q_i} 
  	 \prod_{j=1}^{i -1}a_{q_{j}q{_{j +1} }}^*a_{s,s+1}^* \kappa\cdot D(a_{s,n+1}^*)\\
&\quad  + (1-\delta_{s,n})  \sum_{ \mathbf q; s+1=q_i} 
   	\prod_{j=1}^{i -1}a_{q_{j}q{_{j +1} }}^*a_{s+1,n+1}^* \kappa\cdot( {\boldsymbol \lambda} +D) (a_{s,s+1}^*) \\
&\quad	+  (1-\delta_{s,n}) \sum_{ \mathbf q; s+1=q_i} 
   	\prod_{j=1}^{i -1}a_{q_{j}q{_{j +1} }}^*a_{s,s+1}^*\kappa\cdot D(a_{s+1,n+1}^*  ) \\ 
&\quad  - \sum_{\mathbf q;s+1= q_i;s\geq q_{i-1}} 
	\left(\prod_{l=1}^{i-1} a_{q_l q_{l+1}}^*\right) a_{s,n+1}^* \kappa\cdot {\boldsymbol \lambda}   +  
        \sum_{\mathbf q; s-1 \geq  q_i } 
    \prod_{j=1}^{i} a_{q_{j}q{_{j +1} }}^* A_{s,s-1}\kappa\cdot({\boldsymbol \lambda}+D)a_{s,s+1}^*    \\
&\quad+2 
        \sum_{ \mathbf q; s \geq  q_i } 
    \prod_{j=1}^{i} a_{q_{j}q{_{j +1} }}^* \kappa\cdot({\boldsymbol \lambda}+D)a_{s,s+1}^*   + 
        \sum_{ s+1 \geq  q_i ,\mathbf q} 
    \prod_{j=1}^{i} a_{q_{j}q{_{j +1} }}^* A_{s,s+1}\kappa\cdot({\boldsymbol \lambda}+D)a_{s,s+1}^*   .
\end{align*} 

We break the above into three summations; collecting terms with  $\kappa\cdot Da_{s+1,n+1}$,  with $\kappa\cdot Da_{s,n+1}$, with  $\kappa \cdot Da_{r,n+1}, r\geq s+2$ and over $\kappa\cdot Da_{s,s+1}$ (there is some overlap with $\kappa\cdot Da_{s,n+1}$ and $\kappa\cdot Da_{s,s+1}$ when $s=n$ ):
\begin{align*}
 - \sum_{\mathbf q\atop  s+1=q_i }&
   	 a_{1,q_2}^*\cdots a_{q_{i-2},s}^*(a_{s,s+1}^*)^2 \kappa \cdot Da_{s+1,n+1}^*  \\
	  &-
     \sum_{ \mathbf q\atop s+1=q_i} \sum_{k=1}^{s-1} a_{1q_2}^*\cdots a_{q_{t-1}k}^*a_{s,s+1}^* a^*_{k,s+1} \kappa \cdot Da_{s+1,n+1}^* \\
& +   \sum_{ \mathbf q,s+1=q_i}  a_{1q_2}^*\cdots a_{q_{i-1},s+1}^*a_{s,s+1}^*\kappa\cdot D(a_{s+1,n+1}^*  ) 
 =0 .
\end{align*} 
Next we have summands with $ (  {\boldsymbol \lambda} +D)a_{s,n+1}^*$ in them:
\begin{align*}
    (1&-\delta_{s,n})  \sum_{\mathbf q; s+1=q_i} 
  	 \prod_{j=1}^{i -1}a_{q_{j}q{_{j +1} }}^*\kappa\cdot ( {\boldsymbol \lambda} +D)a_{s,n+1}^*  \\
	&  -
     \sum_{k=1}^{s-1}\sum_{ \mathbf q; s=q_i,\atop \exists t\leq i-1\,:\,
       (q_t,q_{t+1})=(k,s)} 
 \prod_{j=1,j\neq t}^{i -1}a_{q_{j}q{_{j +1} }}^*  a^*_{k,s+1} \kappa \cdot Da_{s,n+1}^* \\
&  \ -  (1-\delta_{s,n})\sum_{ \mathbf q; s=q_i} 
  	 \prod_{j=1}^{i -1}a_{q_{j}q{_{j +1} }}^*a_{s,s+1}^* \kappa\cdot D(a_{s,n+1}^*)  -  \sum_{\stack{\mathbf q;s+1= q_i}{s\geq q_{i-1}}   }
	\left(\prod_{l=1}^{i-1} a_{q_l q_{l+1}}^*\right) a_{s,n+1}^* \kappa\cdot {\boldsymbol \lambda}  \\   \\
 & =   \sum_{\mathbf q \atop s+1=q_i} a_{1,q_2}^*\cdots a_{q_{i-1},s+1}^*
	\kappa\cdot (  {\boldsymbol \lambda} +D)a_{s,n+1}^*  -
     \sum_{k=1}^{s-1}\sum_{\stack{ \mathbf q; s=q_i, \exists t\leq i-1\,:\,}{
       (q_t,q_{t+1})=(k,s)}} a_{1,q_2}\cdots a_{q_{i-2}k}^* a^*_{k,s+1} \kappa \cdot Da_{s,n+1}^* \\
&\  -  \sum_{ \mathbf q; s=q_i} 
  	a_{1,q_2}^*\cdots a_{q_{i-1},s}^*a_{s,s+1}^* \kappa\cdot D(a_{s,n+1}^*)  - \sum_{\mathbf q;s+1= q_i} a_{1,q_2}^*\cdots a_{q_{i-1},s+1}^*
	 a_{s,n+1}^* \kappa\cdot {\boldsymbol \lambda} \\
& \   -\delta_{s,n} \left( \sum_{\mathbf q; s+1=q_i} 
  	 \prod_{j=1}^{i -1}a_{q_{j}q{_{j +1} }}^*\kappa\cdot ( {\boldsymbol \lambda} +D)a_{s,n+1}^*-\sum_{\stack{\mathbf q;s= q_i}{s\geq q_{i-1}}   }
	\left(\prod_{l=1}^{i-1} a_{q_l q_{l+1}}^*\right) a_{s,s+1}^* \kappa\cdot D  (a_{s,n+1}^*)\right)  \\ \\
& =     -\delta_{s,n} \left( \sum_{\mathbf q; n+1=q_i} 
  	 \prod_{j=1}^{i -1}a_{q_{j}q{_{j +1} }}^*\kappa\cdot (  {\boldsymbol \lambda} +D)a_{n,n+1}^*-\sum_{\mathbf q;n= q_i}   
	\left(\prod_{l=1}^{i-1} a_{q_l q_{l+1}}^*\right) a_{n,n+1}^* \kappa\cdot D  (a_{n,n+1}^*)\right).
\end{align*} 

Moreover the summands with $( {\boldsymbol \lambda} +D) a_{s,s+1}^*$ sum to 
\begin{align*}
  - & \delta_{s,n}  \sum_{\mathbf q; n=q_i} 
   	\prod_{j=1}^{i -1}a_{q_{j}q{_{j +1} }}^*\kappa\cdot ( {\boldsymbol \lambda} +D)(a_{n,n+1}^*)^2  
	 -(1-\delta_{s,n})\sum_{ \stack{\mathbf q;}{s=q_i}} a_{1,q_2}^*\cdots a_{q_{i-1},s}^*  
	a_{s,n+1}^*\kappa\cdot( {\boldsymbol \lambda} +D)(a_{s,s+1}^*) \\
&\   +(1-\delta_{s,n})  \sum_{ \stack{\mathbf q;}{ s+1=q_i} }
   	a_{1,q_2}^*\cdots a_{q_{i-1},s+1}^* a_{s+1,n+1}^* \kappa\cdot( {\boldsymbol \lambda} +D) (a_{s,s+1}^*)   \\
	&\ - 
        \sum_{\stack{\mathbf q;}{ s-1 \geq  q_i } }a_{1,q_2}^*\cdots a_{q_{i},n+1}^* \kappa\cdot({\boldsymbol \lambda}+D)(a_{s,s+1}^*)    \\
&\ +2 
        \sum_{\stack{ \mathbf q; }{s \geq  q_i } } a_{1,q_2}^*\cdots a_{q_{i},n+1}^* 
     \kappa\cdot({\boldsymbol \lambda}+D)(a_{s,s+1}^*)    +(\delta_{s,n}-1)
        \sum_{ s+1 \geq  q_i ,\mathbf q} 
    a_{1,q_2}^*\cdots a_{q_{i},n+1}^*  \kappa\cdot({\boldsymbol \lambda}+D)(a_{s,s+1}^*)  \\   
 =&  -  \sum_{ \mathbf q\atop s=q_i} a_{1,q_2}^*\cdots a_{q_{i-1},s}^*  
 	a_{s,n+1}^*\kappa\cdot( {\boldsymbol \lambda} +D)(a_{s,s+1}^*)   +  \sum_{ \mathbf q\atop s+1=q_i} 
   	a_{1,q_2}^*\cdots a_{q_{i-1},s+1}^* a_{s+1,n+1}^* \kappa\cdot( {\boldsymbol \lambda} +D) (a_{s,s+1}^*)  \\ 
&\  -  
        \sum_{\mathbf q\atop s-1 \geq  q_i } a_{1,q_2}^*\cdots a_{q_{i},n+1}^* \kappa\cdot({\boldsymbol \lambda}+D)(a_{s,s+1}^*) +
        \sum_{ \mathbf q; s \geq  q_i } a_{1,q_2}^*\cdots a_{q_{i},n+1}^* 
     \kappa\cdot({\boldsymbol \lambda}+D)(a_{s,s+1}^*)    \\
&  \ - 
        \sum_{ \mathbf q; s+1=  q_i } 
    a_{1,q_2}^*\cdots a_{q_{i-1},s+1}^*a_{s+1,n+1}  \kappa\cdot({\boldsymbol \lambda}+D)(a_{s,s+1}^*)    \\
& \ -\delta_{s,n} \sum_{\mathbf q; n=q_i} 
   	\prod_{j=1}^{i -1}a_{q_{j}q{_{j +1} }}^*\kappa\cdot (  {\boldsymbol \lambda} +D)(a_{n,n+1}^*)^2  
	 +\delta_{s,n}\sum_{ \mathbf q; s=q_i} a_{1,q_2}^*\cdots a_{q_{i-1},s}^*  
	a_{s,n+1}^*\kappa\cdot( {\boldsymbol \lambda} +D)(a_{s,s+1}^*) \\
&\    -\delta_{s,n}  \sum_{ \mathbf q; s+1=q_i} 
   	a_{1,q_2}^*\cdots a_{q_{i-1},s+1}^* a_{s+1,n+1}^* \kappa\cdot( {\boldsymbol \lambda} +D) (a_{s,s+1}^*)    \\
	& \quad +\delta_{s,n} 
        \sum_{ \stack{ \mathbf q}{s+1 \geq  q_i}} 
    a_{1,q_2}^*\cdots a_{q_{i},n+1}^*  \kappa\cdot({\boldsymbol \lambda}+D)(a_{s,s+1}^*)  \\    
 =& -\delta_{s,n} \sum_{\mathbf q; n=q_i} 
   	\prod_{j=1}^{i -1}a_{q_{j}q{_{j +1} }}^*a_{n,n+1}^*\kappa\cdot (  {\boldsymbol \lambda} +D)(a_{n,n+1}^*)
	 -\delta_{s,n} \sum_{\mathbf q; n=q_i} 
   	\prod_{j=1}^{i -1}a_{q_{j}q{_{j +1} }}^*a_{n,n+1}^*\kappa\cdot D(a_{n,n+1}^*)    \\
&\ 
	 +\delta_{s,n}\sum_{ \mathbf q; n=q_i} a_{1,q_2}^*\cdots a_{q_{i-1},n}^*  
	a_{n,n+1}^*\kappa\cdot( {\boldsymbol \lambda} +D)(a_{n,n+1}^*) \\
&\     +\delta_{s,n} 
        \sum_{ \stack{\mathbf q}{n+1 \geq  q_i }}
    a_{1,q_2}^*\cdots a_{q_{i},n+1}^*  \kappa\cdot({\boldsymbol \lambda}+D)(a_{n,n+1}^*)  \\    \\
&= 	 -\delta_{s,n} \sum_{\mathbf q; n=q_i} 
   	\prod_{j=1}^{i -1}a_{q_{j}q{_{j +1} }}^*a_{n,n+1}^*\kappa\cdot D(a_{n,n+1}^*)          +\delta_{s,n} 
        \sum_{ \stack{ \mathbf q}{n\geq  q_i }} 
    a_{1,q_2}^*\cdots a_{q_{i},n+1}^*  \kappa\cdot({\boldsymbol \lambda}+D)(a_{n,n+1}^*) .
\end{align*} 
Hence the summands with $({\boldsymbol \lambda}+D)a_{s,s+1}^*$ and $ ( {\boldsymbol \lambda} +D)a_{s,n+1}^*$, sum to zero.

 Next we have the summands containing $\kappa \cdot Da_{r,n+1}^* $, which contribute
 \begin{align*}
 & -\sum_{s+2\leq r <n+1}\sum_{\mathbf q; r=q_i\atop \exists t:(q_t,q_{t+1})=(s,s+1)} 
   	 a_{1,q_2}^*\cdots a_{q_{i-2},s}^*a_{s,s+1}^*\cdots a_{q_{i-1}r}^*
	 a_{s,s+1}^* \kappa \cdot Da_{r,n+1}^*   \\ 
& \quad +\sum_{s+2\leq r <n+1} \sum_{k=s+2}^{r} \sum_{l=1}^{s-1}
	\sum_{\mathbf q; r=q_i\atop \exists t:(q_t,q_{t+1})=(s+1,k)} a_{1,q_2}^*\cdots a_{l,s+1}^*
	a_{k,q_{t+2}}^*\cdots a_{q_{i-1},r}^* a_{sk}^* \kappa \cdot Da_{r,n+1}^*   \\  
& \quad +\sum_{s+2\leq r <n+1} \sum_{k=s+2}^{r} 
	\sum_{\mathbf q; r=q_i\atop \exists t:(q_t,q_{t+1})=(s+1,k)} a_{1,q_2}^*\cdots a_{q_{t-2},s}^*
	a_{s,s+1}^*
	a_{k,q_{t+2}}^*\cdots a_{q_{i-1},r}^* a_{sk}^* \kappa \cdot Da_{r,n+1}^*   \\  
& \quad-
      \sum_{s+2\leq r \leq n+1}\sum_{k=1}^{s-1}\sum_{ \mathbf q; r=q_i,\atop \exists t\leq i-1\,:\,
       (q_t,q_{t+1})=(k,s)} a_{1,q_2}^*\cdots a_{q_{t-1},k}^*
	a_{s,q_{t+2}}^*  \cdots a_{q_{i-1},r}^* a^*_{k,s+1} \kappa \cdot Da_{r,n+1}^* \\
&\quad- \sum_{s+2\leq r <n+1} \sum_{k=s+2}^{r}
	\sum_{\stack{\mathbf q\ ; r=q_i,}{ \exists t\leq i-1, (q_t,q_{t+1})=(s,k)}} 
	a_{1,q_2}^*\cdots a_{q_{t-1},s}^*a_{s,k}^*
	a_{k,q_{t+2}}^*  \cdots a_{q_{i-1},r}^* 
   	 a_{s,s+1}^* \kappa \cdot Da_{r,n+1}^*  \\ \\
&\quad+  \sum_{s+2\leq r <n+1}
	 \sum_{k=s+2}^{r}a_{1q_2}^*\cdots 
	a_{s+1,k}^*a_{k,q{_{t+2} }}^*\cdots 
	a_{q_{i-1},r}^*a_{s,s+1}^* \kappa \cdot Da_{r,n+1}^*  \\ \\
&= -\sum_{s+2\leq r <n+1}\sum_{\mathbf q; r=q_i\atop \exists t:(q_t,q_{t+1})=(s,s+1)} 
   	 a_{1,q_2}^*\cdots a_{q_{i-2},s}^*a_{s,s+1}^*\cdots a_{q_{i-1}r}^*
	 a_{s,s+1}^* \kappa \cdot Da_{r,n+1}^*   \\ 
& \quad +\sum_{s+2\leq r <n+1} \sum_{k=s+2}^{r} \sum_{l=1}^{s-1}
	\sum_{\mathbf q; r=q_i\atop \exists t:(q_t,q_{t+1})=(s+1,k)} a_{1,q_2}^*\cdots a_{l,s+1}^*
	a_{k,q_{t+2}}^*\cdots a_{q_{i-1},r}^* a_{sk}^* \kappa \cdot Da_{r,n+1}^*   \\  
& \quad-
      \sum_{s+2\leq r \leq n+1}\sum_{k=1}^{s-1}\sum_{ \mathbf q; r=q_i,\atop \exists t\leq i-1\,:\,
       (q_t,q_{t+1})=(k,s)} a_{1,q_2}^*\cdots a_{q_{t-1},k}^*
	a_{s,q_{t+2}}^*  \cdots a_{q_{i-1},r}^* a^*_{k,s+1} \kappa \cdot Da_{r,n+1}^* \\
&\quad+  \sum_{s+2\leq r <n+1}
	 \sum_{k=s+2}^{r}a_{1q_2}^*\cdots 
	a_{s+1,k}^*a_{k,q{_{t+2} }}^*\cdots 
	a_{q_{i-1},r}^*a_{s,s+1}^* \kappa \cdot Da_{r,n+1}^*  
\end{align*}
\begin{align*}
&= -\sum_{s+2\leq r <n+1}\sum_{\mathbf q; r=q_i\atop \exists t:(q_t,q_{t+1})=(s,s+1)} 
   	 a_{1,q_2}^*\cdots a_{q_{i-2},s}^*a_{s,s+1}^*\cdots a_{q_{i-1}r}^*
	 a_{s,s+1}^* \kappa \cdot Da_{r,n+1}^*   \\ 
& \quad +\sum_{s+2\leq r <n+1} \sum_{k=s+2}^{r} \sum_{l=1}^{s-1}
	\sum_{\mathbf q; r=q_i\atop \exists t:(q_t,q_{t+1})=(s+1,k)} a_{1,q_2}^*\cdots a_{l,s+1}^*
	a_{k,q_{t+2}}^*\cdots a_{q_{i-1},r}^* a_{sk}^* \kappa \cdot Da_{r,n+1}^*   \\  
& \quad-
      \sum_{s+2\leq r \leq n+1}\sum_{k=1}^{s-1}\sum_{ \mathbf q; r=q_i,\atop \exists t\leq i-1\,:\,
       (q_t,q_{t+1})=(k,s)} a_{1,q_2}^*\cdots a_{q_{t-1},k}^*
	a_{s,q_{t+2}}^*  \cdots a_{q_{i-1},r}^* a^*_{k,s+1} \kappa \cdot Da_{r,n+1}^* \\
&\quad+  \sum_{s+2\leq r <n+1}
	\sum_{l=1}^{s-1}a_{1q_2}^*\cdots a_{l,s+1}^*\cdots 
	a_{q_{i-1},r}^*a_{s,s+1}^* \kappa \cdot Da_{r,n+1}^*  \\ 
&\quad+  \sum_{s+2\leq r <n+1}
	\sum_{\mathbf q; r=q_i\atop \exists t:(q_t,q_{t+1})=(s,s+1)} a_{1q_2}^*\cdots a_{s,s+1}^*
	\cdots 
	a_{q_{i-1},r}^*a_{s,s+1}^* \kappa \cdot Da_{r,n+1}^*  \\ \\
&= \sum_{s+2\leq r <n+1} \sum_{k=s+2}^{r} \sum_{l=1}^{s-1}
	\sum_{\mathbf q; r=q_i\atop \exists t:(q_t,q_{t+1})=(s+1,k)} a_{1,q_2}^*\cdots a_{l,s+1}^*
	a_{sk}^*a_{k,q_{t+2}}^*\cdots a_{q_{i-1},r}^*  \kappa \cdot Da_{r,n+1}^*   \\  
& \quad-
      \sum_{s+2\leq r \leq n+1}\sum_{k=1}^{s-1}\sum_{ \mathbf q; r=q_i,\atop \exists t\leq i-1\,:\,
       (q_t,q_{t+1})=(k,s)} a_{1,q_2}^*\cdots a_{q_{t-1},k}^*a^*_{k,s+1} 
	a_{s,q_{t+2}}^*  \cdots a_{q_{i-1},r}^* \kappa \cdot Da_{r,n+1}^* \\
&\quad+  \sum_{s+2\leq r <n+1}
	\sum_{l=1}^{s-1}a_{1q_2}^*\cdots a_{l,s+1}^*\cdots 
	a_{q_{i-1},r}^*a_{s,s+1}^* \kappa \cdot Da_{r, n+1}^*  \\  \\
&= \sum_{s+2\leq r <n+1} \sum_{k=s+2}^{r} \sum_{l=1}^{s-1}
	\sum_{\mathbf q; r=q_i\atop \exists t:(q_t,q_{t+1})=(s+1,k)} a_{1,q_2}^*\cdots a_{l,s+1}^*
	a_{sk}^*a_{k,q_{t+2}}^*\cdots a_{q_{i-1},r}^*  \kappa \cdot Da_{r,n+1}^*   \\  
& \quad-
      \sum_{s+2\leq r \leq n+1}\sum_{k=1}^{s-1}\sum_{l=s+2}^{r}\sum_{ \mathbf q; r=q_i,\atop \exists t\leq i-1\,:\,
       (q_t,q_{t+1})=(k,s)} a_{1,q_2}^*\cdots a_{q_{t-1},k}^*a^*_{k,s+1} 
	a_{s,l}^*a_{l,q_{t+3}}^*  \cdots a_{q_{i-1},r}^* \kappa \cdot Da_{r,n+1}^* \\
& \quad-
      \sum_{s+2\leq r \leq n+1}\sum_{k=1}^{s-1}\sum_{ \mathbf q; r=q_i,\atop \exists t\leq i-1\,:\,
       (q_t,q_{t+1})=(k,s)} a_{1,q_2}^*\cdots a_{q_{t-1},k}^*a^*_{k,s+1} 
	a_{s+1,q_{t+3}}^*  \cdots a_{q_{i-1},r}^* a_{s,s+1}^*\kappa \cdot Da_{r,n+1}^* \\
&\quad+  \sum_{s+2\leq r <n+1}
	\sum_{l=1}^{s-1}a_{1q_2}^*\cdots a_{l,s+1}^*\cdots 
	a_{q_{i-1},r}^*a_{s,s+1}^* \kappa \cdot Da_{r,n+1}^*  \\ 
&=0.
 \end{align*}
Now for the terms without any $\kappa\cdot D$ or $b_i$'s in them.
 First we consider the summands with $a_{k,j}$ terms with $k\leq s-1$:
\begin{align*}
&\quad \sum_{k=1}^{s-1}a_{ks}  \sum_{\mathbf q;s+1= q_i;k\geq q_{i-1}} 		
	 a_{1,q_2}^*\cdots a_{q_{i-1},s+1}^*   a_{k,n+1}^*\\ 
&\quad  -\sum_{k=1}^{s-1}\sum_{1\leq r<j  \leq n+1} a_{rj}  
	\sum_{\mathbf q;j= q_i;r\geq q_{i-1}\atop
	\exists t\leq i-1:(q_t,q_{t+1})=(k,s)}a_{1,q_2}^*\cdots a_{q_{t-1},k}^*a_{s,q_{t+2}}^*\cdots 
	a_{q_{i-1},j}^*   a_{k,s+1}^*a_{r,n+1}^* \\ 
&=\quad  \sum_{k=1}^{s-1}a_{ks}  \sum_{\mathbf q;s+1= q_i;k\geq q_{i-1}} 		
	 a_{1,q_2}^*\cdots a_{q_{i-1},s+1}^*   a_{k,n+1}^*\\ 
&\quad  -\sum_{k=1}^{s-1}\sum_{1\leq r\leq s-1} a_{rs}  
	\sum_{\mathbf q;s= q_i;r\geq q_{i-1}\atop
	\exists t\leq i-1:(q_t,q_{t+1})=(k,s)}a_{1,q_2}^*\cdots a_{q_{t-1},k}^*   a_{k,s+1}^*a_{r,n+1}^* \\ 
&\quad  -\sum_{k=1}^{s-1}\sum_{s\leq r<j  \leq n+1;j\geq s+1} a_{rj}  
	\sum_{\mathbf q;j= q_i;r\geq q_{i-1}\atop
	\exists t\leq i-1:(q_t,q_{t+1})=(k,s)}a_{1,q_2}^*\cdots a_{q_{t-1},k}^*a_{s,q_{t+2}}^*\cdots 
	a_{q_{i-1},j}^*   a_{k,s+1}^*a_{r,n+1}^* \\ 
&= -\sum_{k=1}^{s-1}\sum_{s\leq r<j  \leq n+1 } a_{rj}  
	\sum_{\mathbf q;j= q_i;r\geq q_{i-1}\atop
	\exists t\leq i-1:(q_t,q_{t+1})=(k,s)}a_{1,q_2}^*\cdots a_{q_{t-1},k}^*a_{s,q_{t+2}}^*\cdots 
	a_{q_{i-1},j}^*   a_{k,s+1}^*a_{r,n+1}^* \\ 
\end{align*}
This last summation splits up into three summations which contribute to later summations:  The first summation has a factor of $a_{sj}$, the second a factor of $a_{s+1,j}$ and the last summation has factors of the form $a_{rj}$ with $s+1<r<j$.   We deal with these consecutively in the following:

We consider the summands with a factor of $a_{sj}$:
\begin{align*}
&\quad -\delta_{s,n}a_{n,n+1}(a_{n,n+1}^*)^2   \sum_{\mathbf q}  
	a_{1,q_2}^*\cdots a_{q_{i-1},n+1}^*       \\  
&\quad -\sum_{s<j  \leq n+1} a_{sj}  
	\sum_{\mathbf q;j= q_i;s\geq q_{i-1}\atop \exists t: (q_t,q_{t+1})=(s,s+1)} 
	a_{1,q_2}\cdots a_{q_{t-1},s}^* a_{s,s+1}^*\cdots a_{q_{i-1},j}^*  a_{s,s+1}^*
    	a_{s,n+1}^*  \\ 
&\quad +2a_{s,s+1}a_{s,s+1}^*  \sum_{\mathbf q;s+1= q_i} a_{1,q_2}^*\cdots a_{q_{i-1},s+1}^*  
	a_{s,n+1}^*  \\
&\quad-\sum_{k=1}^{s-1}\sum_{s<j  \leq n+1 } a_{sj}  
	\sum_{\mathbf q;j= q_i;s\geq q_{i-1}\atop
	\exists t\leq i-1:(q_t,q_{t+1})=(k,s)}a_{1,q_2}^*\cdots a_{q_{t-1},k}^*a_{s,q_{t+2}}^*\cdots 
	a_{q_{i-1},j}^*   a_{k,s+1}^*a_{s,n+1}^*  \\ 
& \quad +\sum_{j=s+2}^{n+1}  a_{sj}    a_{s,s+1}^* 
        \sum_{\mathbf q;j= q_i;s\geq q_{i-1}} a_{1,q_2}^*\cdots 
	a_{q_{i-1},j}^*  a_{s,n+1}^*  \\ 
& \quad +\sum_{k=s+2}^{n+1}  
    a_{sk}a_{sk}^*
         \sum_{\mathbf q;s+1= q_i }a_{1,q_2}^*\cdots 
	a_{q_{i-1},s+1}^*a_{s,n+1}^*   \\
&\quad -\sum_{k=s+2}^{n+1} \sum_{s<j  \leq n+1} a_{sj} 
         \sum_{\mathbf q;j= q_i;s\geq q_{i-1}\atop \exists t;\enspace (q_t,q_{t+1})=(s,k)}  
         a_{1,q_2}^*\cdots a_{q_{t-1},s}^*a_{s,k}^*\cdots a_{q_{i-1},j}^*   a_{s,s+1}^* a_{s,n+1}^*  \\ 
&\quad  -   (1-\delta_{s,n}) \sum_{ s<j  \leq n+1} a_{sj}  
        \sum_{\mathbf q;j= q_i;s\geq q_{i-1}} a_{1,q_2}^*\cdots 
	a_{q_{i-1},j}^*a_{s,n+1}^*a_{s,s+1}^*  \\   \\
&= -a_{s,s+1}  
	\sum_{\mathbf q;s+1= q_i} 
	a_{1,q_2}\cdots a_{q_{i-2},s}^* a_{s,s+1}^*  a_{s,s+1}^*
    	a_{s,n+1}^*  \\ 
&\quad +a_{s,s+1}a_{s,s+1}^*  \sum_{\mathbf q;s+1= q_i} a_{1,q_2}^*\cdots a_{q_{i-1},s+1}^*  
	a_{s,n+1}^*  \\
&\quad-\sum_{k=1}^{s-1}\sum_{s<j  \leq n+1 } a_{sj}  
	\sum_{\mathbf q;j= q_i}a_{1,q_2}^*\cdots a_{q_{t-1},k}^*a_{s,j}^*   a_{k,s+1}^*a_{s,n+1}^*  \\ 
& \quad +\sum_{k=s+2}^{n+1}  
    a_{sk}a_{sk}^*
         \sum_{\mathbf q;s+1= q_i }a_{1,q_2}^*\cdots 
	a_{q_{i-1},s+1}^*a_{s,n+1}^*   \\
&\quad -\sum_{k=s+2}^{n+1} \sum_{s<j  \leq n+1} a_{sj} 
         \sum_{\mathbf q;j= q_i;s\geq q_{i-1}\atop \exists t;\enspace (q_t,q_{t+1})=(s,k)}  
         a_{1,q_2}^*\cdots a_{q_{t-1},s}^*a_{s,k}^*\cdots a_{q_{i-1},j}^*   a_{s,s+1}^* a_{s,n+1}^*  \\  \\
&=a_{s,s+1}a_{s,s+1}^*  \sum_{\mathbf q; q_{i}\leq s-1} a_{1,q_2}^*\cdots a_{q_{i-1},s+1}^*  
	a_{s,n+1}^* -\sum_{k=1}^{s-1}\sum_{s<j  \leq n+1 } a_{sj}  
	\sum_{\mathbf q;j= q_i}a_{1,q_2}^*\cdots a_{q_{t-1},k}^*a_{s,j}^*   a_{k,s+1}^*a_{s,n+1}^*  \\ 
& \quad +\sum_{k=s+2}^{n+1}  
    a_{sk}a_{sk}^*
         \sum_{\mathbf q;s+1= q_i }a_{1,q_2}^*\cdots 
	a_{q_{i-1},s+1}^*a_{s,n+1}^*    -\sum_{k=s+2}^{n+1}  a_{sk} 
         \sum_{\mathbf q;k= q_i}  
         a_{1,q_2}^*\cdots a_{q_{i-2},s}^*a_{s,k}^*    a_{s,s+1}^* a_{s,n+1}^*   
\end{align*}

\begin{align*}
=  & -\sum_{j=s+2 } ^{  n+1}\sum_{k=1}^{s-1}a_{sj}  a_{s,j}^* 
	\sum_{\mathbf q;j= q_i}a_{1,q_2}^*\cdots a_{q_{t-1},k}^*  a_{k,s+1}^*a_{s,n+1}^*   +\sum_{k=s+2}^{n+1}  
    a_{sk}a_{sk}^*
         \sum_{\mathbf q;s+1= q_i }a_{1,q_2}^*\cdots 
	a_{q_{i-1},s+1}^*a_{s,n+1}^*   \\
&\quad -\sum_{k=s+2}^{n+1}  a_{sk} a_{sk}^*
         \sum_{\mathbf q;k= q_i}  
         a_{1,q_2}^*\cdots a_{q_{i-2},s}^*    a_{s,s+1}^* a_{s,n+1}^*  \\ 
&= 0. 
\end{align*}
Next we calculate the terms with $a_{s+1,j}$ in them:
\begin{align*}
-&\sum_{s+1<j  \leq n+1} a_{s+1,j}  
	\sum_{\stack{\mathbf q;j= q_i;s= q_{i-1}}{ \exists t: (q_t,q_{t+1})=(s,s+1)} }
	a_{1,q_2}^*\cdots a_{q_{i-3},s}^* a_{s,s+1}^* a_{s+1,j}^*  a_{s,s+1}^*
    	a_{s+1,n+1}^*  \\ 
  & -\sum_{k=s+2}^{n+1}a_{s+1,k}\sum_{\stack{\mathbf q: k=q_i}{ s\geq q_{i-1}}} a_{1,q_2}^*\cdots 		a_{q_{i-1},k}^*   a_{s,n+1}^*\\ 
  & +\sum_{k=s+2}^{n+1} a_{s+1,k}
	\sum_{\stack{\mathbf q: k=q_i;s+1=q_{i-1}}{\exists t:(q_t,q_{t+1})=(s+1,k)}}
	a_{1,q_2}^*\cdots a_{q_{t-1},s+1}^*   
	a_{s,k}^* a_{s+1,n+1}^*\\ 
 & + \sum_{ s+1<j\leq n+1}a_{s+1,j}\sum_{\mathbf q: j=q_i\atop s+1\geq q_{i-1}}
	a_{1,q_2}^*\cdots a_{q_{i-1},j}^*  a_{s,n+1}^*  \\
  & -\sum_{k=1}^{s-1}\sum_{s+1<j  \leq n+1} a_{s+1,j}  
	\sum_{\mathbf q;j= q_i;s+1\geq q_{i-1}\atop
	\exists t\leq i-1:(q_t,q_{t+1})=(k,s)}a_{1,q_2}^*\cdots a_{q_{t-1},k}^*a_{s,q_{t+2}}^*\cdots 
	a_{q_{i-1},j}^*   a_{k,s+1}^*a_{s+1,n+1}^* \\ 
  &-\sum_{j=s+2}^{n+1}   a_{s+1,j}    a_{s,s+1}^* 
        \sum_{\mathbf q;j= q_i;s+1\geq q_{i-1}}a_{1,q_2}^*\cdots 
	a_{q_{i-1},j}^* a_{s+1,n+1}^*  \\ 
 & -\sum_{k=s+2}^{n+1}          a_{s+1,k}a_{s+1,k}^* 
         \sum_{\mathbf q;s+1= q_i }a_{1,q_2}^*\cdots 
	a_{q_{i-1},s+1}^*a_{s,n+1}^*   \\
& -\sum_{k=s+2}^{n+1} \sum_{s+1<j  \leq n+1} a_{s+1,j} 
         \sum_{\mathbf q;j= q_i;s+1\geq q_{i-1}\atop \exists t;\enspace (q_t,q_{t+1})=(s,k)}  
         a_{1,q_2}^*\cdots a_{q_{t-1},s}^*a_{s,k}^*\cdots a_{q_{i-1},j}^*   a_{s,s+1}^* a_{s+1,n+1}^*  \\ 
  & +\sum_{k=s+2}^{n+1} a_{s+1,k} 
         \sum_{\mathbf q;k= q_i;s+1\geq q_{i-1}\atop \exists t;\enspace (q_t,q_{t+1})=(s+1,k)} 
           a_{1,q_2}^*\cdots a_{q_{t-1},s+1}^*a_{s+1,k}^*  
        a_{s,s+1}^* a_{s+1,n+1}^*  \\ 
  &+   \sum_{s+1<j  \leq n+1} a_{s+1,j}  
        \sum_{\mathbf q;j= q_i;s+1\geq q_{i-1}}a_{1,q_2}^*\cdots 
	a_{q_{i-1},j}^* a_{s+1,n+1}^* a_{s,s+1}^*\\
\end{align*}
\begin{align*}
 =& -\sum_{k=s+2}^{n+1} a_{s+1,k}  
	\sum_{\mathbf q;k= q_i;s= q_{i-1}\atop \exists t: (q_t,q_{t+1})=(s,s+1)} 
	a_{1,q_2}^*\cdots a_{q_{i-3},s}^* a_{s,s+1}^* a_{s+1,k}^*  a_{s,s+1}^*
    	a_{s+1,n+1}^*  \\ 
  &+\sum_{k=s+2}^{n+1} a_{s+1,k} 
         \sum_{\mathbf q;k= q_i;s+1\geq q_{i-1}\atop \exists t;\enspace (q_t,q_{t+1})=(s+1,k)} 
           a_{1,q_2}^*\cdots a_{q_{t-1},s+1}^*a_{s+1,k}^*  
        a_{s,s+1}^* a_{s+1,n+1}^*  \\ 
   &-\sum_{k=s+2}^{n+1}a_{s+1,k}\sum_{\mathbf q: k=q_i\atop s\geq q_{i-1}}a_{1,q_2}^*\cdots 		a_{q_{i-1},k}^*   a_{s,n+1}^*\\ 
  &+ \sum_{k=s+2}^{n+1}a_{s+1,k}\sum_{\mathbf q: k=q_i\atop s+1\geq q_{i-1}}
	a_{1,q_2}^*\cdots a_{q_{i-1},k}^*  a_{s,n+1}^*  \\
   &+\sum_{k=s+2}^{n+1} a_{s+1,k}
	\sum_{\stack{\mathbf q: k=q_i;s+1=q_{i-1}}{\exists t:(q_t,q_{t+1})=(s+1,k)}}
	a_{1,q_2}^*\cdots a_{q_{t-1},s+1}^*   
	a_{s,k}^* a_{s+1,n+1}^*\\ 
  &-\sum_{k=s+2}^{n+1}          a_{s+1,k}a_{s+1,k}^* 
         \sum_{\mathbf q;s+1= q_i }a_{1,q_2}^*\cdots 
	a_{q_{i-1},s+1}^*a_{s,n+1}^*   \\
  &-\sum_{k=1}^{s-1}\sum_{s+1<j  \leq n+1} a_{s+1,j}  
	\sum_{\mathbf q;j= q_i;s+1\geq q_{i-1}\atop
	\exists t\leq i-1:(q_t,q_{t+1})=(k,s)}a_{1,q_2}^*\cdots a_{q_{t-1},k}^*a_{s,q_{t+2}}^*\cdots 
	a_{q_{i-1},j}^*   a_{k,s+1}^*a_{s+1,n+1}^* \\ 
  &-\sum_{k=s+2}^{n+1} \sum_{j=s+2}^{n+1}  a_{s+1,k} 
         \sum_{\mathbf q;k= q_i;s+1\geq q_{i-1}\atop \exists t;\enspace (q_t,q_{t+1})=(s,j)}  
         a_{1,q_2}^*\cdots a_{q_{t-1},s}^*a_{s,j}^*\cdots a_{q_{i-1},k}^*   a_{s,s+1}^* a_{s+1,n+1}^*   \\
         \end{align*}
         \begin{align*}
 = & \sum_{k=s+2}^{n+1} \sum_{l=1}^{s-1}a_{s+1,k} 
         \sum_{\stack{ \mathbf q;k= q_i}{s+1\geq q_{i-1}} }
           a_{1,q_2}^*\cdots a_{l,s+1}^*a_{s+1,k}^*  
        a_{s,s+1}^* a_{s+1,n+1}^*  \\ 
  &+ \sum_{k=s+2}^{n+1}a_{s+1,k}\sum_{\mathbf q: k=q_i }
	a_{1,q_2}^*\cdots a_{q_{i-2},s+1}^*  a_{s+1,k}^*a_{s,n+1}^*  \\
  &+\sum_{k=s+2}^{n+1} a_{s+1,k}
	\sum_{\stack{\mathbf q: k=q_i;s+1=q_{i-1}}{\exists t:(q_t,q_{t+1})=(s+1,k)}}
	a_{1,q_2}^*\cdots a_{q_{t-1},s+1}^*   
	a_{s,k}^* a_{s+1,n+1}^*\\ 
 & -\sum_{k=s+2}^{n+1}          a_{s+1,k}a_{s+1,k}^* 
         \sum_{\mathbf q;s+1= q_i }a_{1,q_2}^*\cdots 
	a_{q_{i-1},s+1}^*a_{s,n+1}^*   \\
  &- \sum_{k=s+2}^{n+1}  \sum_{j=1}^{s-1}a_{s+1,k}  
	\sum_{\mathbf q;k= q_i;s+1\geq q_{i-1}\atop
	\exists t\leq i-1:(q_t,q_{t+1})=(j,s)}a_{1,q_2}^*\cdots a_{q_{t-1},j}^*a_{s,q_{t+2}}^*\cdots 
	a_{q_{i-1},k}^*   a_{j,s+1}^*a_{s+1,n+1}^* \\ 
  &-\sum_{k=s+2}^{n+1} \sum_{j=s+2}^{n+1}  a_{s+1,k} 
         \sum_{\mathbf q;k= q_i;s+1\geq q_{i-1}\atop \exists t;\enspace (q_t,q_{t+1})=(s,j)}  
         a_{1,q_2}^*\cdots a_{q_{t-1},s}^*a_{s,j}^*\cdots a_{q_{i-1},k}^*   a_{s,s+1}^* a_{s+1,n+1}^*  \\
        \end{align*}
\begin{align*}
 =  &\sum_{k=s+2}^{n+1} \sum_{l=1}^{s-1}a_{s+1,k} 
         \sum_{\mathbf q;k= q_i;s+1\geq q_{i-1}} 
           a_{1,q_2}^*\cdots a_{l,s+1}^*a_{s+1,k}^*  
        a_{s,s+1}^* a_{s+1,n+1}^*  \\ 
  &+\sum_{k=s+2}^{n+1} a_{s+1,k}a_{s,k}^*
	\sum_{\mathbf q: k=q_i;s+1=q_{i-1},\exists t:(q_t,q_{t+1})=(s+1,k)}
	a_{1,q_2}^*\cdots a_{q_{t-1},s+1}^*   
	 a_{s+1,n+1}^*\\ 
  &-\sum_{k=s+2}^{n+1}   a_{s+1,k}    
        \sum_{\mathbf q;k= q_i;s+1\geq q_{i-1}}a_{1,q_2}^*\cdots 
	a_{q_{i-1},k}^* a_{s,s+1}^* a_{s+1,n+1}^*  \\ 
  &- \sum_{k=s+2}^{n+1}  \sum_{j=1}^{s-1}a_{s+1,k}  
	\sum_{\mathbf q;k= q_i;s+1\geq q_{i-1}\atop
	\exists t\leq i-1:(q_t,q_{t+1})=(j,s)}a_{1,q_2}^*\cdots a_{q_{t-1},j}^*a_{s,q_{t+2}}^*\cdots 
	a_{q_{i-1},k}^*   a_{j,s+1}^*a_{s+1,n+1}^* \\ 
 & -\sum_{k=s+2}^{n+1} \sum_{j=s+2}^{n+1}  a_{s+1,k} 
         \sum_{\mathbf q;k= q_i;s+1\geq q_{i-1}\atop \exists t;\enspace (q_t,q_{t+1})=(s,j)}  
         a_{1,q_2}^*\cdots a_{q_{t-1},s}^*a_{s,j}^*\cdots a_{q_{i-1},k}^*   a_{s,s+1}^* a_{s+1,n+1}^*  \\ 
   &+   \sum_{k=s+2}^{n+1}  a_{s+1,k}  
        \sum_{\mathbf q;k= q_i;s+1\geq q_{i-1}}a_{1,q_2}^*\cdots 
	a_{q_{i-1},k}^* a_{s+1,n+1}^* a_{s,s+1}^*  
	\end{align*}
	\begin{align*}
 = & \sum_{k=s+2}^{n+1} \sum_{l=1}^{s-1}a_{s+1,k} 
         \sum_{\mathbf q;k= q_i;s+1\geq q_{i-1}} 
           a_{1,q_2}^*\cdots a_{l,s+1}^*a_{s+1,k}^*  
        a_{s,s+1}^* a_{s+1,n+1}^*  \\ 
   &+\sum_{k=s+2}^{n+1} a_{s+1,k}a_{s,k}^*
	\sum_{\stack{\mathbf q: k=q_i;s+1=q_{i-1}}{\exists t:(q_t,q_{t+1})=(s+1,k)}}
	a_{1,q_2}^*\cdots a_{q_{t-1},s+1}^*   
	 a_{s+1,n+1}^*\\ 
   &- \sum_{k=s+2}^{n+1}  \sum_{j=1}^{s-1}a_{s+1,k}  
	\sum_{\mathbf q;k= q_i;s+1\geq q_{i-1}\atop
	\exists t\leq i-1:(q_t,q_{t+1})=(j,s)}a_{1,q_2}^*\cdots a_{q_{t-1},j}^*a_{s,q_{t+2}}^*\cdots 
	a_{q_{i-1},k}^*   a_{j,s+1}^*a_{s+1,n+1}^* \\ 
  &-\sum_{k=s+2}^{n+1} \sum_{j=s+2}^{n+1}  a_{s+1,k} 
         \sum_{\mathbf q;k= q_i;s+1\geq q_{i-1}\atop \exists t;\enspace (q_t,q_{t+1})=(s,j)}  
         a_{1,q_2}^*\cdots a_{q_{t-1},s}^*a_{s,j}^*\cdots a_{q_{i-1},k}^*   a_{s,s+1}^* a_{s+1,n+1}^*  \\
\end{align*}
\begin{align*}
 =  &\sum_{k=s+2}^{n+1} \sum_{l=1}^{s-1}a_{s+1,k} 
         \sum_{\stack{\mathbf q;k= q_i}{s+1\geq q_{i-1}} }
           a_{1,q_2}^*\cdots a_{l,s+1}^*a_{s+1,k}^*  \\
   &+\sum_{k=s+2}^{n+1} a_{s+1,k}a_{s,k}^*
	\sum_{\stack{\mathbf q: k=q_i;s+1=q_{i-1}}{\exists t:(q_t,q_{t+1})=(s+1,k)}}
	a_{1,q_2}^*\cdots a_{q_{t-1},s+1}^*   
	 a_{s+1,n+1}^*\\ 
  &- \sum_{k=s+2}^{n+1}  \sum_{j=1}^{s-1}a_{s+1,k}  
	\sum_{\mathbf q;k= q_i;s+1\geq q_{i-1}\atop
	\exists t\leq i-1:(q_t,q_{t+1})=(j,s)}a_{1,q_2}^*\cdots a_{q_{t-1},j}^*a_{s,q_{t+2}}^*\cdots 
	a_{q_{i-1},k}^*   a_{j,s+1}^*a_{s+1,n+1}^* \\ 
&  -\sum_{k=s+2}^{n+1}   a_{s+1,k} a_{s,k}^* 
         \sum_{\mathbf q;k= q_i;s+1\geq q_{i-1}\atop \exists t;\enspace (q_t,q_{t+1})=(s,j)}  
         a_{1,q_2}^*\cdots a_{q_{t-1},s}^*   a_{s,s+1}^* a_{s+1,n+1}^*  \\
%
 =  &\sum_{k=s+2}^{n+1} \sum_{l=1}^{s-1}a_{s+1,k} 
         \sum_{\stack{\mathbf q;k= q_i}{s+1\geq q_{i-1}} }
           a_{1,q_2}^*\cdots a_{l,s+1}^*a_{s+1,k}^*  
        a_{s,s+1}^* a_{s+1,n+1}^*  \\ 
  & +\sum_{k=s+2}^{n+1} \sum_{l=1}^{s-1}a_{s+1,k}a_{s,k}^*
	\sum_{\mathbf q: k=q_i;s+1=q_{i-1},\exists t:(q_t,q_{t+1})=(s+1,k)}
	a_{1,q_2}^*\cdots a_{l,s+1}^* a_{s+1,n+1}^*\\ 
 &  - \sum_{k=s+2}^{n+1}  \sum_{j=1}^{s-1}a_{s+1,k}  a_{s,k}^*
	\sum_{\mathbf q;k= q_i;s= q_{i-1}\atop
	\exists t\leq i-1:(q_t,q_{t+1})=(j,s)}a_{1,q_2}^*\cdots a_{q_{t-1},j}^*   a_{j,s+1}^*a_{s+1,n+1}^* \\ 
  & - \sum_{k=s+2}^{n+1}  \sum_{j=1}^{s-1}a_{s+1,k}  
	\sum_{\mathbf q;k= q_i;s+1=q_{i-1}}a_{1,q_2}^*\cdots a_{q_{t-1},j}^*a_{s,s+1}^* 
	a_{s+1,k}^*   a_{j,s+1}^*a_{s+1,n+1}^* \\ 
 =& 0.
\end{align*}
Finally we calculate the last summations with $a_{r,j}$ $r\geq s+2$:
\begin{align*}
&-\sum_{s+2\leq r<j  \leq n+1} a_{rj}  
	\sum_{\mathbf q;j= q_i;r\geq q_{i-1}\atop \exists t: (q_t,q_{t+1})=(s,s+1)} 
	a_{1,q_2}^*\cdots a_{q_{t-1},s}^* a_{s,s+1}^*\cdots a_{q_{i-1},j}^*  a_{s,s+1}^*
    	a_{r,n+1}^*  \\ 
&\quad  +\sum_{k=s+2}^{n+1}\sum_{s+2\leq r<j\leq n+1}a_{rj}
	\sum_{\mathbf q: j=q_i\atop r \geq q_{i-1},\exists t:(q_t,q_{t+1})=(s+1,k)}
	a_{1,q_2}^*\cdots a_{q_{t-1},s+1}^*a_{k,q_{t+2}}^*\cdots a_{q_{i-1},j}^*   
	a_{s,k}^* a_{r,n+1}^*\\ 
&\quad  -\sum_{k=1}^{s-1}\sum_{s+2\leq r<j  \leq n+1} a_{rj}  
	\sum_{\mathbf q;j= q_i;r\geq q_{i-1}\atop
	\exists t\leq i-1:(q_t,q_{t+1})=(k,s)}a_{1,q_2}^*\cdots a_{q_{t-1},k}^*a_{s,q_{t+2}}^*\cdots 
	a_{q_{i-1},j}^*   a_{k,s+1}^*a_{r,n+1}^* \\ 
&\quad -\sum_{k=s+2}^{n+1} \sum_{s+2\leq r<j  \leq n+1} a_{rj} 
         \sum_{\mathbf q;j= q_i;r\geq q_{i-1}\atop \exists t;\enspace (q_t,q_{t+1})=(s,k)}  
         a_{1,q_2}^*\cdots a_{q_{t-1},s}^*a_{s,k}^*\cdots a_{q_{i-1},j}^*   a_{s,s+1}^* a_{r,n+1}^*  \\ 
&\quad +\sum_{k=s+2}^{n+1} \sum_{s+2\leq r<j  \leq n+1} a_{rj} 
         \sum_{\mathbf q;j= q_i;r\geq q_{i-1}\atop \exists t;\enspace (q_t,q_{t+1})=(s+1,k)} 
           a_{1,q_2}^*\cdots a_{q_{t-1},s+1}^*a_{s+1,k}^*\cdots a_{q_{i-1},j}^*   
        a_{s,s+1}^* a_{r,n+1}^*  \\
&= \sum_{k=s+2}^{n+1}\sum_{s+2\leq r<j\leq n+1}a_{rj}
	\sum_{\mathbf q: j=q_i\atop r \geq q_{i-1},\exists t:(q_t,q_{t+1})=(s+1,k)}
	a_{1,q_2}^*\cdots a_{q_{t-1},s+1}^*a_{k,q_{t+2}}^*\cdots a_{q_{i-1},j}^*   
	a_{s,k}^* a_{r,n+1}^*\\ 
&\quad  -\sum_{k=1}^{s-1}\sum_{s+2\leq r<j  \leq n+1} a_{rj}  
	\sum_{\mathbf q;j= q_i;r\geq q_{i-1}\atop
	\exists t\leq i-1:(q_t,q_{t+1})=(k,s)}a_{1,q_2}^*\cdots a_{q_{t-1},k}^*a_{s,q_{t+2}}^*\cdots 
	a_{q_{i-1},j}^*   a_{k,s+1}^*a_{r,n+1}^* \\ 
&\quad - \sum_{s+2\leq r<j  \leq n+1} \sum_{k=s+2}^{n+1}a_{rj} 
         \sum_{\mathbf q;j= q_i;r\geq q_{i-1}\atop \exists t;\enspace q_t =s}  
         a_{1,q_2}^*\cdots a_{q_{t-1},s}^*a_{s,k}^*\cdots a_{q_{i-1},j}^*   a_{s,s+1}^* a_{r,n+1}^*  \\ 
&\quad - \sum_{s+2\leq r<j  \leq n+1} a_{rj} 
         \sum_{\mathbf q;j= q_i;r\geq q_{i-1}\atop \exists t;\enspace q_t =s}  
         a_{1,q_2}^*\cdots a_{q_{t-1},s}^*a_{s,s+1}^*\cdots a_{q_{i-1},j}^*   a_{s,s+1}^* a_{r,n+1}^*  \\ 
&\quad +  \sum_{s+2\leq r<j  \leq n+1} \sum_{l=1}^{s-1}a_{rj} 
         \sum_{\mathbf q;j= q_i;r\geq q_{i-1}\atop \exists t;\enspace  q_t =s+1} 
           a_{1,q_2}^*\cdots a_{l,s+1}^*a_{s+1,k}^*\cdots a_{q_{i-1},j}^*   
        a_{s,s+1}^* a_{r,n+1}^*  \\
&\quad +  \sum_{s+2\leq r<j  \leq n+1} a_{rj} 
         \sum_{\mathbf q;j= q_i;r\geq q_{i-1}\atop \exists t;\enspace  q_t =s+1} 
           a_{1,q_2}^*\cdots a_{s,s+1}^*a_{s+1,q_{t+2}}^*\cdots a_{q_{i-1},j}^*   
        a_{s,s+1}^* a_{r,n+1}^*
\end{align*}

\begin{align*}
&= \sum_{k=s+2}^{n+1}\sum_{s+2\leq r<j\leq n+1}a_{rj}
	\sum_{\mathbf q: j=q_i\atop r \geq q_{i-1},\exists t:(q_t,q_{t+1})=(s+1,k)}
	a_{1,q_2}^*\cdots a_{q_{t-1},s+1}^*a_{k,q_{t+2}}^*\cdots a_{q_{i-1},j}^*   
	a_{s,k}^* a_{r,n+1}^*\\ 
&\quad  -\sum_{k=1}^{s-1}\sum_{s+2\leq r<j  \leq n+1} a_{rj}  
	\sum_{\mathbf q;j= q_i;r\geq q_{i-1}\atop
	\exists t\leq i-1:(q_t,q_{t+1})=(k,s)}a_{1,q_2}^*\cdots a_{q_{t-1},k}^*a_{s,q_{t+2}}^*\cdots 
	a_{q_{i-1},j}^*   a_{k,s+1}^*a_{r,n+1}^* \\ 
&\quad - \sum_{s+2\leq r<j  \leq n+1} \sum_{k=s+2}^{n+1}a_{rj} 
         \sum_{\mathbf q;j= q_i;r\geq q_{i-1}\atop \exists t;\enspace q_t =s}  
         a_{1,q_2}^*\cdots a_{q_{t-1},s}^*a_{s,k}^*\cdots a_{q_{i-1},j}^*   a_{s,s+1}^* a_{r,n+1}^*  \\ 
&\quad +  \sum_{s+2\leq r<j  \leq n+1} \sum_{l=1}^{s-1}a_{rj} 
         \sum_{\mathbf q;j= q_i;r\geq q_{i-1}\atop \exists t;\enspace  q_t =s+1} 
           a_{1,q_2}^*\cdots a_{l,s+1}^*a_{s+1,q_{t+2}}^*\cdots a_{q_{i-1},j}^*   
        a_{s,s+1}^* a_{r,n+1}^*  \\
&= \sum_{k=s+2}^{n+1}\sum_{s+2\leq r<j\leq n+1}a_{rj}
	\sum_{\mathbf q: j=q_i\atop r \geq q_{i-1},\exists t:(q_t,q_{t+1})=(s+1,k)}
	a_{1,q_2}^*\cdots a_{q_{t-1},s+1}^*a_{k,q_{t+2}}^*\cdots a_{q_{i-1},j}^*   
	a_{s,k}^* a_{r,n+1}^*\\ 
&\quad  -\sum_{k=1}^{s-1}\sum_{l=s+2}^{n+1}\sum_{s+2\leq r<j  \leq n+1} a_{rj}  
	\sum_{\mathbf q;j= q_i;r\geq q_{i-1}\atop
	\exists t\leq i-1:(q_t,q_{t+1})=(k,s)}a_{1,q_2}^*\cdots a_{q_{t-1},k}^*a_{s,l}^*\cdots 
	a_{q_{i-1},j}^*   a_{k,s+1}^*a_{r,n+1}^* \\ 
&\quad  -\sum_{k=1}^{s-1}\sum_{s+2\leq r<j  \leq n+1} a_{rj}  
	\sum_{\mathbf q;j= q_i;r\geq q_{i-1}\atop
	\exists t\leq i-1:(q_t,q_{t+1})=(k,s)}a_{1,q_2}^*
		\cdots a_{q_{t-1},k}^* a_{k,s+1}^*a_{s+1,q_{t+3}}^*\cdots 
	a_{q_{i-1},j}^*  a_{s,s+1}^*a_{r,n+1}^* \\ 
&\quad - \sum_{s+2\leq r<j  \leq n+1} \sum_{k=s+2}^{n+1}a_{rj} 
         \sum_{\mathbf q;j= q_i;r\geq q_{i-1}\atop \exists t;\enspace q_t =s}  
         a_{1,q_2}^*\cdots a_{q_{t-1},s}^*a_{s,k}^*\cdots a_{q_{i-1},j}^*   a_{s,s+1}^* a_{r,n+1}^*  \\ 
&\quad +  \sum_{s+2\leq r<j  \leq n+1} \sum_{l=1}^{s-1}a_{rj} 
         \sum_{\mathbf q;j= q_i;r\geq q_{i-1}\atop \exists t;\enspace  q_t =s+1} 
           a_{1,q_2}^*\cdots a_{l,s+1}^*a_{s+1,q_{t+2}}^*\cdots a_{q_{i-1},j}^*   
        a_{s,s+1}^* a_{r,n+1}^*   
\end{align*}

\begin{align*}
&= \sum_{k=s+2}^{n+1}\sum_{l=1}^{s-1}\sum_{s+2\leq r<j\leq n+1}a_{rj}
	\sum_{\mathbf q: j=q_i\atop r \geq q_{i-1},\exists t:(q_t,q_{t+1})=(s+1,k)}
	a_{1,q_2}^*\cdots a_{q_{t-2},l}^*a_{l,s+1}^*a_{k,q_{t+2}}^*\cdots a_{q_{i-1},j}^*   
	a_{s,k}^* a_{r,n+1}^*\\ 
&\quad +\sum_{k=s+2}^{n+1}\sum_{s+2\leq r<j\leq n+1}a_{rj}
	\sum_{\mathbf q: j=q_i\atop r \geq q_{i-1},\exists t:(q_t,q_{t+1})=(s+1,k)}
	a_{1,q_2}^*\cdots a_{q_{t-2},s}a_{s,s+1}^*a_{k,q_{t+2}}^*\cdots a_{q_{i-1},j}^*   
	a_{s,k}^* a_{r,n+1}^*\\ 
&\quad - \sum_{s+2\leq r<j  \leq n+1} \sum_{k=s+2}^{n+1}a_{rj} 
         \sum_{\mathbf q;j= q_i;r\geq q_{i-1}\atop \exists t;\enspace q_t =s}  
         a_{1,q_2}^*\cdots a_{q_{t-1},s}^*a_{s,k}^*\cdots a_{q_{i-1},j}^*   a_{s,s+1}^* a_{r,n+1}^*  \\ 
&\quad  -\sum_{k=1}^{s-1}\sum_{l=s+2}^{n+1}\sum_{s+2\leq r<j  \leq n+1} a_{rj}  
	\sum_{\mathbf q;j= q_i;r\geq q_{i-1}\atop
	\exists t\leq i-1:(q_t,q_{t+1})=(k,s)}a_{1,q_2}^*\cdots a_{q_{t-1},k}^*a_{s,l}^*\cdots 
	a_{q_{i-1},j}^*   a_{k,s+1}^*a_{r,n+1}^* \\ 
&\quad  -\sum_{k=1}^{s-1}\sum_{s+2\leq r<j  \leq n+1} a_{rj}  
	\sum_{\mathbf q;j= q_i;r\geq q_{i-1}\atop
	\exists t\leq i-1:(q_t,q_{t+1})=(k,s)}a_{1,q_2}^*
		\cdots a_{q_{t-1},k}^* a_{k,s+1}^*a_{s+1,q_{t+3}}^*\cdots 
	a_{q_{i-1},j}^*  a_{s,s+1}^*a_{r,n+1}^* \\ 
&\quad +  \sum_{s+2\leq r<j  \leq n+1} \sum_{l=1}^{s-1}a_{rj} 
         \sum_{\mathbf q;j= q_i;r\geq q_{i-1}\atop \exists t;\enspace  q_t =s+1} 
           a_{1,q_2}^*\cdots a_{l,s+1}^*a_{s+1,q_{t+2}}^*\cdots a_{q_{i-1},j}^*   
        a_{s,s+1}^* a_{r,n+1}^*  \\  \\
&= \sum_{k=s+2}^{n+1}\sum_{l=1}^{s-1}\sum_{s+2\leq r<j\leq n+1}a_{rj}
	\sum_{\mathbf q: j=q_i\atop r \geq q_{i-1},\exists t:(q_t,q_{t+1})=(s+1,k)}
	a_{1,q_2}^*\cdots a_{q_{t-2},l}^*a_{l,s+1}^*a_{k,q_{t+2}}^*\cdots a_{q_{i-1},j}^*   
	a_{s,k}^* a_{r,n+1}^*\\ 
&\quad  -\sum_{k=1}^{s-1}\sum_{l=s+2}^{n+1}\sum_{s+2\leq r<j  \leq n+1} a_{rj}  
	\sum_{\mathbf q;j= q_i;r\geq q_{i-1}\atop
	\exists t\leq i-1:(q_t,q_{t+1})=(k,s)}a_{1,q_2}^*\cdots a_{q_{t-1},k}^*a_{s,l}^*\cdots 
	a_{q_{i-1},j}^*   a_{k,s+1}^*a_{r,n+1}^* \\ 
&= \sum_{k=s+2}^{n+1}\sum_{l=1}^{s-1}\sum_{s+2\leq r<j\leq n+1}a_{rj}
	\sum_{\mathbf q: j=q_i\atop r \geq q_{i-1},\exists t:(q_t,q_{t+1})=(s+1,k)}
	a_{1,q_2}^*\cdots a_{q_{t-2},l}^*a_{s,k}^*a_{k,q_{t+2}}^*\cdots a_{q_{i-1},j}^*   
	 a_{l,s+1}^*a_{r,n+1}^*\\ 
&\quad  -\sum_{k=1}^{s-1}\sum_{l=s+2}^{n+1}\sum_{s+2\leq r<j  \leq n+1} a_{rj}  
	\sum_{\mathbf q;j= q_i;r\geq q_{i-1}\atop
	\exists t\leq i-1:(q_t,q_{t+1})=(k,s)}a_{1,q_2}^*\cdots a_{q_{t-1},k}^*a_{s,l}^*\cdots 
	a_{q_{i-1},j}^*   a_{k,s+1}^*a_{r,n+1}^* \\ 
	&=0.
\end{align*}
\end{proof}

We are now left with the Serre relations:

\begin{lem}[T4]

\begin{align*} 
[\rho(F_r) (\mathbf w)_{\boldsymbol \lambda} \rho(F_s) (\mathbf w)]&=[\rho(E_r) (\mathbf w)_{\boldsymbol \lambda} \rho(E_s) (\mathbf w)]=0
  \ \text{if}\ 
    A_{rs}\neq -1 \\
 [\rho(F_r) (\mathbf w)_{\boldsymbol \lambda}[\rho(F_r) (\mathbf w)_{\boldsymbol \mu}\rho(F_s) (\mathbf w)]]
   &=[\rho(E_r) (\mathbf w)_{\boldsymbol \lambda}[\rho(E_r) (\mathbf w)_{\boldsymbol \mu}\rho(E_s) (\mathbf w)]]=0
    \ \text{if}\  A_{rs}= -1
\end{align*} 
\end{lem}

\begin{proof}  As in the previous lemmas we first assume $rs\neq 0$.   In this case the proof is exactly the same as in \cite[Lemma 3.5]{MR2003g:17034} with the exception of a sign change in the formulation of the $\rho(F_r)$.
Let us check the relations for $\rho(F_r)$. 

Now suppose $r=0$, so that
for $s\neq 0$ we get 
\begin{align}
[\rho(F_0)&_{\boldsymbol \lambda} \rho(F_s)]\label{f0fs}\\
&=\Big[ \big(- \sum_{1\leq r<j  \leq n+1} a_{rj}  \sum_{\mathbf q;j= q_i,r\geq q_{i-1}} \prod_{l=1}^{i-1} 
	a_{q_l q_{l+1}}^* a_{r,n+1}^*   -\sum_{1\leq r <n+1}\sum_{\mathbf q;\,r\geq q_i } 
    	\prod_{j=1}^i a_{q_{j}q{_{j +1} }}^*   \Phi b_r   \notag \\ 
&\quad -\sum_{1\leq r <n+1}\sum_{\mathbf q;\, r= q_i } 
   \prod_{j=1}^{i -1}a_{q_{j}q{_{j +1} }}^* \kappa\cdot Da_{r,n+1}^* \big) _{\boldsymbol \lambda} 
   \big(
  a_{s,s+1}-\sum_{j=1}^{s-1}a_{j,s+1}a_{js}^*\big)\Big]  \notag\\ \notag\\
&=\Big[ \big(-\sum_{1\leq r<j  \leq n+1} a_{rj}  \sum_{\mathbf q;j= q_i,r\geq q_{i-1}}  \prod_{l=1}^{i-1} a_{q_l q_{l+1}}^* a_{r,n+1}^*   \big) {_{\boldsymbol \lambda}}(
 a_{s,s+1})\Big] \notag\\
&\quad -\Big[\big(\sum_{1\leq r <n+1}\sum_{\mathbf q;\, r\geq q_i } 
    \prod_{j=1}^i a_{q_{j}q{_{j +1} }}^*   \Phi b_r \big){_{\boldsymbol \lambda} }(
 a_{s,s+1} )\Big]  \notag\\ 
&\quad -\big(\sum_{1\leq r <n+1}\sum_{\mathbf q;\, r= q_i } 
   \prod_{j=1}^{i -1}a_{q_{j}q{_{j +1} }}^* \kappa\cdot Da_{r,n+1}^* \big) _{\boldsymbol \lambda} 
   \big(
 a_{s,s+1} \big)\Big]  \notag\\ \notag\\
 &\quad-\Big[ \big( -\sum_{1\leq r<j  \leq n+1} a_{rj}  \sum_{j= q_i\mathbf q} \prod_{l=1}^{i-1} a_{q_l q_{l+1}}^* a_{r,n+1}^*\big)_{\boldsymbol \lambda} 
   \big( \sum_{j=1}^{s-1}a_{j,s+1}a_{js}^*\big)\Big] \notag\\
&\quad +\Big[\big(\sum_{1 \leq r <n+1}\sum_{\mathbf q\atop r\geq q_i} 
    \prod_{j=1}^i a_{q_{j}q{_{j +1} }}^*   \Phi b_r  )_{\boldsymbol \lambda} 
   \big( \sum_{j=1}^{s-1}a_{j,s+1}a_{js}^*\big)\Big] \notag  \\ 
&\quad +\Big[\big(\sum_{1\leq r <n+1}\sum_{\mathbf q,r=q_i} 
   \prod_{j=1}^{i -1}a_{q_{j}q{_{j +1} }}^* \kappa\cdot Da_{r,n+1}^* \big) _{\boldsymbol \lambda} 
   \big( \sum_{j=1}^{s-1}a_{j,s+1}a_{js}^*\big)\Big] \notag
\end{align} 
(In the above we use the notation $\mathbf q=(q_1,q_2,\dots, q_i)$ where we sum over $1=q_1<\cdots <q_i<q_{i+1}=n+1$.  We also some times let $l(\mathbf q)$ denote the second to th last index in the multi-indexed element $\mathbf q$ if its last index is $n+1$, otherwise we let $l(\mathbf q)$ denote the last index.  For example if $n=5$ and $\mathbf q=(1,3,5)$ then $l(\mathbf q)=2$ and if $\mathbf q=(1,2,4)$, then $l(\mathbf q)=3$. ) Now we break this up into the following calculations simplifying the six summations above:   The first summand becomes

\begin{align*}
&-\Big[   \sum_{1\leq r<j  \leq n+1} a_{rj}  \sum_{j= q_i;r\geq q_{i-1}\mathbf q} \prod_{l=1}^{i-1} a_{q_l q_{l+1}}^* a_{r,n+1}^*  {_{\boldsymbol \lambda}}(
 a_{s,s+1})\Big] \\
&=   \delta_{s,n}a_{n,n+1}   \sum_{\mathbf q\atop n+1= q_i  } \prod_{k=1}^{i-1} a_{q_k q_{k+1}}^*   
 - \sum_{s\leq r< j \leq n+1} a_{rj}   \sum_{\mathbf q ;\, j=q_i;r\geq q_{i-1}} \prod_{k=1}^{i-1} \Big[ a_{q_k q_{k+1}}^*   {_{\boldsymbol \lambda}}
  a_{s,s+1}\Big]a_{r,n+1}^* \\
&=    \delta_{s,n}a_{n,n+1}   \sum_{\mathbf q\atop n+1= q_i  } \prod_{k=1}^{i-1} a_{q_k q_{k+1}}^*  
 + \sum_{s\leq r< j \leq n+1}  \sum_{\mathbf q ;\, j=q_i;r\geq q_{i-1} \atop \exists 1\leq k\leq i-1: \,(q_k,q_{k+1})=(s,s+1)}  a_{rj}  \prod_{j=1,j\neq k}^{i-1} a_{q_j q_{j+1}}^*   a_{r,n+1}^*  
\end{align*} 
The second and third summands in \eqnref{f0fs} simplify to 
\begin{align*}
\sum_{r\leq n}\Big[&\big(\sum_{\mathbf q \atop  q_i\leq r } 
    \prod_{j=1}^i a_{q_{j}q{_{j +1} }}^*   \Phi b_r \big){_{\boldsymbol \lambda} }(
 a_{s,s+1} )\Big]  =-\sum_{r\leq n}\sum_{\mathbf q \atop r\geq q_i \exists 1\leq k\leq i: \,(q_k,q_{k+1})=(s,s+1)} 
    \prod_{j=1,j\neq k}^i a_{q_{j}q{_{j +1} }}^*   \Phi b_r 
\end{align*}
and 
\begin{align*}
\sum_{1\leq r\leq n}\Big[\big(&\sum_{\mathbf q  ;q_i=r} 
   \prod_{j=1}^{i -1}a_{q_{j}q{_{j +1} }}^* \kappa\cdot Da_{r,n+1}^* \big) _{\boldsymbol \lambda} 
   \big(
 a_{s,s+1} \big)\Big]   \\
 &=-(1-\delta_{s,n})\sum_{1\leq r\leq n}\sum_{\mathbf q \atop   \exists 1\leq k\leq i-1: \,(q_k,q_{k+1})=(s,s+1);q_i=r} 
   \prod_{j=1,j\neq k}^{i -1}a_{q_{j}q{_{j +1} }}^* \kappa\cdot Da_{r,n+1}^* \\
 &\quad +\delta_{s,n}\sum_{\mathbf q ;q_i=n} \kappa\cdot (D+{\boldsymbol \lambda})
   \prod_{j=1}^{i -1}a_{q_{j}q{_{j +1} }}^*  .
\end{align*}
The last three summands in \eqnref{f0fs} are
\begin{align*}
 &-\Big[ \big( -\sum_{1\leq r<j  \leq n+1} a_{rj}  \sum_{\mathbf q;j= q_i;r\geq q_{i-1}} \prod_{l=1}^{i-1} a_{q_l q_{l+1}}^* a_{r,n+1}^*  \big) )_{\boldsymbol \lambda} 
   \big( \sum_{l=1}^{s-1}a_{l,s+1}a_{ls}^*\big)\Big] \\
 &= \sum_{l=1}^{s-1} \sum_{1\leq r<j  \leq n+1} a_{rj}  \sum_{\mathbf q;j= q_i;r\geq q_{i-1}} \prod_{l=1}^{i-1} a_{q_l q_{l+1}}^* \Big[a_{r,n+1}^*  {_{\boldsymbol \lambda}} 
    a_{l,s+1}\Big]a_{ls}^*  \\
 &\quad +\sum_{l=1}^{s-1} \sum_{1\leq r<j  \leq n+1}\Big[ a_{rj} {_{\boldsymbol \lambda}} a_{ls}^* \Big]   a_{l,s+1} \sum_{j= q_i\mathbf q} \prod_{l=1}^{i-1} a_{q_l q_{l+1}}^* a_{r,n+1}^* 
    \\
&\quad+\sum_{l=1}^{s-1} \sum_{1\leq r<j  \leq n+1} a_{rj}  \sum_{\mathbf q;j= q_i;r\geq q_{i-1}} \prod_{l=1}^{i-1} \Big[a_{q_l q_{l+1}}^*   {_{\boldsymbol \lambda}} 
    a_{l,s+1}\Big]a_{ls}^* a_{r,n+1}^*\\
&= -\delta_{s,n} \sum_{l=1}^{s-1}\sum_{l<j\leq n+1}a_{lj}  \sum_{\mathbf q;j= q_i;l\geq q_{i-1}} \prod_{l=1}^{i-1} a_{q_l q_{l+1}}^*a_{ls}^* \\
 &\quad +\sum_{l=1}^{s-1} a_{l,s+1} \sum_{\mathbf q;s= q_i;l\geq q_{i-1}} \prod_{l=1}^{i-1} a_{q_l q_{l+1}}^* a_{l,n+1}^* 
    \\
 &\quad -\sum_{l=1}^{s-1} \sum_{r< j } a_{rj}   \sum_{\mathbf q; j=q_i; \,r\geq q_{i-1}  \atop \exists a,(q_a,q_{a+1})=(l,s+1) } \left(\prod_{k=1,k\neq a}^i a_{q_k q_{k+1}}^* \right),
  a_{ls}^* a_{r,n+1}^*,
\end{align*}

\begin{align*}
  \sum_{ 1\leq  r<n+1}&\sum_{\mathbf q \atop r\geq q_i }\sum_{l=1}^{s-1} \Big[
    \prod_{j=1}^i a_{q_{j}q{_{j +1} }}^*   \Phi b_r { _{\boldsymbol \lambda} }
  a_{l,s+1}a_{ls}^* \Big]  \\
  =& -\sum_{l=1}^{s-1}\sum_{  1\leq  r<n+1}\ \sum_{\mathbf q; r\geq q_i,\atop\exists a,(q_a,q_{a+1})=(l,s+1) }a_{ls}^* \prod_{k=1,k\neq a}^i a_{q_k q_{k+1}}^* 
 \Phi b_r    \end{align*}
and
\begin{align*}
  \sum_{  1\leq  r<n+1}&\Big[ \sum_{\mathbf q;\,q_i=r}  
  \big( \prod_{j=1}^{i -1}a_{q_{j}q{_{j +1} }}^* \kappa\cdot Da_{r,n+1}^* \big) _{\boldsymbol \lambda} 
   \big( \sum_{l=1}^{s-1}a_{l,s+1}a_{ls}^*\big)\Big] \\
&= - (1-\delta_{s,n})\sum_{  1\leq  r<n+1} \sum_{l=1}^{s-1}  \sum_{\mathbf q;\,q_i=r ,\exists a,(q_a,q_{a+1})=(l,s+1) } \left(\prod_{k=1,k\neq a}^{i-1} a_{q_k q_{k+1}}^* \right)a_{ls}^*\kappa\cdot Da_{r,n+1}^*   \\
 &+\delta_{s,n}\sum_{l=1}^{n-1} \sum_{\mathbf q;q_i=l} \kappa\cdot(D+{\boldsymbol \lambda})
  \left( \prod_{j=1}^{i -1}a_{q_{j}q{_{j +1} }}^*\right)  a_{ln}^*.
\end{align*}

Putting it all together we get
\begin{align}\label{f0fs.0}
[\rho(F_0)_{\boldsymbol \lambda} \rho(F_s)]&=  \delta_{s,n}a_{n,n+1}   \sum_{\mathbf q\atop n+1= q_i  } \prod_{k=1}^{i-1} a_{q_k q_{k+1}}^*   \\
 &+\sum_{s\leq r< j \leq n+1} a_{rj}   \sum_{\mathbf q ;\, j=q_i;\exists 1\leq k\leq i: \,(q_k,q_{k+1})=(s,s+1)}  	\prod_{j=1,j\neq k}^{i-1} a_{q_j q_{j+1}}^* a_{r,n+1}^*  \notag \\
&+\sum_{r\leq n}\sum_{\mathbf q \atop r\geq q_i \exists 1\leq k\leq i: \,(q_k,q_{k+1})=(s,s+1)} 
    	\prod_{j=1,j\neq k}^i a_{q_{j}q{_{j +1} }}^*   \Phi b_r \notag \\
&+ (1-\delta_{s,n})\sum_{1\leq r\leq n}\sum_{\mathbf q \atop   \exists 1\leq k\leq i-1: \,(q_k,q_{k+1})=(s,s+1);q_i=r} 
   	\prod_{j=1,j\neq k}^{i -1}a_{q_{j}q{_{j +1} }}^* \kappa\cdot Da_{r,n+1}^* \notag \\
& -\delta_{s,n}\sum_{\mathbf q ;q_i=n}  \kappa\cdot (D+{\boldsymbol \lambda})
   	\prod_{j=1}^{i -1}a_{q_{j}q{_{j +1} }}^* \notag \\
&-\delta_{s,n} \sum_{l=1}^{s-1}\sum_{l<j\leq n+1}a_{lj}  \sum_{\mathbf q;j= q_i;l\geq q_{i-1}} 	
	\prod_{l=1}^{i-1} a_{q_l q_{l+1}}^*a_{ls}^* \notag \\
 &+\sum_{l=1}^{s-1} a_{l,s+1} \sum_{\mathbf q;s= q_i;r\geq q_{i-1}} \prod_{l=1}^{i-1} 
 	a_{q_l q_{l+1}}^* a_{l,n+1}^* \notag \\
 &  -\sum_{l=1}^{s-1} \sum_{r< j } a_{rj}   
 	\sum_{\mathbf q; j=q_i; \,r\geq q_{i-1}  \atop \exists a,(q_a,q_{a+1})=(l,s+1) } 
	\left(\prod_{k=1,k\neq a}^{i-1} a_{q_k q_{k+1}}^* \right)
  	a_{ls}^* a_{r,n+1}^* \notag\\ 
&- \sum_{l=1}^{s-1}\sum_{  1\leq  r<n+1}\ \sum_{\mathbf q\atop r\geq q_i,\exists a,(q_a,q_{a+1})=(l,s+1) } 	\left(\prod_{k=1,k\neq a}^i a_{q_k q_{k+1}}^* \right)a_{ls}^* \Phi b_r    \notag\\
&- (1-\delta_{s,n}) \sum_{  1\leq  r<n+1} \sum_{l=1}^{s-1} 
	 \sum_{\stack{\mathbf q;\,q_i=r }{\exists a,(q_a,q_{a+1})=(l,s+1) } }	
	\left(\prod_{k=1,k\neq a}^{i-1} a_{q_k q_{k+1}}^* \right)a_{ls}^*\kappa\cdot Da_{q_i,n+1}^*  \notag \\
 &+\delta_{s,n}\sum_{l=1}^{n-1} \sum_{\mathbf q;q_i=l} \kappa\cdot (D+{\boldsymbol \lambda})
 	 \left( \prod_{j=1}^{i -1}a_{q_{j}q{_{j +1} }}^*\right)  a_{ln}^* \notag
\end{align}

Case I: $1<s<n$:   If $s>1$, then the third summation with $\Phi(b_r)$ in it, sums over all partitions $\mathbf q$ where there exists some $1\leq k\leq i$ and $q_i\leq r$,  such that 
$$
(1,q_2,\dots,q_k,q_{k+1},\dots ,q_i,n+1)=(1,q_2,\dots,s,s+1\dots ,q_i,n+1)
$$
where the $s$ and $s+1$ are in the $k$th, respectively $k+1$-st entry. Note $k>1$ as otherwise $s=1$.  Thus if $s\leq n-1$, then products appearing in this third summation look like
$$
\prod_{j=1,j\neq k}^i a_{q_{j}q{_{j +1} }}^*  \Phi(b_r) =a_{1,q_2}^*\cdots a_{q_{k-1},s}^*a_{s+1,q_{k+2}}^*\cdots a_{q_{i-1},q_i}^*a_{q_i,n+1}^* \Phi(b_r)
$$
In the ninth summation with $\Phi(b_r)$ in it, the sum is over all partitions $\mathbf q$ where there exists some $1\leq a\leq i-1$ and $q_i\leq r$ such that 
$$
(1,q_2,\dots,q_a,q_{a+1},\dots ,q_i,n+1)=(1,q_2,\dots,l,s+1\dots ,q_i,n+1)
$$
where the $l$ and $s+1$ are in the $a$th, respectively $a+1$-st entry. Note $1<a\leq n$ as otherwise $s=1$.  Thus if $s\leq n-1$,  products appearing in this ninth summation look like
$$
\left(\prod_{k=1,k\neq a}^i a_{q_k q_{k+1}}^* \right)a_{ls}^* \Phi(b_r)  =a_{1,q_2}^*\cdots a_{q_{k-1},l}^*a_{s+1,q_{k+2}}^*\cdots a_{q_{i-1},q_i}^* a_{q_i,n+1}^*a_{ls}^* \Phi(b_r)
$$
where $l\leq s-1$ and $r\geq q_i\geq l$.  Thus the third and ninth summations are equal but of opposite sign and they cancel when $s\leq n-1$.

   If $1<s<n$, then the fourth summation with $\kappa\cdot Da_{q_i,n+1}^*$ in it,  sums over all partitions $\mathbf q$ where there exists some $1\leq k\leq i-1$ such that 
$$
(1,q_2,\dots,q_k,q_{k+1},\dots ,q_{i-1},q_i,n+1)=(1,q_2,\dots,s,s+1\dots ,q_{i-1},r,n+1)
$$
where the $s$ and $s+1$ are in the $k$th, respectively $k+1$-st entry. Note $k>1$ as otherwise $s=1$.  The products appearing in this fourth summation look like
$$
\prod_{j=1,j\neq k}^{i -1}a_{q_{j}q{_{j +1} }}^* \kappa\cdot Da_{q_i,n+1}^*
=a_{1,q_2}^*\cdots a_{q_{k-1},s}^*a_{s+1,q_{k+2}}^*\cdots a_{q_{i-1}q_i}^*\kappa\cdot Da_{r ,n+1}^*.
$$
In the tenth summation with $\kappa\cdot Da_{q_i,n+1}^*$ in it, the sum is over all partitions $\mathbf q$ where there exists some $1\leq a\leq i-1$ such that 
$$
(1,q_2,\dots,q_a,q_{a+1},\dots ,q_{i-1},q_i,n+1)=(1,q_2,\dots,l,s+1\dots ,q_{i-1},r,n+1)
$$
where the $l$ and $s+1$ are in the $a$th, respectively $a+1$-st entry. Note $a>1$ as otherwise $s=1$.  The products appearing in this tenth summation look like (after a change of indices $k\mapsto j$ and $a\mapsto k$)
$$
 \left(\prod_{j=1,j\neq k}^{i-1} a_{q_j q_{j+1}}^* \right)a_{ls}^*\kappa\cdot Da_{q_i,n+1}^*  =a_{1,q_2}^*\cdots a_{q_{k-1},l}^*a_{s+1,q_{k+2}}^*\cdots a_{q_{i-1}q_i}^*a_{ls}^*\kappa\cdot Da_{r,n+1}^* 
$$
where $l\leq s-1$ and $q_i= r<n+1$.  Thus the forth and the tenth summations are equal but of opposite sign and they cancel.  Moreover  the fifth and the last summations are zero as $s<n$.

Case II : $s=n$:  

If $s=n$, then the indices in the third summation are $n=s=q_k$, $k=i$, $s+1=q_{i+1}=n+1$ and $r= q_i=n$. Then in this caee the third summation consists of products of the form
$$
\prod_{j=1}^{l-1} a_{q_{j}q{_{j +1} }}^*  \Phi(b_n) =a_{1,q_2}^*\cdots a_{q_{i-1},n}^*  \Phi(b_n)
$$
 as 
 $$
(1,q_2,\dots,q_k,q_{k+1} )=(1,q_2,\dots,q_{i-1},n,n+1),
$$
whereas the ninth summation consists of products of the form
$$
\left(\prod_{k=1,k\neq i-1}^i a_{q_k q_{k+1}}^* \right)a_{ln}^* \Phi(b_r)  =a_{1,q_2}^*\cdots a_{q_{i-1},l}^*a_{ln}^* \Phi(b_r)
$$
where $l\leq n-1$ and $r\geq q_i=l$ (as $(q_a,q_{a+1})=(l,n+1)$ implies that $a+1=i+1$ and hence $q_i=l$).   Note the difference in the coefficient in front of $\Phi(b_n)$ in the two summands, so that not all the terms with the $\Phi(b_r)$ in them cancel when $s=n$.  Thus for $s=n$, we are left  with 
\begin{align*}
&+\sum_{r\leq n}\sum_{\mathbf q \atop r\geq q_i \exists 1\leq k\leq i: \,(q_k,q_{k+1})=(n,n+1)} 
    	\prod_{j=1,j\neq k}^i a_{q_{j}q{_{j +1} }}^*   \Phi b_r  \\
&- \sum_{l=1}^{n-1}\sum_{  1\leq  r<n+1}\ \sum_{\mathbf q\atop r\geq q_i,\exists a,(q_a,q_{a+1})=(l,n+1) } 	\left(\prod_{k=1,k\neq a}^i a_{q_k q_{k+1}}^* \right)a_{ln}^* \Phi b_r    \\
&=- \sum_{l=1}^{n-1}\sum_{  1\leq  r<n}\ \sum_{\mathbf q\atop r\geq q_{l(\mathbf q)}=l} 	\left(\prod_{k=1}^{l(\mathbf q)-1}a_{q_k q_{k+1}}^* \right)a_{ln}^* \Phi b_r .
\end{align*}

If $s=n$, 
then the fourth summation  and the tenth are both zero due to the factor $(1-\delta_{s,n}$ in both. 
Moreover the second summation in \eqnref{f0fs.0} has products of the form
\begin{equation}\label{mess6.1}
 a_{rj}\prod_{j=1,j\neq k}^{i-1} a_{q_j q_{j+1}}^*= a_{rj}a_{1q_1}^*\cdots a_{q_{k-1}s}^*a_{s+1,q_{k+1}}^*\cdots a_{q_{i-1}j}^*a_{r,n+1}^*
\end{equation}
with $s\leq r<j=q_i$, $s+1\leq q_i=j$, whereas the seventh summation, after setting $r=l$, has products of the form
\begin{equation}\label{seventh}
 a_{r,s+1}\prod_{t=1}^{i-1} a_{q_t q_{t+1}}^*a_{r,n+1}^*= a_{r,s+1}a_{1q_1}^*\cdots a_{q_{i-1}s}^*a_{r,n+1}^*,
 \end{equation}
with $r<s=q_i$ and $q_{i-1}\leq r<j$.  
The eighth summation 
\begin{align*}
a_{rj} \left(\prod_{k=1,k\neq a}^i a_{q_k q_{k+1}}^* \right)
  a_{ls}^* a_{r,n+1}^*&=a_{rj} a_{1,q_1}^*\cdots a_{q_{a-1},l}^*a_{s+1,q_{a+2}}^*a_{q_{i-1} j}^*  
  a_{ls}^* a_{r,n+1}^*  \\
  &=a_{rj} a_{1,q_1}^*\cdots a_{q_{a-1},l}^*a_{ls}^* a_{s+1,q_{a+2}}^*a_{q_{i-1} j}^*  
  a_{r,n+1}^*  
\end{align*}
where $r$ has the restriction that $q_{i-1}\leq r<j=q_i$, $l\leq s-1$ and $s+1\leq q_i=j$.   If we consider the eighth summation when $a=i-1$ so $(q_{i-1},j)=(q_{i-1},q_{i})=(l,s+1)$, then this part of the eighth summation has products of the form 
\begin{equation}\label{mess6.4}
a_{r,s+1} a_{1,q_1}^*\cdots a_{q_{a-1},l}^*a_{ls}^*
  a_{r,n+1}^*.
\end{equation}
These summands cancel when $r<s$ with summands in the seventh summation in \eqnref{seventh}.  When $r=s$ (and $a=i-1$ so $(q_{i-1},j)=(q_{i-1},q_{i})=(l,s+1)$), the summands \eqnref{mess6.4} in the eighth summation cancel with the summands in the second summation in \eqnref{mess6.1} where $(r,j)=(s,s+1)$.  
If $a<i-1$, then in the eighth summation $a+1\leq i-1$, so that $s+1=q_{a+1}\leq q_{i-1}\leq r$ and these  terms in the eighth summation cancel with the remaining summands in the second summation of the form \eqnref{mess6.1}.

Hence for $s>1$, we get from \eqnref{f0fs.0} 
\begin{align}
[\rho(F_0)_{\boldsymbol \lambda} \rho(F_s)]&=\delta_{s,n}a_{n,n+1}   \sum_{\mathbf q\atop  q_i=n+1 } \prod_{k=1}^{i-1} a_{q_k q_{k+1}}^*  - \delta_{s,n}\sum_{l=1}^{n-1}\sum_{l<j\leq n+1}a_{lj}  \sum_{\mathbf q;j= q_i;l\geq q_{i-1}} \prod_{m=1}^{i-1} a_{q_m q_{m+1}}^*a_{ln}^*   \\ 
&\quad - \delta_{s,n}\sum_{\mathbf q ;q_i=n}  \kappa\cdot (D+{\boldsymbol \lambda})
   	\prod_{j=1}^{i -1}a_{q_{j}q{_{j +1} }}^* \notag \\
  &\quad - \delta_{s,n} \sum_{l=1}^{n-1}\sum_{  1\leq  r<n}
  	\sum_{\mathbf q\atop r\geq q_i,\exists a,(q_a,q_{a+1})=(l,n+1) } 	
	\left(\prod_{k=1,k\neq a}^i a_{q_k q_{k+1}}^* \right)a_{ln}^* \Phi b_r    \notag\\
&\quad +\delta_{s,n}\sum_{l=1}^{n-1} \sum_{\mathbf q;q_i=l} \kappa\cdot (D+{\boldsymbol \lambda})
 	 \left( \prod_{j=1}^{i -1}a_{q_{j}q{_{j +1} }}^*\right)  a_{ln}^*.   \notag
\end{align} 
This proves the Serre relation for $s\neq 0,1,n$.    If $s=n$ we have 
\begin{align*}
[\rho(F_0)_{\boldsymbol \lambda} \rho(F_n)]&=a_{n,n+1}   \sum_{\mathbf q\atop  q_i=n+1 } \prod_{k=1}^{i-1} a_{q_k q_{k+1}}^*  - \sum_{l=1}^{n-1}\sum_{l<j\leq n+1}a_{lj}  \sum_{\mathbf q;j= q_i;l\geq q_{i-1}} \prod_{m=1}^{i-1} a_{q_m q_{m+1}}^*a_{ln}^*   \\ 
&\quad  - \sum_{\mathbf q ;q_i=n}  \kappa\cdot (D+{\boldsymbol \lambda})
   	\prod_{j=1}^{i -1}a_{q_{j}q{_{j +1} }}^* \\
  &\quad -  \sum_{l=1}^{n-1}\sum_{  1\leq  r<n}
  	\sum_{\mathbf q\atop r\geq q_i,\exists a,(q_a,q_{a+1})=(l,n+1) } 	
	\left(\prod_{k=1,k\neq a}^i a_{q_k q_{k+1}}^* \right)a_{ln}^* \Phi b_r    \\
&\quad +\sum_{l=1}^{n-1} \sum_{\mathbf q;q_i=l} \kappa\cdot (D+{\boldsymbol \lambda})
 	 \left( \prod_{j=1}^{i -1}a_{q_{j}q{_{j +1} }}^*\right)  a_{ln}^* .
\end{align*} 
We want to show $[[\rho(F_0)_{\boldsymbol \lambda} \rho(F_n)]_{\boldsymbol \mu}\rho(F_n)]=0$.  To prove this first recall
$\displaystyle{
\rho(F_n)=a_{n,n+1}-\sum_{p=1}^{n-1}a_{p,n+1}a_{p,n}^*
}$.

Now 
\begin{align*}
[[\rho(F_0)_{\boldsymbol \lambda}& \rho(F_n)]_{\boldsymbol \mu} a_{n,n+1}]\\
=  &
 [\big(a_{n,n+1}   \sum_{\mathbf q\atop  q_i=n+1 } \prod_{k=1}^{i-1} a_{q_k q_{k+1}}^*  -  \sum_{l=1}^{n-1}\sum_{l<j\leq n+1}a_{lj}  \sum_{r,\mathbf q;j= q_i;l\geq q_{i-1}} \prod_{m=1}^{i-1} a_{q_m q_{m+1}}^*a_{ln}^* \big){_{\boldsymbol \mu}} a_{n,n+1}]  \\
 =&-a_{n,n+1}   \sum_{\mathbf q\atop  q_{i-1}=n } \prod_{k=1}^{i-2} a_{q_k q_{k+1}}^* ,
\end{align*} 
whereas 
\begin{align*}
-\sum_{p=1}^{n-1}&[[\rho(F_0)_{\boldsymbol \lambda}  \rho(F_n)]_{\boldsymbol \mu} a_{p,n+1}a_{p,n}^*] \\
=  
&-\sum_{p=1}^{n-1}[\big(a_{n,n+1}   \sum_{\mathbf q\atop q_i=n+1 } \prod_{k=1}^{i -1}a_{q_k q_{k+1}}^*  - \sum_{1\leq l<j\leq n+1}a_{lj}  \sum_{\mathbf q;j= q_i;l\geq q_{i-1}} \prod_{m=1}^{i-1} a_{q_m q_{m+1}}^*a_{ln}^* \big){_{\boldsymbol \mu}} a_{p,n+1}a_{p,n}^*]  \\ \\
 =&a_{n,n+1}  \sum_{p=1}^{n-1}  \sum_{\mathbf q\atop q_{i-1}=p } \prod_{k=1}^{i-2} a_{q_k q_{k+1}}^*  
a_{p,n}^* +\sum_{l=1}^{n-1}a_{l,n+1}  \sum_{\mathbf q;q_i=n;l\geq q_{i-1}} \prod_{m=1}^{i-1} a_{q_m q_{m+1}}^*a_{l,n}^*   \\
&  - \sum_{l=1}^{n-1}  a_{l,n+1} \sum_{p=1}^{l} \sum_{\mathbf q;n+1= q_i;l\geq q_{i-1}=p} \prod_{m=1}^{i-2} a_{q_m q_{m+1}}^*  a_{p,n}^*a_{l,n}^*  
\end{align*} 
Hence 
\begin{equation}
\label{f0fnfn }
[[\rho(F_0)_{\boldsymbol \lambda} \rho(F_n)]_{\boldsymbol \mu}  \rho(F_n)]=0.
\end{equation}

For $s=1$ we get $\rho(F_1)=a_{1,2}$ and hence
\begin{align*}
[\rho(F_0)_{\boldsymbol \lambda} \rho(F_1)]&= \sum_{1\leq r< j \leq n+1} a_{rj}   
	\sum_{\mathbf q ;\, j=q_i, r\geq q_{i-1},  q_2=2}  \prod_{j=2}^{i-1} a_{q_j q_{j+1}}^*a_{r,n+1}^*     \\
&\quad +\sum_{r< n+1}\sum_{\mathbf q \atop r=q_i: q_2=2} 
    \prod_{j=2}^i a_{q_{j}q{_{j +1} }}^*   \Phi b_r  
+\sum_{1\leq r\leq n}\sum_{\mathbf q \atop  q_2=2;q_i=r} 
   \prod_{j=2}^{i -1}a_{q_{j}q{_{j +1} }}^* \kappa\cdot Da_{r,n+1}^* \\
\end{align*} 
Thus \begin{equation}
\label{f0f1f1 }
[[\rho(F_0)_{\boldsymbol \lambda} \rho(F_1)]_{\boldsymbol \mu}\rho(F_1)]=0.
\end{equation}

Next up is the calculation for $[\rho(F_0)_{\boldsymbol \lambda}[\rho(F_0)_{\boldsymbol \mu} \rho(F_1)]]$:  For a partition $\mathbf q=(1=q_1,q_2,\dots, q_i,n+1)$ recall we set $l(\mathbf q)=i$.  We now write
\begin{align*}
[\rho(F_0)&_{\boldsymbol \lambda} \rho(F_1)]=A_{01}+B_{01}+C_{01} \quad\text{where}\\
A_{01}&= \sum_{1\leq r< j \leq n+1} a_{rj}  
	 \sum_{\mathbf q ;\, j=q_{l(\mathbf q)};  r\geq q_{l(\mathbf p)-1},q_2=2}  
	 	\prod_{j=2}^{l(\mathbf q)-1} a_{q_j q_{j+1}}^* a_{r,n+1}^*    \\
B_{01}&=\sum_{r\leq n}\sum_{\mathbf q \atop r\geq q_{l(\mathbf q)}: q_2=2} 
    \prod_{j=2}^{l(\mathbf q)} a_{q_{j}q{_{j +1} }}^*   \Phi b_r  \\ 
C_{01}&=\sum_{1\leq r\leq n}\sum_{\mathbf q \atop  q_2=2;q_{l(\mathbf q)}=r} 
   \prod_{j=2}^{l(\mathbf q)-1}a_{q_{j}q{_{j +1} }}^* \kappa\cdot Da_{r,n+1}^* \\
\end{align*} 
and $F_0=A+B+C$ where
\begin{align*}
A& =\sum_{1\leq r<j  \leq n+1} -a_{rj}  \sum_{\mathbf q;j=q_{ l(\mathbf q)};r\geq q_{l(\mathbf p)-1}} \prod_{l=1}^{l(\mathbf q)-1} a_{q_l q_{l+1}}^* a_{r,n+1}^*  \\
B&= -\sum_{1 \leq r <n+1}\sum_{ r \geq  q_{l(\mathbf q)} ,\mathbf q} 
    \prod_{j=1}^{l(\mathbf q)} a_{q_{j}q{_{j +1} }}^*  \Phi(b_r)    \\ 
C&= -\sum_{1\leq r <n+1}\sum_{ r=q_{l(\mathbf q)},\mathbf q} 
   \prod_{j=1}^{l(\mathbf q)-1}a_{q_{j}q{_{j +1} }}^* \kappa \cdot Da_{r,n+1}^* .
\end{align*}
Then
\begin{align*}
[F_0{_{\boldsymbol \lambda}} [F_0{_{\boldsymbol \mu}}F_1]]
& =[A_{\boldsymbol \lambda} A_{01}]+[B_{\boldsymbol \lambda} B_{01}]   +[A_{\boldsymbol \lambda} B_{01}]+[B_{\boldsymbol \lambda} A_{01}]+[A_{\boldsymbol \lambda} C_{01}]+[C_{\boldsymbol \lambda} A_{01}] .
 \end{align*}
 Now we calculate each summand above
\begin{align*}
[A_{\boldsymbol \lambda} &A_{01}] = - \sum_{1\leq r<j  \leq n+1\atop 1\leq s<k  \leq n+1} 
	\sum_{\mathbf q;j= q_{l(\mathbf q)};r\geq q_{l(\mathbf q)-1}
	\atop \mathbf p;k= p_l(\mathbf p) ;s\geq p_{l(\mathbf p) -1},p_2=2}[a_{rj}   
	\prod_{l=1}^{l(\mathbf q)-1} a_{q_l q_{l+1}}^*	
	a_{r,n+1}^*{_{\boldsymbol \lambda}} a_{sk}    \prod_{\xi=2}^{l(\mathbf p) -1} a_{p_\xi p_{\xi+1}}^* a_{s,n+1}^*]   \\
 = & -\sum_{1\leq r<j  \leq n+1\atop 1\leq s<k  \leq n+1} 
	\sum_{\mathbf q;j= q_{l(\mathbf q)};r\geq q_{l(\mathbf q)-1}
	\atop \mathbf p;k= p_{l(\mathbf p)} ;s\geq p_{{l(\mathbf p)} -1},p_2=2} a_{sk}  
	\left( [a_{rj}{_{\boldsymbol \lambda}} 	 \prod_{\xi=2}^{l(\mathbf p)-1}
	  a_{p_\xi p_{\xi+1}}^*] a_{s,n+1}^* \right. \\ 
	  & \hskip 150pt \left. +   
	   \prod_{\xi=2}^{l(\mathbf p)-1}a_{p_\xi p_{\xi+1}}^*[a_{rj} {_{\boldsymbol \lambda}} 
	a_{s,n+1}^*]\right) \prod_{l=1}^{l(\mathbf q)-1} a_{q_l q_{l+1}}^* a_{r,n+1}^*  \\  
 & -  \sum_{1\leq r<j  \leq n+1\atop 1\leq s<k  \leq n+1} 
	\sum_{\mathbf q;j= q_{l(\mathbf q)};r\geq q_{l(\mathbf q)-1}
	\atop \mathbf p;k= p_{l(\mathbf p)} ;s\geq p_{{l(\mathbf p)} -1},p_2=2}
	a_{rj}   \left([\prod_{l=1}^{l(\mathbf q)-1}a_{q_l q_{l+1}}^* {_{\boldsymbol \lambda}}  
	a_{sk}]a_{r,n+1}^*
	\right. \\ 
	  &  \hskip 150pt  \left.  + \prod_{l=1}^{l(\mathbf q)-1}a_{q_l q_{l+1}}^*
	[a_{r,n+1}^*{_{\boldsymbol \lambda}} a_{sk} ] \right) \prod_{\xi=2}^{l(\mathbf p) -1} 
	a_{p_\xi p_{\xi+1}}^* a_{s,n+1}^*.  
	\end{align*}
	Re-indexing the above gives
	\begin{align*}
  =  &-\sum_{1\leq r<j  \leq n+1\atop 1\leq s<k  \leq n+1} 
	\sum_{\mathbf q;j= q_{l(\mathbf q)};r\geq q_{l(\mathbf q)-1}
	\atop \mathbf p;k= p_{l(\mathbf p)} ;s\geq p_{{l(\mathbf p)} -1},p_2=2} a_{sk}  
	\left( [a_{rj}{_{\boldsymbol \lambda}}  	 \prod_{\xi=2}^{l(\mathbf p)-1}
	  a_{p_\xi p_{\xi+1}}^*] a_{s,n+1}^* +   
	   \prod_{\xi=2}^{l(\mathbf p)-1}a_{p_\xi p_{\xi+1}}^*[a_{rj} {_{\boldsymbol \lambda}} 
	a_{s,n+1}^*]\right) \\
	&\hskip 200pt \cdot \prod_{l=1}^{l(\mathbf q)-1} a_{q_l q_{l+1}}^* a_{r,n+1}^*  \\  
& -  \sum_{1\leq s <k   \leq n+1\atop 1\leq r<j   \leq n+1} 
	\sum_{\mathbf q;k = q_{l(\mathbf q)};s \geq q_{l(\mathbf q)-1}
	\atop \mathbf p;j = p_{l(\mathbf p)} ;r\geq p_{{l(\mathbf p)} -1},p_2=2}
	a_{s k }   \left([\prod_{l=1}^{l(\mathbf q)-1}a_{q_l q_{l+1}}^* {_{\boldsymbol \lambda}} a_{rj }]a_{s ,n+1}^*
	+\prod_{l=1}^{l(\mathbf q)-1} a_{q_l q_{l+1}}^*
	[a_{s ,n+1}^*{_{\boldsymbol \lambda}} a_{rj } ] \right) \\
	&\hskip 200pt \cdot\prod_{\xi=2}^{l(\mathbf p) -1} 
	a_{p_\xi p_{\xi+1}}^* a_{r,n+1}^*  
\end{align*}
\begin{align*}
&=  -\sum_{ 1\leq s<k  \leq n+1} 
	\sum_{\mathbf q;n+1= q_{l(\mathbf q)};s\geq q_{l(\mathbf q)-1}
	\atop \mathbf p;k= p_{l(\mathbf p)} ;s\geq p_{{l(\mathbf p)} -1},p_2=2} a_{sk}  
	   \prod_{\xi=2}^{l(\mathbf p)-1}a_{p_\xi p_{\xi+1}}^*
	   \prod_{l=1}^{l(\mathbf q)-1} a_{q_l q_{l+1}}^* a_{s,n+1}^*  \\  
&\quad +  \sum_{1\leq s <k   \leq n+1 } 
	\sum_{\mathbf q;k = q_{l(\mathbf q)};s \geq q_{l(\mathbf q)-1}
	\atop \mathbf p;n+1 = p_{l(\mathbf p)} ;s\geq p_{{l(\mathbf p)} -1},p_2=2} a_{s k }  
	\prod_{l=1}^{l(\mathbf q)-1} a_{q_l q_{l+1}}^*
	 \prod_{\xi=2}^{l(\mathbf p) -1} 
	a_{p_\xi p_{\xi+1}}^* a_{s,n+1}^*   \\ 
&\quad -\sum_{1\leq r<j  \leq n+1\atop 1\leq s<k  \leq n+1} 
	\sum_{\mathbf q;j= q_{l(\mathbf q)};r\geq q_{l(\mathbf q)-1}
	\atop \mathbf p;k= p_{l(\mathbf p)} ;s\geq p_{{l(\mathbf p)} -1},p_2=2} a_{sk}  
	 [a_{rj}{_{\boldsymbol \lambda}} 	 \prod_{\xi=2}^{l(\mathbf p)-1}
	  a_{p_\xi p_{\xi+1}}^*] a_{s,n+1}^*  \prod_{l=1}^{l(\mathbf q)-1} a_{q_l q_{l+1}}^* a_{r,n+1}^*  \\  
&\quad -  \sum_{1\leq s <k   \leq n+1\atop 1\leq r<j   \leq n+1} 
	\sum_{\mathbf q;k = q_{l(\mathbf q)};s \geq q_{l(\mathbf q)-1}
	\atop \mathbf p;j = p_{l(\mathbf p)} ;r\geq p_{{l(\mathbf p)} -1},p_2=2}
	a_{s k }   [\prod_{l=1}^{l(\mathbf q)-1}a_{q_l q_{l+1}}^* {_{\boldsymbol \lambda}} a_{rj }]a_{s ,n+1}^*
	  \prod_{\xi=2}^{l(\mathbf p) -1} 
	a_{p_\xi p_{\xi+1}}^* a_{r,n+1}^*     \\ 
&=  -\sum_{ 1\leq s<k  \leq n+1} 
	\sum_{\mathbf q;n+1= q_{l(\mathbf q)};s\geq q_{l(\mathbf q)-1}
	\atop \mathbf p;k= p_{l(\mathbf p)} ;s\geq p_{{l(\mathbf p)} -1},p_2=2} a_{sk}  
	   \prod_{\xi=2}^{l(\mathbf p)-1}a_{p_\xi p_{\xi+1}}^*
	   \prod_{l=1}^{l(\mathbf q)-1} a_{q_l q_{l+1}}^* a_{s,n+1}^*  \\  
&\quad +  \sum_{1\leq s <k   \leq n+1 } 
	\sum_{\mathbf q;k = q_{l(\mathbf q)};s \geq q_{l(\mathbf q)-1}
	\atop \mathbf p;n+1 = p_{l(\mathbf p)} ;s\geq p_{{l(\mathbf p)} -1},p_2=2} a_{s k }  
	\prod_{l=1}^{l(\mathbf q)-1} a_{q_l q_{l+1}}^*
	 \prod_{\xi=2}^{l(\mathbf p) -1} 
	a_{p_\xi p_{\xi+1}}^* a_{s,n+1}^*   \\ 
&\quad -\sum_{1\leq r<j  \leq n+1\atop 1\leq s<k  \leq n+1} 
	\sum_{\mathbf q;j= q_{l(\mathbf q)};r\geq q_{l(\mathbf q)-1,
	\exists t\leq l(\mathbf p)-1; (r,j)=(p_t,p_{t+1})}
	\atop \mathbf p;k= p_{l(\mathbf p)} ;s\geq p_{{l(\mathbf p)} -1},p_2=2} a_{sk}  
	 [a_{rj}{_{\boldsymbol \lambda}} 	 \prod_{\xi=2}^{l(\mathbf p)-1}
	  a_{p_\xi p_{\xi+1}}^*] a_{s,n+1}^*  \prod_{l=1}^{l(\mathbf q)-1} a_{q_l q_{l+1}}^* a_{r,n+1}^*  \\  
&\quad -  \sum_{1\leq r\leq s <k   \leq n+1\atop 1\leq r<j   \leq n+1} 
	\sum_{\mathbf q;k = q_{l(\mathbf q)};s \geq q_{l(\mathbf q)-1},
	\exists t\leq l(\mathbf q)-1; (r,j)=(q_t,q_{t+1})
	\atop \mathbf p;j = p_{l(\mathbf p)} ;r\geq p_{{l(\mathbf p)} -1},p_2=2}
	a_{s k }   [\prod_{l=1}^{l(\mathbf q)-1}a_{q_l q_{l+1}}^* {_{\boldsymbol \lambda}} a_{rj }]a_{s ,n+1}^*
	  \prod_{\xi=2}^{l(\mathbf p) -1} 
	a_{p_\xi p_{\xi+1}}^* a_{r,n+1}^*    
\end{align*}
\begin{align*}
 = & -\sum_{ 1\leq s<k  \leq n+1} 
	\sum_{\mathbf q;n+1= q_{l(\mathbf q)};s\geq q_{l(\mathbf q)-1}
	\atop \mathbf p;k= p_{l(\mathbf p)} ;s\geq p_{{l(\mathbf p)} -1},p_2=2} a_{sk}  
	   \prod_{\xi=2}^{l(\mathbf p)-1}a_{p_\xi p_{\xi+1}}^*
	   \prod_{l=1}^{l(\mathbf q)-1} a_{q_l q_{l+1}}^* a_{s,n+1}^*  \\  
& +  \sum_{1\leq s <k   \leq n+1 } 
	\sum_{\mathbf q;k = q_{l(\mathbf q)};s \geq q_{l(\mathbf q)-1}
	\atop \mathbf p;n+1 = p_{l(\mathbf p)} ;s\geq p_{{l(\mathbf p)} -1},p_2=2} a_{s k }  
	\prod_{l=1}^{l(\mathbf q)-1} a_{q_l q_{l+1}}^*
	 \prod_{\xi=2}^{l(\mathbf p) -1} 
	a_{p_\xi p_{\xi+1}}^* a_{s,n+1}^*   \\ 
&\  -\sum_{2\leq r\leq s,j\leq k  \leq n+1} 
	\sum_{\mathbf q;j= q_{l(\mathbf q)};r\geq q_{l(\mathbf q)-1,
	\exists t;2\leq t \leq l(\mathbf p)-1; (r,j)=(p_t,p_{t+1})}
	\atop \mathbf p;k= p_{l(\mathbf p)} ;s\geq p_{{l(\mathbf p)} -1},p_2=2} a_{sk}  
	 \prod_{\xi=2,\xi\neq t}^{l(\mathbf p)-1}
	  a_{p_\xi p_{\xi+1}}^* a_{s,n+1}^*  \prod_{l=1}^{l(\mathbf q)-1} a_{q_l q_{l+1}}^* a_{r,n+1}^*  \\  
&  +  \sum_{1\leq r\leq s <k   \leq n+1\atop 1\leq r<j   \leq n+1} 
	\sum_{\mathbf q;k = q_{l(\mathbf q)};s \geq q_{l(\mathbf q)-1},
	\exists t\leq l(\mathbf q)-1; (r,j)=(q_t,q_{t+1})
	\atop \mathbf p;j = p_{l(\mathbf p)} ;r\geq p_{{l(\mathbf p)} -1},p_2=2}
	a_{s k }  \prod_{l=1,l\neq t}^{l(\mathbf q)-1}a_{q_l q_{l+1}}^* a_{s ,n+1}^*
	  \prod_{\xi=2}^{l(\mathbf p) -1} 
	a_{p_\xi p_{\xi+1}}^* a_{r,n+1}^*   .
\end{align*}
To show that the above summation is zero reduces to showing the following are zero:
\begin{align*}
& I_1:=  - 
	\sum_{\mathbf q;n+1= q_{l(\mathbf q)};s\geq q_{l(\mathbf q)-1}
	\atop \mathbf p;k= p_{l(\mathbf p)} ;s\geq p_{{l(\mathbf p)} -1},p_2=2} 
	   \prod_{\xi=2}^{l(\mathbf p)-1}a_{p_\xi p_{\xi+1}}^*
	   \prod_{l=1}^{l(\mathbf q)-1} a_{q_l q_{l+1}}^*    \\  
& \quad +  \sum_{ 1\leq r<j   \leq n+1} 
	\sum_{\mathbf q;k = q_{l(\mathbf q)};s \geq q_{l(\mathbf q)-1},
	\exists t\leq l(\mathbf q)-1; (r,j)=(q_t,q_{t+1})
	\atop \mathbf p;j = p_{l(\mathbf p)} ;r\geq p_{{l(\mathbf p)} -1},p_2=2}
	   \prod_{l=1,l\neq t}^{l(\mathbf q)-1}a_{q_l q_{l+1}}^*  
	  \prod_{\xi=2}^{l(\mathbf p) -1} 
	a_{p_\xi p_{\xi+1}}^* a_{r,n+1}^*     \\ \\
\end{align*}
Note that the first summation is over all partitions $\mathbf p$ and $\mathbf q$ and has summands of the form
$$
a_{2p_3}^*\cdots a_{p_{l(\mathbf p)-1}k}^*a_{1,q_2}^*\cdots a_{q_{l(\mathbf q)-1}n+1}^*
$$
with $p_{l(\mathbf p)-1}\leq s$ and $q_{l(\mathbf q)-1}\leq s$. 
The second summation is over all partitions $\mathbf p$ and $\mathbf q$ and has summands of the form
\begin{align*}
&a_{1,q_2}^*\cdots a_{q_{t-1},r}^*a_{j,q_{t+2}}^*\cdots a_{q_{l(\mathbf q)-1},k}^*a_{2,p_3}^*\cdots a_{p_{l(\mathbf p)-1},j}^*a_{r,n+1}^*  \\
&=(a_{2,p_3}^*\cdots a_{p_{l(\mathbf p)-1},j}^*a_{j,q_{t+2}}^*\cdots a_{q_{l(\mathbf q)-1},k}^*) (a_{1,q_2}^*\cdots a_{q_{t-1},r}^*a_{r,n+1}^*  )\\
\end{align*}
where $s \geq q_{l(\mathbf q)-1},s\geq  r\geq p_{{l(\mathbf p)} -1}, r<j\leq k$.  But these two sets of partitions are the same so $I_1=0$.  

Similarly if we look at the partitions for the summands of 
\begin{align*}
&I_2:= 
	\sum_{\mathbf q;k = q_{l(\mathbf q)};s \geq q_{l(\mathbf q)-1}
	\atop \mathbf p;n+1 = p_{l(\mathbf p)} ;s\geq p_{{l(\mathbf p)} -1},p_2=2}  
	\prod_{l=1}^{l(\mathbf q)-1} a_{q_l q_{l+1}}^*
	 \prod_{\xi=2}^{l(\mathbf p) -1} 
	a_{p_\xi p_{\xi+1}}^*     \\ 
&\quad -\sum_{2\leq r <j \leq n+1} 
	\sum_{\stack{\stack{\mathbf q;j= q_{l(\mathbf q)};r\geq q_{l(\mathbf q)-1}
	}{\exists t;2\leq t \leq l(\mathbf p)-1; (r,j)=(p_t,p_{t+1})}}  
	{\mathbf p;k= p_{l(\mathbf p)} ;s\geq p_{{l(\mathbf p)} -1},p_2=2}   }
	 \prod_{\xi=2,\xi\neq t}^{l(\mathbf p)-1}
	  a_{p_\xi p_{\xi+1}}^*    \prod_{l=1}^{l(\mathbf q)-1} a_{q_l q_{l+1}}^* a_{r,n+1}^*  \\  
\end{align*}
Collecting partitions shows that $I_2=0$.  Hence $[A_{\boldsymbol \lambda} A_{01}]=0$.

\begin{align*}
[B_{\boldsymbol \lambda} B_{01}]  &=-\sum_{1 \leq r <n+1}
	\sum_{1 \leq k <n+1}\sum_{ \mathbf q; r \geq  q_{l(\mathbf q)}} 
	\sum_{ \mathbf p,k \geq  p_{l(\mathbf p)} ,p_2=2} 
     	[\prod_{j=1}^{l(\mathbf q)} a_{q_{j}q{_{j +1} }}^* \Phi(b_r){_{\boldsymbol \lambda}}
	\prod_{l=2}^{l(\mathbf p)} a_{p_{l}p{_{l +1} }}^* \Phi(b_k)]  \\
&=\sum_{1 \leq r <n+1}\sum_{1 \leq k <n+1}A_{rk}\sum_{ \mathbf q; r \geq  q_{l(\mathbf q)} } 
	\sum_{ \mathbf p,k \geq  p_{l(\mathbf p)} ,p_2=2} 
   \kappa\cdot ({\boldsymbol \lambda} +D)\left( \prod_{j=1}^{l(\mathbf q)} a_{q_{j}q{_{j +1} }}^*\right)
    \prod_{l=2}^{l(\mathbf p)} a_{p_{l}p{_{l +1} }}^* \\
&=\sum_{1 \leq r <n+1}\sum_{1 \leq k <n+1}(-\delta_{r,k-1}+2\delta_{rk}-\delta_{r,k+1})
	\sum_{ \stack{\mathbf q}{r \geq  q_{l(\mathbf q)}} }\sum_{ \stack{\mathbf p}{k \geq  p_{l(\mathbf p)} ,p_2=2} }
   	\kappa\cdot ({\boldsymbol \lambda} +D)\left( \prod_{j=1}^{l(\mathbf q)} a_{q_{j}q{_{j +1} }}^*\right)
    	\prod_{l=2}^{l(\mathbf p)} a_{p_{l}p{_{l +1} }}^* \\
&=-\sum_{1 \leq r <n }  \sum_{\mathbf q; r \geq  q_{l(\mathbf q)}} 
	\sum_{ \mathbf p ; r+1 \geq  p_{l(\mathbf p)} ,p_2=2} 
   	\kappa\cdot ({\boldsymbol \lambda} +D)\left( \prod_{j=1}^{l(\mathbf q)} a_{q_{j}q{_{j +1} }}^*\right)
    	\prod_{l=2}^{l(\mathbf p)} a_{p_{l}p{_{l +1} }}^* \\
&\quad +2 \sum_{1 \leq r <n+1}   \sum_{ \mathbf q; r \geq  q_{l(\mathbf q)} } 
	\sum_{ \mathbf p ; r \geq  p_{l(\mathbf p)} ,p_2=2} 
   	\kappa\cdot ({\boldsymbol \lambda} +D)\left( \prod_{j=1}^{l(\mathbf q)} a_{q_{j}q{_{j +1} }}^*\right)
    	\prod_{l=2}^{l(\mathbf p)} a_{p_{l}p{_{l +1} }}^* \\
&\quad - \sum_{2 \leq r <n+1}  \sum_{\mathbf q;  r \geq  q_{l(\mathbf q)}}
	 \sum_{ \mathbf p; r-1 \geq  p_{l(\mathbf p)} ,p_2=2} 
   	\kappa\cdot ({\boldsymbol \lambda} +D)\left( \prod_{j=1}^{l(\mathbf q)} a_{q_{j}q{_{j +1} }}^*\right)
    	\prod_{l=2}^{l(\mathbf p)} a_{p_{l}p{_{l +1} }}^* \\ 
&=-\sum_{1 \leq r <n }  \sum_{\mathbf q; r \geq  q_{l(\mathbf q)}} 
	\sum_{ \mathbf p ; r+1\geq p_{l(\mathbf p)} ,p_2=2} 
   	\kappa\cdot ({\boldsymbol \lambda} +D)\left( \prod_{j=1}^{l(\mathbf q)} a_{q_{j}q{_{j +1} }}^*\right)
    	\prod_{l=2}^{l(\mathbf p)} a_{p_{l}p{_{l +1} }}^* \\
&\quad +2 \sum_{1 \leq r <n+1}   \sum_{ \mathbf q; r \geq  q_{l(\mathbf q)} } 
	\sum_{ \mathbf p ; r \geq  p_{l(\mathbf p)} ,p_2=2} 
   	\kappa\cdot ({\boldsymbol \lambda} +D)\left( \prod_{j=1}^{l(\mathbf q)} a_{q_{j}q{_{j +1} }}^*\right)  
	\prod_{l=2}^{l(\mathbf p)} a_{p_{l}p{_{l +1} }}^*\\
&\quad - \sum_{1 \leq r <n}  \sum_{\mathbf q;  r+1 \geq  q_{l(\mathbf q)}}
	 \sum_{ \mathbf p; r  \geq  p_{l(\mathbf p)} ,p_2=2} 
   	\kappa\cdot ({\boldsymbol \lambda} +D)\left( \prod_{j=1}^{l(\mathbf q)} a_{q_{j}q{_{j +1} }}^*\right)
    	\prod_{l=2}^{l(\mathbf p)} a_{p_{l}p{_{l +1} }}^* \\ \
&=- \sum_{1 \leq r <n }  \sum_{\mathbf q; r \geq  q_{l(\mathbf q)}} 
	\sum_{ \mathbf p ; r+1=p_{l(\mathbf p)} ,p_2=2} 
   	\kappa\cdot ({\boldsymbol \lambda} +D)\left( \prod_{j=1}^{l(\mathbf q)} a_{q_{j}q{_{j +1} }}^*\right)
    	\prod_{l=2}^{l(\mathbf p)} a_{p_{l}p{_{l +1} }}^* \\
&\quad -     \sum_{1 \leq r <n } \sum_{ \mathbf q; r \geq  q_{l(\mathbf q)} } 
	\sum_{ \mathbf p ; r \geq  p_{l(\mathbf p)} ,p_2=2} 
   	\kappa\cdot ({\boldsymbol \lambda} +D) \left( \prod_{j=1}^{l(\mathbf q)} a_{q_{j}q{_{j +1} }}^*\right)
    	\prod_{l=2}^{l(\mathbf p)} a_{p_{l}p{_{l +1} }}^* \\
&\quad +2 \sum_{1 \leq r <n }   \sum_{ \mathbf q; r \geq  q_{l(\mathbf q)} } 
	\sum_{ \mathbf p ; r \geq  p_{l(\mathbf p)} ,p_2=2} 
   	\kappa\cdot ({\boldsymbol \lambda} +D)\left( \prod_{j=1}^{l(\mathbf q)} a_{q_{j}q{_{j +1} }}^*\right)  
	\prod_{l=2}^{l(\mathbf p)} a_{p_{l}p{_{l +1} }}^*\\
&\quad +2    \sum_{ \mathbf q; n \geq  q_{l(\mathbf q)} } 
	\sum_{ \mathbf p ; n \geq  p_{l(\mathbf p)} ,p_2=2} 
   	\kappa\cdot ({\boldsymbol \lambda} +D)\left( \prod_{j=1}^{l(\mathbf q)} a_{q_{j}q{_{j +1} }}^*\right)  
	\prod_{l=2}^{l(\mathbf p)} a_{p_{l}p{_{l +1} }}^*\\
&\quad - \sum_{1 \leq r <n}  \sum_{\mathbf q;  r+1 =  q_{l(\mathbf q)}}
	 \sum_{ \mathbf p; r  \geq  p_{l(\mathbf p)} ,p_2=2} 
   	\kappa\cdot ({\boldsymbol \lambda} +D)\left( \prod_{j=1}^{l(\mathbf q)} a_{q_{j}q{_{j +1} }}^*\right)
    	\prod_{l=2}^{l(\mathbf p)} a_{p_{l}p{_{l +1} }}^* \\ 
&\quad - \sum_{1 \leq r <n}  \sum_{\mathbf q;  r  \geq  q_{l(\mathbf q)}}
	 \sum_{ \mathbf p; r  \geq  p_{l(\mathbf p)} ,p_2=2} 
   	\kappa\cdot ({\boldsymbol \lambda} +D)\left( \prod_{j=1}^{l(\mathbf q)} a_{q_{j}q{_{j +1} }}^*\right)
    	\prod_{l=2}^{l(\mathbf p)} a_{p_{l}p{_{l +1} }}^*  
\end{align*}
\begin{align*}
&=- \sum_{1 \leq r <n }  \sum_{\mathbf q; r \geq  q_{l(\mathbf q)}} 
	\sum_{ \mathbf p ; r+1=p_{l(\mathbf p)} ,p_2=2} 
   	\kappa\cdot ({\boldsymbol \lambda} +D)\left( \prod_{j=1}^{l(\mathbf q)} a_{q_{j}q{_{j +1} }}^*\right)
    	\prod_{l=2}^{l(\mathbf p)} a_{p_{l}p{_{l +1} }}^* \\
&\quad +2    \sum_{ \mathbf q; n \geq  q_{l(\mathbf q)} } 
	\sum_{ \mathbf p ; n \geq  p_{l(\mathbf p)} ,p_2=2} 
   	\kappa\cdot ({\boldsymbol \lambda} +D)\left( \prod_{j=1}^{l(\mathbf q)} a_{q_{j}q{_{j +1} }}^*\right)  
	\prod_{l=2}^{l(\mathbf p)} a_{p_{l}p{_{l +1} }}^*\\
&\quad - \sum_{1 \leq r <n}  \sum_{\mathbf q;  r+1 =  q_{l(\mathbf q)}}
	 \sum_{ \mathbf p; r  \geq  p_{l(\mathbf p)} ,p_2=2} 
   	\kappa\cdot ({\boldsymbol \lambda} +D)\left( \prod_{j=1}^{l(\mathbf q)} a_{q_{j}q{_{j +1} }}^*\right)
    	\prod_{l=2}^{l(\mathbf p)} a_{p_{l}p{_{l +1} }}^*.
\end{align*}

Now 
\begin{align*}
[A_{\boldsymbol \lambda} B_{01}]&=-\sum_{1\leq r<j  \leq n+1}
	 \sum_{\stack{\mathbf q}{j= q_{l(\mathbf q)};r\geq q_{l(\mathbf q)-1}}}
 	\sum_{s=1}^n\sum_{\stack{\mathbf p}{ s \geq  p_{l(\mathbf p)} ,p_2=2}  }
 	\prod_{l=1}^{l(\mathbf q)-1} a_{q_l q_{l+1}}^* 
   	[a_{rj} {_{\boldsymbol \lambda}} \prod_{j=2}^{{l(\mathbf p)}} a_{p_{j}p{_{j +1} }}^*] a_{r,n+1}^* \Phi(b_s)     \\
&=-\sum_{s=1}^n\sum_{2\leq r<j  \leq n+1}
	 \sum_{\stack{\mathbf q}{j= q_{l(\mathbf q)};r\geq q_{l(\mathbf q)-1}}}
 	\sum_{\stack{\mathbf p}{s \geq  p_{l(\mathbf p)} ,p_2=2}  }
 	\prod_{l=1}^{l(\mathbf q)-1} a_{q_l q_{l+1}}^* 
   	[a_{rj} {_{\boldsymbol \lambda}} \prod_{j=2}^{{l(\mathbf p)}} a_{p_{j}p{_{j +1} }}^*] a_{r,n+1}^* \Phi(b_s)     \\
 [B{_{\boldsymbol \lambda}} A_{01}]   &=-\sum_{1\leq r<j  \leq n +1} 
 	\sum_{\stack{\mathbf q}{j= q_{l(\mathbf q)};r\geq q_{{l(\mathbf q)}-1} ,q_2=2} 
 	\sum_{s=1}^{n}\sum_{ \mathbf p; s \geq  p_{l(\mathbf p)}}  }
	\prod_{l=2}^{{l(\mathbf q)}-1} a_{q_l q_{l+1}}^* 
   	[ \prod_{j=1}^{{l(\mathbf p)}} a_{p_{j}p{_{j +1} }}^*{_{\boldsymbol \lambda}}a_{rj} ] a_{r,n+1}^* \Phi(b_s)  \\ 
&= -\sum_{s=1}^{n}\sum_{2\leq r<j  \leq n +1} 
 	\sum_{\stack{\mathbf q}{j= q_{l(\mathbf q)};r\geq q_{{l(\mathbf q)}-1} ,q_2=2} }
 	\sum_{ \stack{\mathbf p}{ s \geq  p_{l(\mathbf p)}}  }
	\prod_{l=2}^{{l(\mathbf q)}-1} a_{q_l q_{l+1}}^* 
   	[ \prod_{j=1}^{{l(\mathbf p)}} a_{p_{j}p{_{j +1} }}^*{_{\boldsymbol \lambda}}a_{rj} ] a_{r,n+1}^* \Phi(b_s) .
 \end{align*}
Hence $[A_{\boldsymbol \lambda} B_{01}]+[B_{01}{_{\boldsymbol \lambda}}A]=0$ follows from the calculation below each fixed $1\leq s\leq n$, and $r\geq 2$
\begin{align*}
&- \sum_{r<j  \leq n+1}
	 \sum_{\mathbf q;j= q_{l(\mathbf q)};r\geq q_{l(\mathbf q)-1}}
 	\sum_{\mathbf p; s \geq  p_{l(\mathbf p)} ,p_2=2}  
 	\prod_{l=1}^{l(\mathbf q)-1} a_{q_l q_{l+1}}^* 
   	[a_{rj} {_{\boldsymbol \lambda}} \prod_{j=2}^{{l(\mathbf p)}} a_{p_{j}p{_{j +1} }}^*]    \\
& - \sum_{r<j  \leq n +1} 
 	\sum_{\mathbf q;j= q_{l(\mathbf q)};r\geq q_{{l(\mathbf q)}-1} ,q_2=2} 
 	\sum_{ \mathbf p; s \geq  p_{l(\mathbf p)}}  
	\prod_{l=2}^{{l(\mathbf q)}-1} a_{q_l q_{l+1}}^* 
   	[ \prod_{j=1}^{{l(\mathbf p)}} a_{p_{j}p{_{j +1} }}^*{_{\boldsymbol \lambda}}a_{rj} ]   \\
&=- \sum_{r<j  \leq n+1}
	 \sum_{\mathbf q;j= q_{l(\mathbf q)};r\geq q_{l(\mathbf q)-1}}
 	\sum_{\mathbf p; s \geq  p_{l(\mathbf p)} ,p_2=2\atop \exists t; (r,j)=(p_t,p_{t+1})}  
 	\prod_{l=1}^{l(\mathbf q)-1} a_{q_l q_{l+1}}^* 
   	  \prod_{j=2,j\neq t}^{{l(\mathbf p)}} a_{p_{j}p{_{j +1} }}^*    \\
& + \sum_{r<j  \leq n +1} 
 	\sum_{\mathbf q;j= q_{l(\mathbf q)};r\geq q_{{l(\mathbf q)}-1} ,q_2=2} 
 	\sum_{ \mathbf p; s \geq  p_{l(\mathbf p)}\atop \exists t; (r,j)=(p_t,p_{t+1})}  
	\prod_{l=2}^{{l(\mathbf q)}-1} a_{q_l q_{l+1}}^* 
   	  \prod_{j=1,j\neq t}^{{l(\mathbf p)}} a_{p_{j}p{_{j +1} }}^*  \\
&=- \sum_{r<j  \leq n+1}
 	\sum_{\mathbf p; s \geq  p_{l(\mathbf p)} ,p_2=2\atop \exists t; (r,j)=(p_t,p_{t+1})}  
	 \sum_{\mathbf q;j=p{_{t+1} }= q_{l(\mathbf q)};r=p{_{t} }\geq q_{l(\mathbf q)-1}}
 	\prod_{l=1}^{l(\mathbf q)-1} a_{q_l q_{l+1}}^* 
   	  \prod_{j=2,j\neq t}^{{l(\mathbf p)}} a_{p_{j}p{_{j +1} }}^*    \\
& + \sum_{r<j  \leq n +1} 
 	\sum_{ \mathbf p; s \geq  p_{l(\mathbf p)}\atop \exists \tau; (r,j)=(p_ \tau,p_{\tau +1})}  
	\sum_{\mathbf q;p_{\tau+1}=j= q_{l(\mathbf q)};p_{\tau}=r\geq q_{{l(\mathbf q)}-1} ,q_2=2} 
	\prod_{l=2}^{{l(\mathbf q)}-1} a_{q_l q_{l+1}}^* 
   	  \prod_{j=1,j\neq \tau}^{{l(\mathbf p)}} a_{p_{j}p{_{j +1} }}^*.
 \end{align*}
The factors appearing in the first summation is over all partitions $\mathbf p$ and $\mathbf q$ and looks like
$$
a_{1q_2}^*\cdots  a_{q_{l(\mathbf q)-2}q_{l(\mathbf q)-1}}^*a_{q_{l(\mathbf q)-1}j}^*a_{2p_3}^*\cdots a_{p_{t-2},p_{t-1}}^*a_{p_{t-1}r}^*a_{jp_{t+2}}^*\cdots a_{l(\mathbf p),n+1}^*
$$
with $l(\mathbf q)-1\leq r$ and $p_{t-1}\leq r$ whereas the second summation is over all partitions $\mathbf p$ and $\mathbf q$ resulting in
$$
a_{2q_3}^*\cdots a_{q_{l(\mathbf q)-2}q_{l(\mathbf q)-1}}^*a_{q_{l(\mathbf q)-1}j}^*a_{1p_2}^*\cdots a_{p_{\tau-2},p_{\tau-1}}^*a_{p_{\tau-1}r}^*a_{jp_{\tau+2}}^*\cdots a_{l(\mathbf p),n+1}^*
$$
with  $l(\mathbf q)-1\leq r$ and $p_{\tau-1}\leq r$.  Renaming $(1,q_1,\dots, q_{l(\mathbf q)-2})$ to 
$(1,p_1',\dots, p_{\tau-1}')$, and  $(p_3,\dots, p_{t-1})$ to 
$(q_3',\dots, q_{l(\mathbf q')-1}')$ , we see that these two summations cancel.

Next we calculate
\begin{align*}
[A &{}_{\boldsymbol \lambda} C_{01}] \\
 =&-\sum_{1\leq r<j  \leq n+1} [a_{rj} 
		 \sum_{\mathbf q;j= q_{l(\mathbf q)};r\geq q_{l(\mathbf q)-1}} 
		\prod_{l=1}^{l(\mathbf q)-1} a_{q_l q_{l+1}}^* a_{r,n+1}^*{_{\boldsymbol \lambda}}
		\sum_{1\leq s <n+1}\sum_{\mathbf p;  s=p_{l(\mathbf p)}, p_2=2} 
		 \prod_{k=2}^{l(\mathbf p) -1}a_{p_{k}p{_{k +1} }}^* \kappa \cdot Da_{s,n+1}^*] \\
 =&-\sum_{1\leq  r<j\leq s <n+1} \sum_{\mathbf q;j= q_{l(\mathbf q)};\atop r\geq q_{l(\mathbf q)-1}}
		\sum_{ \mathbf p,s=p_{l(\mathbf p)},p_2=2
		\atop \exists t: (p_t,p_{t+1})=(r,j)}   
		\left(\prod_{l=1 }^{l(\mathbf q)-1} a_{q_l q_{l+1}}^*\right)a_{r,n+1}^* \kappa \cdot Da_{s,n+1}^*
   		\prod_{k=2,k\neq t}^{l(\mathbf p)-1}a_{p_{k}p{_{k +1} }}^* \\
  -&\sum_{1\leq  s <n+1} \sum_{\mathbf q;n+1= q_{l(\mathbf q)};s\geq q_{l(\mathbf q)-1}}
		\sum_{ \mathbf p,s=p_{l(\mathbf p)},p_2=2}   
		\kappa\cdot ({\boldsymbol \lambda}+D)\left(\left(\prod_{l=1 }^{l(\mathbf q)-1} 
		a_{q_l q_{l+1}}^*\right) a_{s,n+1}^*\right)
  		 \prod_{k=2}^{l(\mathbf p)-1}a_{p_{k}p{_{k +1} }}^* .
\end{align*}
While on the other hand we have
\begin{align*}
[C&{}_{\boldsymbol \lambda} A_{01}]\\
=&-[\sum_{1\leq s <n+1}\sum_{ \stack{\mathbf p}{s=p_{l(\mathbf p)}}}
  		 \prod_{k=1}^{l(\mathbf p) -1}a_{p_{k}p{_{k +1} }}^* \kappa \cdot Da_{s,n+1}^*{_{\boldsymbol \lambda}}
   		\sum_{1\leq r<j  \leq n+1} a_{rj}  
		\sum_{\stack{\mathbf q,j= q_{l(\mathbf q)}}{;r\geq q_{l(\mathbf q)-1},q_2=2} }
		\prod_{l=2}^{l(\mathbf q)-1} a_{q_l q_{l+1}}^* a_{r,n+1}^*] \\
 =&\sum_{1\leq  r<j\leq s <n+1} \sum_{\mathbf q;j= q_{l(\mathbf q)};
		\atop r\geq q_{l(\mathbf q)-1},q_2=2}
		\sum_{ \mathbf p,s=p_{l(\mathbf p)}
		\atop \exists t: (p_t,p_{t+1})=(r,j)}   
		\left(\prod_{l=2}^{l(\mathbf q)-1} a_{q_l q_{l+1}}^*\right)a_{r,n+1}^* \kappa \cdot Da_{s,n+1}^*
   		\prod_{k=1,l\neq t }^{l(\mathbf p)-1}a_{p_{k}p{_{k +1} }}^* \\
&-\sum_{1\leq s <n+1} 
		\sum_{\mathbf q;n+1= q_{l(\mathbf q)};
		\atop s\geq q_{l(\mathbf q)-1},q_2=2}
		\sum_{ \mathbf p,s=p_{l(\mathbf p)}}   \prod_{l=2 }^{l(\mathbf q)-1} 
		a_{q_l q_{l+1}}^* \ a_{s,n+1}^*
  		 \kappa \cdot({\boldsymbol \lambda}+ D) \prod_{k=1}^{l(\mathbf p)-1}a_{p_{k}p{_{k +1} }}^* 
\end{align*}
Thus
\begin{align*}
[A&{}_{\boldsymbol \lambda} C_{01}]+[C_{\boldsymbol \lambda} A_{01}]\\
=&-\sum_{1\leq  r<j\leq s <n+1} \sum_{\mathbf q;j= q_{l(\mathbf q)};\atop r\geq q_{l(\mathbf q)-1}}
		\sum_{ \mathbf p,s=p_{l(\mathbf p)},p_2=2
		\atop \exists t: (p_t,p_{t+1})=(r,j)}   
		\left(\prod_{l=1 }^{l(\mathbf q)-1} a_{q_l q_{l+1}}^*\right)a_{r,n+1}^* 
   		\prod_{k=2,k\neq t}^{l(\mathbf p)-1}a_{p_{k}p{_{k +1} }}^* \kappa \cdot Da_{s,n+1}^* \\
 & -\sum_{1\leq  s <n+1} \sum_{\mathbf q;n+1= q_{l(\mathbf q)};s\geq q_{l(\mathbf q)-1}}
		\sum_{ \mathbf p,s=p_{l(\mathbf p)},p_2=2}   
		\kappa\cdot ({\boldsymbol \lambda}+D)\left(\left(\prod_{l=1 }^{l(\mathbf q)-1} 
		a_{q_l q_{l+1}}^*\right) a_{s,n+1}^*\right)
  		 \prod_{k=2}^{l(\mathbf p)-1}a_{p_{k}p{_{k +1} }}^* \\
& +\sum_{1\leq  r<j\leq s <n+1} \sum_{\mathbf q;j= q_{l(\mathbf q)};
		\atop r\geq q_{l(\mathbf q)-1},q_2=2}
		\sum_{ \mathbf p,s=p_{l(\mathbf p)}
		\atop \exists t: (p_t,p_{t+1})=(r,j)}   
		\left(\prod_{l=2}^{l(\mathbf q)-1} a_{q_l q_{l+1}}^*\right)a_{r,n+1}^*
   		\prod_{k=1,l\neq t }^{l(\mathbf p)-1}a_{p_{k}p{_{k +1} }}^*  \kappa \cdot Da_{s,n+1}^*\\
&  -\sum_{1\leq s <n+1} 
		\sum_{\mathbf q;n+1= q_{l(\mathbf q)};
		\atop s\geq q_{l(\mathbf q)-1},q_2=2}
		\sum_{ \mathbf p,s=p_{l(\mathbf p)}}   \prod_{l=2 }^{l(\mathbf q)-1} 
		a_{q_l q_{l+1}}^* \ a_{s,n+1}^*
  		 \kappa \cdot({\boldsymbol \lambda}+ D) \prod_{k=1}^{l(\mathbf p)-1}a_{p_{k}p{_{k +1} }}^* 
\end{align*}
Then 
\begin{align*}
[B&{}_{\boldsymbol \lambda} B_{01}]+[A_{\boldsymbol \lambda} C_{01}]+[C_{\boldsymbol \lambda} A_{01}] \\
 =&- \sum_{1 \leq r <n }  \left(\sum_{\mathbf q; r \geq  q_{l(\mathbf q)}} 
	\kappa\cdot ({\boldsymbol \lambda} +D)\left( \prod_{j=1}^{l(\mathbf q)} a_{q_{j}q{_{j +1} }}^*\right)\right)
    	\left(\sum_{ \mathbf p ; r+1=p_{l(\mathbf p)} ,p_2=2} \prod_{l=2}^{l(\mathbf p)} a_{p_{l}p{_{l +1} }}^* 
	\right)\\
&  +2   \left( \sum_{ \mathbf q; n \geq  q_{l(\mathbf q)} } 
   	\kappa\cdot ({\boldsymbol \lambda} +D)\left( \prod_{j=1}^{l(\mathbf q)} a_{q_{j}q{_{j +1} }}^*\right)  \right)
	\left(\sum_{ \mathbf p ; n \geq  p_{l(\mathbf p)} ,p_2=2} 
	\prod_{l=2}^{l(\mathbf p)} a_{p_{l}p{_{l +1} }}^*\right) \\
&\ - \sum_{1 \leq r <n} \left( \sum_{\mathbf q;  r+1 =  q_{l(\mathbf q)}}
   	\kappa\cdot ({\boldsymbol \lambda} +D)\left( \prod_{j=1}^{l(\mathbf q)} a_{q_{j}q{_{j +1} }}^*\right)\right)
 	 \left(\sum_{ \mathbf p; r  \geq  p_{l(\mathbf p)} ,p_2=2} 
   	\prod_{l=2}^{l(\mathbf p)} a_{p_{l}p{_{l +1} }}^*\right) \\ \\
& -\sum_{1\leq  r<j\leq s <n+1} \sum_{\mathbf q;j= q_{l(\mathbf q)};\atop r\geq q_{l(\mathbf q)-1}}
		\sum_{ \mathbf p,s=p_{l(\mathbf p)},p_2=2
		\atop \exists t: (p_t,p_{t+1})=(r,j)}   
		\left(\prod_{l=1 }^{l(\mathbf q)-1} a_{q_l q_{l+1}}^*\right)a_{r,n+1}^* 
   		\prod_{k=2,k\neq t}^{l(\mathbf p)-1}a_{p_{k}p{_{k +1} }}^* \kappa \cdot Da_{s,n+1}^* \\
&  -\sum_{1\leq  s <n+1} \sum_{\mathbf q;n+1= q_{l(\mathbf q)};s\geq q_{l(\mathbf q)-1}}
		\sum_{ \mathbf p,s=p_{l(\mathbf p)},p_2=2}   
		\kappa\cdot ({\boldsymbol \lambda}+D) \left(\prod_{l=1 }^{l(\mathbf q)-1} 
		a_{q_l q_{l+1}}^*\right) a_{s,n+1}^* 
  		 \prod_{k=2}^{l(\mathbf p)-1}a_{p_{k}p{_{k +1} }}^* \\
& -\sum_{1\leq  s <n+1} \sum_{\mathbf q;n+1= q_{l(\mathbf q)};s\geq q_{l(\mathbf q)-1}}
		\sum_{ \mathbf p,s=p_{l(\mathbf p)},p_2=2}   
		  \left(\prod_{l=1 }^{l(\mathbf q)-1} 
		a_{q_l q_{l+1}}^*\right)\kappa\cdot D( a_{s,n+1}^*) 
  		 \prod_{k=2}^{l(\mathbf p)-1}a_{p_{k}p{_{k +1} }}^* \\
& +\sum_{1\leq  r<j\leq s <n+1} \sum_{\mathbf q;j= q_{l(\mathbf q)};
		\atop r\geq q_{l(\mathbf q)-1},q_2=2}
		\sum_{ \mathbf p,s=p_{l(\mathbf p)}
		\atop \exists t: (p_t,p_{t+1})=(r,j)}   
		\left(\prod_{l=2}^{l(\mathbf q)-1} a_{q_l q_{l+1}}^*\right)a_{r,n+1}^*
   		\prod_{k=1,l\neq t }^{l(\mathbf p)-1}a_{p_{k}p{_{k +1} }}^*  \kappa \cdot Da_{s,n+1}^*\\
&  -\sum_{1\leq s <n+1} 
		\sum_{\mathbf q;n+1= q_{l(\mathbf q)};
		\atop s\geq q_{l(\mathbf q)-1},q_2=2}
		\sum_{ \mathbf p,s=p_{l(\mathbf p)}}   \prod_{l=2 }^{l(\mathbf q)-1} 
		a_{q_l q_{l+1}}^* \ a_{s,n+1}^*
  		 \kappa \cdot({\boldsymbol \lambda}+ D) \prod_{k=1}^{l(\mathbf p)-1}a_{p_{k}p{_{k +1} }}^* 
\end{align*}
If we look at the forth summation we have factors with $\mathbf q;j= q_{l(\mathbf q)};r\geq q_{l(\mathbf q)-1},\enspace  \mathbf p,s=p_{l(\mathbf p)},p_2=2;\exists t: (r,j)=(q_t,q_{t+1})$ and are of the form
\begin{align}
&-\left(\prod_{l=1 }^{l(\mathbf q)-1} a_{q_l q_{l+1}}^*\right)a_{r,n+1}^* 
   		\prod_{k=2,k\neq t}^{l(\mathbf p)-1}a_{p_{k}p{_{k +1} }}^* \kappa \cdot Da_{s,n+1}^* \notag \\
&=-a_{1q_2}^*\cdots a_{l(\mathbf q)-1,j}^* a_{r,n+1}^* a_{2,p_3}^*
		\cdots a_{p_{t-1}r}^*a_{jp_{t+2}}^*\cdots 
   		a_{l(\mathbf p)-1s}  \kappa \cdot Da_{s,n+1}^*\notag  \\
&=-(a_{1q_2}^*\cdots a_{l(\mathbf q)-1,j}^*  a_{jp_{t+2}}^*\cdots 
   		a_{l(\mathbf p)-1s}  \kappa \cdot Da_{s,n+1}^*)a_{2,p_3}^*\cdots a_{p_{t-1}r}^*a_{r,n+1}^*  \notag
\end{align}
This is the same as summing over factors with  of the form 
\begin{align}
&- a_{1q_2'}^*\cdots  a_{q_{l(\mathbf q')-1,s}'}^*\kappa \cdot  D\left(a_{s,n+1}^* \right) 
		\prod_{l=2 }^{l(\mathbf p')}a_{p_l' p_{l+1}'}^*  \label{forth}
\end{align}
where $\mathbf q':=(1,q_2,\dots ,l(\mathbf q)-1,j,p_{t+2},\dots ,l(\mathbf p)-1,s,n+1)$ and $\mathbf p':=
(1,2,p_3,\dots, p_{t-1},r,n+1)$.  In the above we have $r<s$ and $l(\mathbf q')=s$ and $l(\mathbf p')<s$.

If we look at the fifth summation we have factors with $\mathbf q;n+1= q_{l(\mathbf q)};s\geq q_{l(\mathbf q)-1},\enspace  \mathbf p,s=p_{l(\mathbf p)},p_2=2$ and are of the form
\begin{align*}
-\kappa\cdot ({\boldsymbol \lambda}+D) &\left(\prod_{l=1 }^{l(\mathbf q)-1} 
		  a_{q_l q_{l+1}}^*\right) a_{s,n+1}^* 
  		 \prod_{k=2}^{l(\mathbf p)-1}a_{p_{k}p{_{k +1} }}^*\\ 
&=-\kappa\cdot ({\boldsymbol \lambda}+D) \left(a_{1q_2}^*\cdots 
		a_{q_{l(\mathbf q)-1} ,n+1}^*\right) a_{s,n+1}^* a_{2p_3}^*\cdots 
		a_{l(\mathbf p)-1,s}^*  \\
&=-\kappa\cdot ({\boldsymbol \lambda}+D) \left(a_{1q_2}^*\cdots 
		a_{q_{l(\mathbf q)-1} ,n+1}^*\right)  a_{2p_3}^*\cdots 
		a_{l(\mathbf p)-1,s}^*a_{s,n+1}^*  \\ 
&=-\kappa\cdot ({\boldsymbol \lambda}+D) \left(a_{1q_2}^*\cdots 
		a_{q_{l(\mathbf q)-1} ,n+1}^*\right)  a_{2p_3}^*\cdots 
		a_{l(\mathbf p),n+1}^*  \\
&=-\kappa\cdot ({\boldsymbol \lambda}+D) \left(\prod_{l=1}^{l(\mathbf q')}a_{q_l' q_{l+1}'}^*\right) 
		\prod_{k=2}^{l(\mathbf p)}a_{p_{k}p{_{k +1} }}^*
\end{align*}
where $\mathbf q':=(1,q_2,\cdots ,q_{l(\mathbf q)-1},n+1)$ (so $q'_{l(\mathbf q')+1}=n+1$ and $q'_{l(\mathbf q')}=q_{l(\mathbf q)-1}$) and $l(\mathbf q')\leq s=l(\mathbf p)$.

If we look at the sixth summation we have factors with $\mathbf q;n+1= q_{l(\mathbf q)};s\geq q_{l(\mathbf q)-1}, 
\mathbf p,s=p_{l(\mathbf p)},p_2=2$ and are of the form
\begin{align}
- \left(\prod_{l=1 }^{l(\mathbf q)-1} 
		a_{q_l q_{l+1}}^*\right)\kappa\cdot D( a_{s,n+1}^*) 
  		 \prod_{k=2}^{l(\mathbf p)-1}a_{p_{k}p{_{k +1} }}^*\label{sixth} &=-a_{1q_2}^*\cdots a_{l(\mathbf q)-1,n+1}^* \kappa\cdot D( a_{s,n+1}^*) 
  		 a_{2p_3}^*\cdots a_{l(\mathbf p)-1,s}^*   \\
 =& -\left(\prod_{l=1 }^{l(\mathbf q')} 
		a_{q_l' q_{l+1}'}^*\right)
  		 a_{2p_3}^*\cdots a_{l(\mathbf p)-1,s}^* \kappa\cdot D( a_{s,n+1}^*) \notag
\end{align}
where $\mathbf q':=(1,q_2,\cdots ,q_{l(\mathbf q)-1},n+1)$ (so $q'_{l(\mathbf q')+1}=n+1$ and $q'_{l(\mathbf q')}=q_{l(\mathbf q)-1}$) and $l(\mathbf q')\leq s$.

In the second to the last summation we have after we switch the $\mathbf q$ with the $\mathbf p$ are summed over $\mathbf p;j= p_{l(\mathbf p)};r\geq p_{l(\mathbf p)-1},p_2=2, \mathbf q,s=q_{l(\mathbf q)}\exists t: (q_t,q_{t+1})=(r,j)$ and  the factors have the form 
\begin{align}
&\left(\prod_{l=2}^{l(\mathbf p)-1} a_{p_l p_{l+1}}^*\right)a_{r,n+1}^*
   		\prod_{k=1,l\neq t }^{l(\mathbf q)-1}a_{q_{k}q{_{k +1} }}^*  \kappa \cdot D(a_{s,n+1}^*) \\
&=a_{1q_2}^*\cdots a_{q_{t-1}r}^*a_{j,q_{t+2}}^*\cdots a_{q_{l(\mathbf q)-1,s}}^* 
		\kappa \cdot D(a_{s,n+1}^*)
		a_{2p_3}^*\cdots 
		a_{l(\mathbf p)-1,j}^*a_{r,n+1}^* \notag \\
&=(a_{1q_2}^*\cdots a_{q_{t-1}r}^*a_{r,n+1}^* )(a_{2p_3}^*\cdots 
		a_{l(\mathbf p)-1,j}^*a_{j,q_{t+2}}^*\cdots 
		a_{q_{l(\mathbf q)-1,s}}^* \kappa \cdot D(a_{s,n+1}^*))\notag
\end{align}
Thus this second to last summation is cancelled by terms of the sixth summation \eqnref{sixth} leaving us with a summation over $\mathbf q;n+1= q_{l(\mathbf q)};s= q_{l(\mathbf q)-1}, 
\mathbf p,s=p_{l(\mathbf p)},p_2=2$  in the sixth summation.   Thus the remaining summation has factors that can be rewritten in the form
\begin{align} \label{sixthprime}
& -
  		 a_{2p_3}^*\cdots a_{l(\mathbf p)-1,s}^* \kappa\cdot D( a_{s,n+1}^*)\prod_{l=1 }^{l(\mathbf q')} 
		a_{q_l' q_{l+1}'}^* = \kappa\cdot D( a_{s,n+1}^*)
  		 a_{2p_3}^*\cdots a_{l(\mathbf p)-1,s}^* a_{s,n+1}^* \prod_{l=1 }^{l(\mathbf q')-1} 
		a_{q_l' q_{l+1}'}^*.
\end{align}

In the last summation we have after we switch the $\mathbf q$ with the $\mathbf p$ are summed over $\mathbf p;n+1= p_{l(\mathbf p)};s\geq p_{l(\mathbf p)-1},p_2=2, \mathbf q,s=q_{l(\mathbf q)}$ and  the factors have the form 
\begin{align*}
&- \prod_{l=2 }^{l(\mathbf p)-1} 
		a_{p_l p_{l+1}}^* \ a_{s,n+1}^*
  		 \kappa \cdot({\boldsymbol \lambda}+ D) \prod_{k=1}^{l(\mathbf q)-1}a_{q_{k}q{_{k +1} }}^*  =- \kappa \cdot({\boldsymbol \lambda}+ D)\left(a_{1q_2}^*\cdots  a_{q_{l(\mathbf q)-1,s}}^* \right) a_{s,n+1}^*
		\prod_{l=2 }^{l(\mathbf p')}a_{p_l' p_{l+1}'}^*   
\end{align*}
where $\mathbf p':=(1,2,p_3,\cdots ,p_{l(\mathbf p)-1},n+1)$ so that $l(\mathbf p')\leq s$.    Now the last summation combines with \eqnref{forth} and \eqnref{sixthprime} to give us a summation over  $\mathbf p';s\geq p'_{l(\mathbf p')},p_2'=2, \mathbf q',s=q'_{l(\mathbf q')}$  with factors have the form 
\begin{align*}
&- \prod_{l=2 }^{l(\mathbf p)-1} 
		a_{p_l p_{l+1}}^* \ a_{s,n+1}^*
  		 \kappa \cdot({\boldsymbol \lambda}+ D) \prod_{k=1}^{l(\mathbf q)-1}a_{q_{k}q{_{k +1} }}^*  =- \kappa \cdot({\boldsymbol \lambda}+ D)\left(\prod_{l=1 }^{l(\mathbf q')}a_{q_l' q_{l+1}'}^*\right)
		\prod_{l=2 }^{l(\mathbf p')}a_{p_l' p_{l+1}'}^*   \\
\end{align*}

As a consequence we obtain
\begin{align*}
[A_{\boldsymbol \lambda} C_{01}]+[C_{\boldsymbol \lambda} A_{01}]&=
		 -\sum_{1\leq  s <n+1} \sum_{\mathbf q;s\geq q_{l(\mathbf q)}}
		\kappa\cdot ({\boldsymbol \lambda}+D) \left(\prod_{l=1 }^{l(\mathbf q)} 
		a_{q_l q_{l+1}}^*\right)  
  		\left(\sum_{ \mathbf p,s=p_{l(\mathbf p)},p_2=2}   
	 \prod_{k=2}^{l(\mathbf p)}a_{p_{k}p{_{k +1} }}^*\right) \\
&\quad  -\sum_{1\leq  s <n+1} \sum_{\mathbf q;s=q_{l(\mathbf q)}}
		\kappa\cdot ({\boldsymbol \lambda}+D) \left(\prod_{l=1 }^{l(\mathbf q)} 
		a_{q_l q_{l+1}}^*\right)  
  		\left(\sum_{ \mathbf p,s\geq p_{l(\mathbf p)},p_2=2}   
		 \prod_{k=2}^{l(\mathbf p)}a_{p_{k}p{_{k +1} }}^*\right) \\
\end{align*}

Now if we set 
 $$
 \mathcal B(t):=\sum_{\mathbf q; t= q_{l(\mathbf q)}}\prod_{l=1 }^{l(\mathbf q)}  a_{q_l q_{l+1}}^* 
 \qquad \text{ and }\qquad
 \mathcal C(t):= \sum_{ \mathbf p,t=p_{l(\mathbf p)},p_2=2}   
		 \prod_{k=2}^{l(\mathbf p)}a_{p_{k}p{_{k +1} }}^* 
$$
then by induction one can show that 
\begin{align*}
& \sum_{s=1}^{n}\left(\left(\sum_{q=1}^s\mathcal B(q)\right)\mathcal C(s)\right)+\sum_{s=1}^{n}  \left(\mathcal B(s)\left(\sum_{q=1}^s\mathcal C(q)\right)\right)  \\
&=2 \left(\left(\sum_{q=1}^n\mathcal B(q)\right) \left(\sum_{q=1}^n\mathcal C(q)\right)\right) - \sum_{s=1}^{n-1}\left(\left(\sum_{q=1}^s\mathcal B(q)\right)\mathcal C(s+1)\right)-\sum_{s=1}^{n-1}  \left(\mathcal B(s+1)\left(\sum_{q=1}^s\mathcal C(q)\right)\right).
\end{align*}
(The identity above holds for any elements $\mathcal B(q)$ and $\mathcal C(s)$ in an algebra with coefficients in $\mathbb Z$.)
Thus we conclude 
$$
[B_{\boldsymbol \lambda} B_{01}]+
[A_{\boldsymbol \lambda} C_{01}]+[C_{\boldsymbol \lambda} A_{01}]=0.
$$
This completes the proof that 
\begin{equation}
\label{f0f0f1 }
[\rho(F_0{_{\boldsymbol \lambda}}) [\rho(F_0){_{\boldsymbol \mu}}\rho(F_1)]]=0.
\end{equation}

Now the proof that $[\rho(F_0 ){_{\boldsymbol \lambda}}[\rho(F_0){_{\boldsymbol \mu}}\rho(F_n)]]=0$ and $[\rho(F_0){_{\boldsymbol \lambda}}\rho(F_0)]=0$ are proven in a similar manner, where in the end it boils down to the following formal identities:
 \begin{align*}
 \sum_{r=1}^n\left(\sum_{s=1}^{n-1}(-\delta_{r,s-1}+2\delta_{rs}-\delta_{r,s+1})\left(\sum_{t=1}^r\left(\sum_{v=1}^s\mathcal B(t)\mathcal C(v)\right)\right)\right) =\sum_{t=1}^{n-1}B(t)C(t)-B(n)\left(\sum_{t=1}^{n-1}C(t)\right)
 \end{align*}
and

$$
\left(\sum_{r=1}^{n-1}\left(\sum_{q=1}^r\mathcal A(q)\right)\mathcal A(r+1)\right)-
\left(\sum_{q=1}^n\mathcal A(q)\right)^2+\sum_{s=1}^n\left(\left(\sum_{q=1}^s\mathcal A(q)\right)\mathcal A(s)\right)=0.
$$

\color{black}

For the Serrre type relations for the $E_r$, the calculations are the same as those in  \cite[Lemma 3.5]{MR2003g:17034} where
$$
-\gamma b_r(z)-\frac{1}{2}\left(b_{r-1}^+(z)+b_{r+1}^+(z)\right)
$$
 is replaced by $\Phi(b_r)$ and 
$$
-\frac{\gamma}{2}\dot a_{r,r+1}^*(z)
$$
is replace by $\kappa\cdot D a_{r,r+1}^*$. 
We refer the interested reader to that paper for the proof.

\end{proof}

\end{document}